\documentclass[dvipsnames]{amsart}
\usepackage[utf8]{inputenc}
\usepackage{graphicx}
\usepackage{amssymb}
\usepackage{amsmath,tensor}
\usepackage{float}
\usepackage{xcolor}
\usepackage[
style=alphabetic,
maxnames=5
]{biblatex}
\usepackage{mathscinet}
\usepackage{tikz,tikz-cd}

\usepackage{hyperref}
\usepackage{cleveref}
\usepackage{comment}
\usepackage[all]{xy}
\usepackage[dvips]{epsfig}
\usepackage{pinlabel}
\newcommand{\igeps}[2]{\vcenter{\xy (0,0)*{\includegraphics[scale=#1]{fig/#2}} \endxy}}
\newcommand{\igs}[1]{\igeps{.6}{#1}}
\newcommand{\igm}[1]{\igeps{.9}{#1}}
\newcommand{\igb}[1]{\igeps{1.2}{#1}}
\newcommand{\igbb}[1]{\igeps{1.6}{#1}}

\newtheorem{thm}{Theorem}[section] 
\newtheorem{prop}[thm]{Proposition}
\newtheorem{lem}[thm]{Lemma}
\newtheorem{cor}[thm]{Corollary}

\theoremstyle{definition}
\newtheorem{defn}[thm]{Definition}
\newtheorem{notation}[thm]{Notation}

\newtheorem{algorithm}[thm]{Algorithm}
\newtheorem{ex}[thm]{Example}
\newtheorem{exercise}[thm]{Exercise}

\theoremstyle{remark}
\newtheorem{rem}[thm]{Remark}

\newcommand{\uw}{\underline{w}}
\newcommand{\ux}{\underline{x}}

\newcommand{\eb}{\mathbf{e}}

\newcommand{\mg}{\mathfrak{m}}

\def\R{\mathbb{R}}
\def\Z{\mathbb{Z}}
\newcommand{\gdim}{\mathrm{gd}}

\newcommand{\De}{\Delta}

\def\hat{\widehat}
\newcommand{\ot}{\otimes}
\newcommand{\co}{\colon}
\newcommand{\pa}{\partial}
\DeclareMathOperator{\grank}{grk}
\DeclareMathOperator{\Hom}{Hom}
\DeclareMathOperator{\End}{End}

\DeclareMathOperator{\id}{id}
\newcommand{\namedto}[1]{\stackrel{#1}{\longrightarrow}}
\newcommand{\sumset}{\stackrel{\oplus}{\subset}}
\DeclareMathOperator{\Id}{Id}
\newcommand{\im}{\operatorname{Im}}
\newcommand{\inv}{^{-1}}

\newcommand{\dbl}{f_i}
\newcommand{\dbr}{g_i}

\newcommand{\onetensor}{1^\otimes}

\newcommand{\mi}{\underline}
\newcommand{\ma}{\overline}
\newcommand{\expr}{\leftrightharpoons}

\newcommand{\Bim}{\mathrm{Bim}}

\newcommand{\SSBim}{\mathcal{S}\mathbb{S}\textrm{Bim}}

\newcommand{\HC}{\mathbf{H}}
\newcommand{\SHC}{\mathbf{SH}}
\newcommand{\HA}{\mathfrak{H}}

\newcommand{\mt}{\emptyset}
\newcommand{\Sym}{\operatorname{Sym}}
\newcommand{\Frob}{\mathbf{Frob}}

\newcommand{\BSBim}{\mathbb{BS}\textrm{Bim}}
\newcommand{\SBSBim}{\mathcal{S}\mathbb{BS}\textrm{Bim}}
\newcommand{\duality}{\mathcal{D}}

\newcommand{\LL}{\mathbb{LL}}
\newcommand{\SLL}{\mathbb{SLL}}
\newcommand{\ELL}{\mathbb{ELL}}
\newcommand{\DELL}{\duality\mathbb{ELL}}
\newcommand{\LEL}{\rotatebox[origin=c]{180}{$\mathbb{E}$}\mathbb{LL}}
\newcommand{\LLP}{\mathbb{LLP}}
\newcommand{\ELLP}{\mathbb{ELLP}}
\newcommand{\SSLLP}{\mathbb{SSLLP}}
\newcommand{\SSLL}{\mathbb{SSLL}}

\newcommand{\DLL}{\mathbb{DLL}}

\newcommand{\BB}{\mathbb{B}}

\DeclareMathOperator{\rex}{rex}
\DeclareMathOperator{\BS}{BS}

\DeclareMathOperator{\leftdes}{LD}
\DeclareMathOperator{\rightdes}{RD}
\DeclareMathOperator{\leftred}{LR}
\DeclareMathOperator{\rightred}{RR}
\DeclareMathOperator{\rightq}{RQ}
\DeclareMathOperator{\leftq}{LQ}
\DeclareMathOperator{\rightt}{RT}
\DeclareMathOperator{\leftt}{LT}
\DeclareMathOperator{\core}{core}

\newcommand{\evaluation}{\mathcal{F}}
\newcommand{\genha}{b}
\DeclareMathOperator{\habs}{bs}
\DeclareMathOperator{\defect}{def}
\DeclareMathOperator{\poly}{poly}
\DeclareMathOperator{\term}{term}

\newcommand{\lls}{\scr{LL}(p,I_\bullet)}
\newcommand{\llsprime}{\scr{LL}'(p,I_\bullet)}
\DeclareMathOperator{\last}{last}

\newcommand{\leftup}[1]{{}^{#1}\hspace{-0.03cm}}

\newcommand{\restrict}[3]{\tensor*[_{#1}]{{#2}}{_{#3}}}

\newcommand{\sberry}{\textcolor{WildStrawberry}{s}}
\newcommand{\teal}{\textcolor{teal}{t}}
\newcommand{\urial}{\textcolor{Brown}{u}}

\newcommand{\capcw}{{\color{red} \curvearrowright}}

\title{Singular Light Leaves}
\date{\today}

\author[]{Ben Elias}
\address{Department of Mathematics, Fenton Hall, University of Oregon,
Eugene, OR, 97403-1222, USA}
\email{belias@uoregon.edu}

\author[]{Hankyung Ko}
\address{Department of Mathematics, Uppsala University,
Box. 480,
SE-75106, Uppsala, Sweden}
\email{hankyung.ko@math.uu.se}

\author[]{Nicolas Libedinsky}
\address{Departamento de Matemáticas, Facultad de Ciencias, Universidad de Chile}
\email{nlibedinsky@gmail.com}

\author[]{Leonardo Patimo}
\address{Mathematics Institute, Albert-Ludwigs-Universität, Freiburg im Breisgau, Germany}
\email{leonardo.patimo@math.uni-freiburg.de}

\begin{document}
\maketitle

\begin{abstract} 
For any Coxeter system we introduce the concept of \emph{singular light leaves}, answering a question of Williamson raised in 2008. They 
provide a combinatorial basis for Hom spaces between singular Soergel bimodules. 
\end{abstract}


\tableofcontents

\section{Introduction}
 \subsection{Subject of the paper and motivation}

The \emph{Hecke category} $\HC(W)$  is a graded additive monoidal category whose Grothendieck group is isomorphic to the Hecke algebra $H(W)$ of a Coxeter group $W$. When $W$ is a Weyl group, $\HC(W)$ acts monoidally on the regular block of the Bernstein--Gelfand--Gelfand category $\mathcal{O}$, on the category of perverse sheaves on the flag variety, and on related categories in modular representation theory. In this way $\HC(W)$ serves as a link between representation theory and
geometry, and plays a central role in geometric and modular representation theory.  A more sophisticated but no less ubiquitous object, the \emph{singular Hecke (2-)category} $\SHC(W)$ is a graded additive 2-category whose objects are the finite parabolic subgroups of $W$. It acts on singular blocks of category $\mathcal{O}$, and on perverse sheaves on partial flag varieties. In affine type it is the setting for the geometric side of the geometric Satake equivalence.

Though $\SHC$ is defined similarly to $\HC$, and even contains $\HC$ (as the endomorphism category of the trivial parabolic subgroup), it feels quite different in practice. The viewpoint of $\SHC$ is more ``zoomed in'' than the ordinary Hecke category, with objects and morphisms constructed at the atomic rather than the molecular level.

The Hecke category was given an algebraic construction using bimodules over polynomial rings by Soergel \cite{Soe92}, \cite{Soe07}. Then it was given a diagrammatic presentation (via generators and relations) by Elias--Williamson \cite{Soergelcalculus} following earlier work in special cases by Elias, Khovanov, and Libedinsky \cite{Bendihedral}, \cite{KhovanovElias}, \cite{Libgr}. Just as for categorified quantum groups, diagrammatics has paved the way for a number of advances in the field, both computational 
 and abstract. Examples include \cite{EWhodge,Geordiesc,GeordieSimon,LusWil}. It is a natural goal to extend this diagrammatic technology to $\SHC$.

The singular Hecke category was given an algebraic construction by Williamson \cite{SingSb}, and its objects are known as singular Soergel bimodules. The technology to frame a diagrammatic presentation was produced in \cite{ESW}, defining a $2$-category by generators and relations and an evaluation $2$-functor $\evaluation$ to singular Soergel bimodules.
As discussed in \cite{ESW}, the $2$-functor $\evaluation$ is not faithful, and the category from \cite{ESW} is missing many relations.  A complete presentation is still lacking outside of special cases\footnote{For existing presentations in special types, see: dihedral type \cite{Bendihedral}, universal type \cite{ELi17}, affine $A_2$ \cite[Appendix]{ElQS}. For a partial presentation in type $A$ building on unfinished work of Elias--Williamson, see \cite[Section 2.4]{EL}. For further discussion on the status of singular diagrammatics see \cite[Chapter 24]{GBM}. For a presentation of a category that is a quotient of type $A$ singular Soergel bimodules (but in a different language) see \cite{QR}.}. Already there have been numerous applications of this partial diagrammatic presentation, see e.g. \cite{EL, ElQS}.

Let us note that these diagrammatic $2$-categories\footnote{Because we often want to compare $\HC$ and $\SHC$ in the same sentence, we refer to both as $2$-categories, viewing the monoidal category $\HC$ as a $2$-category with one object. Otherwise, the grammar of comparing a category to a $2$-category becomes too difficult. We also do not distinguish in the introduction between $2$-categories, where composition of $1$-morphisms is associative, and bicategories, where composition is only associative up to natural isomorphism.
} technically only form a full sub-$2$-category of the (singular) Hecke category. They encode morphisms between certain bimodules in $\HC$ or $\SHC$ called (singular) Bott--Samelson bimodules. One obtains the entire (singular) Hecke category by taking the idempotent completion.

A crucial tool for studying the Hecke category is the \emph{double leaves basis},  introduced by Libedinsky \cite{LiLi}, \cite{LLL} in the context of algebraic Soergel bimodules. Libedinsky proved that double leaves give a basis for Hom spaces between Bott--Samelson bimodules (when Soergel's Hom formula holds, see Section~\ref{llhom}).  In \cite{Soergelcalculus}, a diagrammatic construction of the double leaves morphisms was given, followed by a convoluted proof that double leaves always span the diagrammatic category. Combining these facts, \cite{Soergelcalculus} deduces that the functor from diagrammatics to bimodules is an equivalence (when Soergel's Hom formula holds).

Let us summarize once more the ingredients used in the Elias--Williamson construction of, and proof of correctness of, the diagrammatic presentation of $\HC$.  \begin{enumerate}
\item Diagrammatics for generating morphisms and some basic relations (e.g. Frobenius relations).
\item An evaluation functor to bimodules.
\item The double leaves basis for morphisms between Bott--Samelson bimodules.
\item A diagrammatic description of double leaves morphisms (they are linearly independent in the diagrammatic category because the evaluation functor sends them to a basis).
\item Additional relations (e.g. Jones--Wenzl relations, Zamolodchikov relations) for the diagrammatic presentation, robust enough to achieve the next point.
\item A diagrammatic proof that double leaves span all diagrams.
\end{enumerate}
As discussed above, the first two ingredients are provided for $\SHC$ in \cite{ESW}. 

The goal of this paper is to diagrammatically produce \emph{singular double leaves}, and prove that (after applying the evaluation functor) they provide a basis for Hom spaces between singular Bott--Samelson bimodules in $\SHC$ (when Williamson's Hom formula holds).\footnote{A diagrammatic description of a bases between indecomposable objects in the special case of Grassmannian has been provided in \cite{PatGrass}, albeit using a different parametrization from the one in the present paper.}  These are the third and fourth ingredients above. We hope to find the last two ingredients in future work.

One of the major motivations for the double leaves basis of $\HC$ is that it is a cellular basis, equipping the Hecke category with the structure of an \emph{(object-adapted) cellular category}, see \cite{ELauda}. For the reader new to cellular bases, one can think of them as bases consisting of morphisms which factor in a nice way through distinguished objects; \S\ref{llhom} will explain what they feel like, if not the precise details. The double leaves basis allows for (relatively) efficient computations of \emph{local intersection forms}. These bilinear forms, which agree with the cellular forms, control decompositions \cite[Cor. 11.75]{GBM}  of objects in both the Hecke category and in the interesting categories on which the Hecke category acts. Computations of local intersection forms allowed Williamson to disprove the expected bounds in Lusztig's conjecture \cite{Geordiesc}. 

The singular double leaves basis of $\SHC$ is not quite a cellular basis, but it is something very close, a \emph{fibered cellular basis} \cite[Def. 2.17]{ELauda}. It is equally useful for computing local intersection forms. Preliminary computations in affine type $A$ suggest that our singular double leaves basis is related to the double ladders basis of \cite{clasp} after applying the geometric Satake equivalence, thus identifying certain local intersection forms in $\SHC$ with those appearing in the representation theory of quantum groups.  A formula for these local intersection forms was conjectured in \cite[Conjecture 1.20]{clasp}, and has been proven in \cite{MartinSpencer}.

 \subsection{Light leaves and the Hom formula}\label{llhom}
 
The construction of the double leaves basis is motivated by Soergel's Hom formula, as we now explain. In the next section, we discuss the singular setting, but for now, we stick to the ordinary Hecke category $\HC$, i.e. Soergel bimodules. 

Suppose $B$ and $B'$ are two Soergel bimodules, which categorify elements $b$ and $b'$ in the Hecke algebra $H$. \emph{Soergel's Hom Formula}  states that the graded rank of $\Hom(B,B')$ as a free left (or right) $R$-module is equal to the standard pairing $(b,b')$ in $H$ (see Definition 3.13 in \cite{GBM}). There is a trick to compute this pairing when $b$ or $b'$ is self-dual under the Kazhdan--Lusztig involution: we can pretend that the standard basis of $H$ is orthonormal, see \cite[Lemma 3.19]{GBM}. Letting $\{h_x\}_{x \in W}$ denote this standard basis, we have
\begin{equation}\label{Homformula} b = \sum c_x h_x, \quad b' = \sum c'_x h_x, \qquad \grank \Hom(B,B') = \sum c_x c'_x. \end{equation}

This suggests the following route for constructing an $R$-basis of $\Hom(B,B')$. For each $x \in W$ we define a $c_x$-dimensional space of maps from $B$ to some special object $X_x$, and a $c'_x$-dimensional space of maps from $X_x$ to $B'$. By composition, we get a subspace of $\Hom(B,B')$, hopefully with dimension $c_x c'_x$. Taking the union over all $x$, we may get a basis for $\Hom(B,B')$. Note that $c_x, c'_x$ are elements of $\mathbb{Z}[v,v^{-1}]$, representing graded dimensions rather than ordinary dimensions.

Let $\HC_{\BS} \subset \HC$ denote the full subcategory of Bott--Samelson objects. Every object in $\HC$ is a direct summand of (a direct sum of grading shifts of) Bott--Samelson objects.  By describing all morphism spaces in $\HC_{\BS}$, we can use the Karoubi envelope construction to describe morphism spaces in $\HC$. So we focus on \eqref{Homformula} as it applies to Bott--Samelson objects $\BS(\uw)$, which are associated to expressions $\uw$. An expression is a sequence of simple reflections, and its \emph{terminus} is the product of those reflections, an element of $W$. If $\uw = (s_1, \ldots, s_d)$ then \[ \BS(\uw) = \BS(s_1) \ot \cdots \ot \BS(s_d).\]

The Deodhar defect formula gives a concrete formulation for how $[\BS(\uw)]$ can be expressed in the standard basis:
\begin{equation} [\BS(\uw)] = \sum_{\eb \subset \uw} v^{\defect(\eb)} h_{\term(\eb)}. \end{equation}
This is a sum over all subexpressions $\eb$ of $\uw$. Each subexpression contributes a power of $v$
times $h_x$, where $x$ is the terminus  of the subexpression. The exponent of $v$ is the defect of the subexpression. See \cite[Section 3.3.4]{GBM} for details.

Thus one expects to construct a basis for morphisms between Bott--Samelson objects by constructing a special morphism $\BS(\uw) \to X_x$ for each subexpression of $\uw$ with terminus $x$, whose degree matches the defect of the subexpression. This morphism is called a \emph{light leaf}. Applying a duality functor to a light leaf $\BS(\uw') \to X_x$, we obtain a map $X_x \to \BS(\uw')$ called an \emph{upside-down light leaf}. By composing light leaves with upside-down light leaves, one obtains the double leaves basis for $\Hom(\BS(\uw), \BS(\uw'))$.

Now the Bruhat order enters the story. For any set of objects in a category, the set of all linear combinations of morphisms factoring through those objects form an ideal. Our objects $\{X_x\}_{x \in W}$ are chosen so that the ideal $\Hom_{< x}$ of morphisms which factor through $\{X_y\}_{y < x}$ agrees with the span of double leaves associated to $y$ for $y < x$. Moreover, $\End(X_x)/\End_{< x}(X_x)$ is spanned by the identity map. This behavior is what makes the category an (object-adapted) cellular category (see \cite[Lemma 2.8]{ELauda}).  Note that $\Hom_{< x}$ also has an independent definition in terms of the support of bimodules and bimodule maps (see \cite[Prop 3.25]{EKLP2}).

The object $X_x$ which is the target of a light leaf can be taken to be $\BS(\ux)$, associated to a reduced expression $\ux$ for $x$. There are many such reduced expressions, though all are related by the braid relations. For each braid relation $\ux \to \ux'$, $\HC$ has a morphism $\BS(\ux) \to \BS(\ux')$, called an \emph{elementary rex move}. An example is the $2m$-valent vertex from \cite{Soergelcalculus}. A \emph{rex move} is a composition of elementary rex moves along a sequence of braid relations. Rex moves descend to isomorphisms modulo the ideal $\Hom_{< x}$. This permits us the flexibility to choose a different reduced expression $\ux$ for each subexpression with terminus $x$. This flexibility is crucial to the algorithm which constructs light leaves. When composing light leaves with upside-down light leaves, the sources and targets may not match, but we use rex moves to go between them. Changing the choices of reduced expressions or the choices of rex moves will change the light leaves and the double leaves basis, but in a well-understood upper-triangular fashion. In this way, the theory of reduced expressions and braid relations enters the story.

 Let us explain the construction of light leaves. Deodhar's definition of the defect analyzes expressions and subexpressions one index at a time, viewing a subexpression as instructions for a stroll around the Bruhat graph. Correspondingly, light leaves are constructed inductively, one index at a time.

Let $\uw=(s_1,\ldots, s_m)$ be an expression in $S$, and $\eb = (e_1, \ldots, e_m) \in \{0,1\}^m$ be a subexpression (or rather, the function which records which elements $s_i$ are included in the subexpression). Let $(x_1, \ldots, x_m)$ be the sequence of elements in $W$ followed by the Bruhat stroll of $\eb$, i.e. $x_i:=s_1^{e_1}\cdots s_i^{e_i}.\ $
  After $k$ steps, we have constructed a map 
\[ \LL_k \co \BS(s_1, \ldots, s_k) \to X_{x_k}. \]
To construct the next step, we only need to produce a \emph{single-step light leaf}, 
a map
\[ \SLL_k \co X_{x_k} \ot \BS(s_{k+1}) \to X_{x_{k+1}} \]
depending on $e_{k+1}$, whose degree matches Deodhar's formula for the contribution of this step to the overall defect. Then we set
\[ \LL_{k+1} = \SLL_k \circ (\LL_k \ot \id_{\BS(s_{k+1})}).\]

The possible single-step light leaves one must construct are parametrized by the triple $(x_k, s_{k+1}, e_{k+1})$ in the large set $W \times S \times \{0,1\}$. An incredible amount of efficiency is gained by the next observation, which reduces to the case when $W$ has type $A_1$. In type $A_1$ there are only four possibilities, called \emph{elementary light leaves}. The observation is trivial but we state it anyway.  For $a,b,c \in W$ write $a.b = c$ if one obtains a reduced expression for $c$ by concatenating reduced expressions for $a$ and $b$.

\begin{lem} \label{lem:reducetoA1} Suppose that $x = ys$ for $x, y \in W$ and $s \in S$. Then there exist $z \in W$ with $z<zs$ in the Bruhat order and $\{u,v\} = \{1,s\}$ such that $x=z.u$ and $y=z.v$. 
\end{lem}

 Now we apply the observation with $x = x_k$ and $s = s_{k+1}$. Let $z$ and $u$ be as in the lemma.
Then set
\[ \SLL_k = \id_{X_z} \ot f,\]
where $f$ is the elementary light leaf associated to the triple  $(u,s_{k+1}, e_{k+1})$ in type $A_1$ (the parabolic subgroup generated by $s_{k+1})$.

Again, this reduction to type $A_1$ is possible because we can choose reduced expressions at will. In particular, we can choose reduced expressions for $x_k$ and $x_{k+1}$ starting with $z$. At the $(k+2)$-nd step this algorithm may necessitate a different reduced expression for $x_{k+1}$, and we go between them using rex moves.


 \subsection{Singular light leaves}

All the moving parts from \ref{llhom} have analogs for the 2-category of singular Soergel bimodules $\SHC$, with the combinatorics of elements of $W$ replaced by the combinatorics of double cosets. However,
there are many additional surprises. We give the construction of singular light leaves in Section \ref{sec:LightLeavesConstruction}.  Singular light leaves are constructed, as in the regular case, by categorifying the singular Deodhar formula given in Corollary \ref{cor:singDeodhar}.

The theory of reduced expressions for double cosets was developed greatly by the first two authors \cite{EKo}. They introduced the \emph{singular braid relations} between reduced expressions, and proved the \emph{singular Matsumoto theorem}, stating that any two reduced expressions for the same double coset are related by braid relations.  In this paper, we construct a morphism between singular Bott--Samelson bimodules for each braid relation, called a \emph{(singular) elementary rex move}. By applying rex moves, we can go between any two Bott-Samelson bimodules for different reduced expressions for the same double coset.

As before, flexibility in the choice of reduced expressions allows one to reduce the construction of single-step light leaves to a more manageable collection of elementary light leaves. In Theorem \ref{thmB} we prove an analogue of Lemma \ref{lem:reducetoA1}, thanks to which it suffices to construct elementary light leaves for pairs of cosets called \emph{Grassmannian pairs} (see Definition \ref{notation:grassmannian}). Grassmannian pairs can be associated to each finite parabolic subgroup (not just those of rank $1$) and there can be many more than four of them. The elementary light leaves themselves are substantially more complicated than in the non-singular setting, e.g. \eqref{ELLA}, but still tractable.

A new phenomenon is that one needs to sprinkle polynomials inside the diagrams to define singular light leaves, accounting for the ``polynomial factors'' which appear in the singular Deodhar formula. In the ordinary setting, the endomorphism ring of $X_x$, modulo lower terms, is the polynomial ring $R$. All such endomorphisms can be realized using the left action of polynomials, whence double leaves form a basis over this left action. In the singular setting, the corresponding endomorphism ring (modulo lower terms) will vary as the double coset varies, and the left action of polynomials is not sufficient to generate the endomorphism ring. Polynomials within the diagram are thus unavoidable. To avoid further confusion, we chose to place polynomials in the middle of our double leaves as well. Our basis is a basis of morphisms as a vector space, not as a module over a polynomial ring.

These polynomial sprinkles have a different flavor than anything before, and bring algebra (joining combinatorics) back into the analysis of light leaves. For this purpose we developed the theory of  Demazure operators for double cosets \cite{EKLP1} and the theory of singular Bruhat order \cite{EKLP2} in previous work.

The main result of this paper is Theorem \ref{mainthm}. It says that an appropriate set of singular double leaves is a basis for each 2-morphism space in $\SHC$. We also prove (Theorem \ref{thm:COB}) that different choices of singular double leaves bases (e.g. for different choices of reduced expressions) are related by an upper-triangular change of basis matrix. Finally, we confirm (Theorem \ref{thm:fibered}) that our basis does equip $\SHC$ with the structure of a fibered cellular 2-category.

\subsection{Outline of the paper}
Part 1 of the paper is primarily a condensed recollection of background material from a variety of sources, but has some new results.

In \Cref{ss.doublecoset} we recall the theory of reduced expressions for double cosets from \cite{EKo}. The notation introduced here will be used throughout the paper. Starting with \Cref{ssec.reductiongrassmannian} we present new results, including the reduction to Grassmannian coset pairs.

In \Cref{HA} we recall the Hecke algebroid from Williamson's thesis \cite{SingSb}. We describe the singular Deodhar formula in \Cref{singdeodhar}, which contains some new material.

In \Cref{sec:SSB} we recall the theory of singular Soergel bimodules from Williamson's thesis \cite{SingSb}. In \Cref{sec:lowerterms} we recall results from \cite{EKLP2} about various ideals of lower terms.

In \Cref{sec:diagrammatics} we recall the diagrammatic calculus for singular Soergel bimodules, introduced in \cite{ESW}, and review the basics of Frobenius extensions.

Part 2 of the paper contains almost entirely new constructions and results, leading up to the construction of double leaves morphisms. 

In \Cref{ch.rexmove}, we define morphisms associated to the braid relations, and prove that certain of these morphisms are isomorphisms, and others are inclusions and projections for a direct summand.

In \Cref{sec:LightLeavesConstruction} we describe the algorithm to construct light leaves and double leaves morphisms, and we give many examples. The algorithm is complex enough that we split it into several stages and substages that introduce different aspects of the complexity. We state our main results in \Cref{ss.conclusion}.

Finally, \Cref{proofs} contains the proofs of all the results stated in \Cref{ch.rexmove} and \Cref{sec:LightLeavesConstruction}. A crucial technique in these proofs is the identification of certain elements within singular Bott-Samelson bimodules on which light leaf morphisms can be nicely evaluated. These elements are indexed by \emph{dual sequences}, a new combinatorial construction associated to paths of double cosets. Dual sequences are used to define a total order on light leaves and a partial order on double leaves. We develop the combinatorics of dual sequences in \Cref{ssec:dualsequences}, and it should have independent interest.
 
 \subsection{Acknowledgements.}

We would like to thank Geordie Williamson for many helpful discussions across the years, and for writing a helpful erratum to his article on singular Soergel bimodules. BE was partially supported by NSF grant DMS-2201387, and appreciates the support given to his research group by NSF grant DMS-2039316.  NL was partially supported by FONDECYT-ANID
grant 1230247.  HK was partially supported by Swedish Research Council.

{\bf New leaves sprouted during production:} Clara and Gael

\part{Background material and Grassmannian reduction}

\section{Double cosets, expressions, and Grassmannian pairs}\label{ss.doublecoset}

The first two sections recall concepts from \cite{EKo}, and we refer to that paper for examples. The rest is largely new.

\subsection{Expressions for double cosets}

We fix a Coxeter system $(W,S)$, and write $\ell$ for the length function. For $I\subset S$, we denote by $W_I$ the subgroup of
$W$ generated by $I$.  When $W_I$ is finite, we say that $I$ is \emph{finitary}, we write $w_I$ for the
longest element of $W_I$, and we set $\ell(I) := \ell(w_I)$.

For $I, J\subset S$, an \emph{$(I,J)$-coset} is an element $p$ in 
 $W_{I}\backslash W/W_{J}$. When we discuss $p$ we really mean the triple $(p,I,J)$. It might happen that  $(p,I,J)\neq (p',I',J')$, even though $p=p'$ as subsets of $W$, and we distinguish between $p$ and $p'$ in this case. If $p$ is an $(I,J)$-coset we denote by $\ma{p}\in W$ and $\mi{p}\in W$ the maximal and minimal elements in the Bruhat order in the set $p$. 
We define a length function on double cosets by the formula
\begin{equation} \ell(p) = 2 \ell(\ma{p}) - \ell(I) - \ell(J). \end{equation}
 
 A (singular) \emph{multistep  expression} is a sequence of finitary subsets of $S$ of the form 
\begin{equation}\label{ex}
 L_{\bullet}=[[ I_0\subset K_1\supset I_1\subset K_2 \supset \cdots \subset K_m\supset I_m]].
 \end{equation}
The \emph{length} of the multistep expression $L_\bullet$ is 
\begin{equation}\label{explength}\ell(L_\bullet):=- \ell(I_0) + 2\ell(K_1)-2\ell(I_1)+2\ell(K_2)-\cdots + 2 \ell(K_m) -\ell(I_{m}).
\end{equation} Equalities are permitted, e.g. $K_1 = I_1$ is allowed. When we write a multistep expression as $[[K_1 \supset I_1 \subset \cdots]]$ we mean that $I_0 = K_1$. Similarly, $[[\cdots \subset K_m]]$ means that $I_m = K_m$. We call $L_{\bullet}$ an \emph{$(I,J)$-expression} if $I_0 = I$ and $I_m = J$.

We say that $L_{\bullet}$ as above is a \emph{reduced expression} for an $(I_0, I_m)$-coset $p$ if 
\begin{equation} \ma{p} = w_{K_1} w_{I_1}^{-1} w_{K_2} w_{I_2}^{-1}\cdots w_{K_m} w_{I_m}^{-1} \ \  \mathrm{and} \  \ \ell(p) = \ell(L_{\bullet}). \end{equation}
We write $L_{\bullet} \expr p$ when $L_{\bullet}$ is a reduced expression for $p$. The notation $\expr$ is also used for non-reduced expressions; we will not have occasion to do so in this paper, nor will we discuss which cosets are expressed by non-reduced expressions.

\begin{notation} Let $p$ be a $(I,K)$ coset, $q$ be a $(I,J)$-coset, and $r$ be a $(J,K)$ coset. We write $p = q.r$ when a reduced expression for $q$ concatenates with a reduced expression for $r$ to yield a reduced expression for $p$. 

Similarly, if $x, y, z \in W$ we write $x = y.z$ if $x = yz$ and $\ell(x) = \ell(y) + \ell(z)$. This is equivalent to the expected statement about concatenations of reduced expressions. \end{notation}

For $L\subset S$ and $s\in S$ we use the notation $Ls$ to denote $L\cup \{s\}$ when $s\notin L$.
A (singular)  \emph{singlestep expression} $I_{\bullet}=[I_0,I_1,\ldots, I_d]$  is a sequence of finitary subsets of $S$ such that, for all $1\leq i\leq d$, either $I_i=I_{i-1}s$  or  $I_i=I_{i-1} \setminus s$ for some $s\in S.$  We may also record a singlestep expression as e.g.
$I_{\bullet} = [I_0 + s + u - t + t \ldots]$, which would indicate that $I_1 = I_0 s$, $I_2 = I_1 u$, $I_3 = I_2 \setminus t$, etcetera.

A singlestep expression is a particular kind of multistep expression. To each singlestep expression, one can also associate a coarser multistep  expression by remembering its local maxima and minima.

Throughout this paper, ordered lists could come in parentheses as in $(s_1, s_2, \ldots)$, or in brackets like $I_{\bullet}$, or in double brackets as in \eqref{ex}. The delimiter indicates the context; e.g., lists in single brackets are associated to a singlestep expression.

\begin{notation} \label{notation:cosetpair} Let $I, J \subset S$ and $s \in S$ with $I$ and $Js$ finitary. Let $q$ be an $(I,Js)$-coset and $p$ be an $(I,J)$-coset, with $p \subset q$. We call $\{p,q\}$ a \emph{coset pair for $(I,J,s)$}. We also refer to the sequences $[p,q]$ and $[q,p]$ as coset pairs, and disambiguate by writing $p \subset q$.
\end{notation}

\begin{defn}\label{defn:subordinatepath} Fixing $I_{\bullet}=[I_0,I_1,\ldots, I_d]$, suppose that $t_k$ is an $(I_0, I_k)$-coset for $0 \le k \le d$. We say that $t_\bullet=[t_0,\ldots,t_d]$ is a \emph{subordinate path} of $I_\bullet$, and we write $t_{\bullet} \subset I_{\bullet}$, if $t_0$ contains the identity element, and $[t_k,t_{k+1}]$ is a coset pair for $(I_0,J,s)$ for all $0 \le k \le d-1$. Here, $J$ (resp. $Js$) is either $I_k$ or $I_{k+1}$, whichever is smaller (resp. bigger).

The \emph{terminus} of the subordinate path is the final $(I_0,I_d)$-coset $\term(t_{\bullet}) = t_d$. Two subordinate paths $t_{\bullet}, u_{\bullet} \subset I_{\bullet}$ are called \emph{coterminal} if they have the same terminus, i.e. $t_d = u_d$.\end{defn}

Paths subordinate to an expression $I_{\bullet}$ are the coset analogue of subexpressions, and the sequence $t_{\bullet}$ is analogous to the Bruhat stroll of \cite[Section 2.4]{Soergelcalculus}.

\begin{defn} \label{def.bruhat}
Let $I,J\subset S$ be finitary and $p \in  W_I\backslash W/W_J$ and let $I_{\bullet}$ be an $(I,J)$-expression. We write $p \le I_{\bullet}$ if there is some path $t_{\bullet}$ subordinate to $I_{\bullet}$ with $p = \term(t_{\bullet})$. If $I_{\bullet}$ is an expression for $q$ (not necessarily reduced) and $p \le I_{\bullet}$ then we write $p \le q$. In \cite[Theorem 2.16]{EKLP2} it is proven that this definition is independent of the choice of expression for $q$, and defines a partial order on double cosets called the \emph{Bruhat order}. Moreover, it agrees with other definitions of the Bruhat order (cf. \cite[\S 2.1]{SingSb}), e.g. $p \le q$ if and only if $\mi{p} \le \mi{q}$. \end{defn}

\subsection{Redundancy}

Let $p$ be a $(I,J)$ coset. Let
\begin{equation} \label{eq:redundancydef} \leftred(p):=I\cap \mi{p}J\mi{p}^{-1}, \qquad \rightred(p):=J \cap \mi{p}^{-1}I\mi{p}=\mi{p}^{-1}\leftred(p)\mi{p}, \end{equation}
and call these the \emph{left} and \emph{right redundancy} of $p$. These redundancy sets and the subgroups they generate are of crucial importance in this paper. When no confusion is possible we will use the abbreviation $\leftred$ for $\leftred(p)$ and $\rightred$ for $\rightred(p)$.

The following gives a  good intuition for what redundancy is, and the proof is trivial. If $s \in I$ and $t \in J$ are simple reflections such that $s\mi{p} = \mi{p} t$, then $s \in \leftred(p)$ and $t \in \rightred(p)$.

\begin{ex} \label{ex:runningstart} In this example, $W$ is the symmetric group $S_{11}$, and $W_J \cong S_8 \times S_3$.
We set $W_I = S_2 \times S_5 \times S_1 \times S_2 \times S_1$, and let $m$ be the double coset with minimal element depicted below on the left.
\begin{equation} \igm{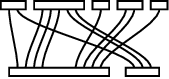} \qquad \igm{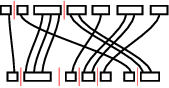} \end{equation}
The boxes represent $W_J$ (on the bottom) and $W_I$ (on the top). The red lines on the right indicate simple reflections that were removed from $I$ or $J$ to reach the redundancy. As can be seen, the redundancy keeps track of which strands are in the same box in both bottom and top. \end{ex}

We frequently use the following consequence of Howlett's theorem, see \cite[Lemma 2.12, (2.8)]{EKo}: for any $(I,J)$-coset $p$ 
we have
\begin{equation} \label{eq:mapmip}  \ma{p} = w_I . \mi{p} . (w_{\rightred}^{-1} w_J) = (w_I w_{\leftred}^{-1}) . \mi{p} . w_J. \end{equation}

\subsection{Reduction to Grassmannian pairs} \label{ssec.reductiongrassmannian}

In this section, we state the analog for double cosets of Lemma \ref{lem:reducetoA1} from the introduction.

\begin{thm}\label{thmB} Let $p \subset q$ be a coset pair for $(I,J,s)$. Let 
$\rightq = \rightred(q)$, and let $n$ be the $(\rightq,Js)$-coset whose underlying set is $W_{Js}$. Then there exists an $(I,\rightq)$-coset $z$ and a $(\rightq,J,s)$ coset pair $m \subset n$ such that 
\begin{equation} z . m = p, \qquad z . n = q. \end{equation}
More precisely, $m, n$ and $z$ are determined by 
\begin{equation} \label{mnzstuff} \mi{q}.\mi{m} =\mi{p}, \qquad \mi{n}=\mathrm{id}, \qquad \mi{z} = \mi{q}. \end{equation}
\end{thm}

We give more properties of $m,n$, and $z$ in Proposition \ref{thmBconti}. First, we discuss what makes the pair $m \subset n$ manageable: it is a Grassmannian pair.

\begin{notation} \label{notation:grassmannian} Let $m \subset n$ be an $(I,J,s)$ coset pair. We call it a \emph{(right) Grassmannian pair} if $I \subset Js$ and $n$ is the $(I,Js)$-coset containing the identity. 
\end{notation} 

To elaborate, for a Grassmannian pair, the underlying set of $n$ is the entire parabolic group $W_{Js}$, and $m \subset W_{Js}$. By \cite[Lemma 5.21]{EKo}, any reduced expression for $m$ or $n$ only involves subsets of $Js$. For all practical purposes, we can assume that $S = Js$ (which is finitary) for the purpose of studying $m \subset n$. In this case, $J \subset Js$ is a maximal parabolic subgroup. When $W_{Js}$ is the symmetric group, any minimal representative for a coset $W_{Js} / W_J$ is a Grassmannian permutation. In particular, $\mi{m}$ is a Grassmannian permutation, which motivated our choice of name.

\begin{ex} \label{ex:runningpq}  This is an example within the symmetric group $S_{21}$, where
\[ W_I = S_3 \times S_7 \times S_1 \times (S_2)^{\times 4} \times S_1 \times S_1, \qquad W_J = S_4 \times S_8 \times S_3 \times (S_1)^{\times 6}. \]
The simple reflection $s = s_{12}$ merges $S_8 \times S_3$ into $S_{11}$. The $(I,J,s)$ coset pair $p \subset q$ is below.
\begin{equation} p = \igm{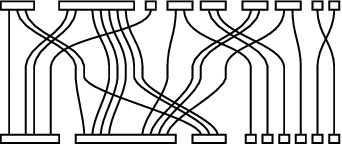} \; \subset \; q = \igm{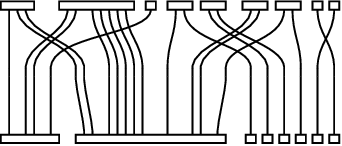} \end{equation}

Below we factor these cosets as $p = z.m$ and $q = z.n$. The parabolic subgroup in the middle is $W_{\rightq}$ for $\rightq = \rightred(q)$.
\begin{equation} \igm{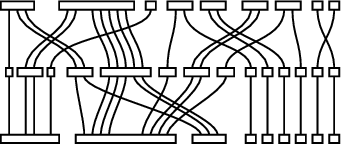} \; \subset \; \igm{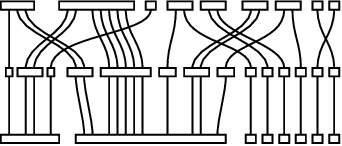} \end{equation}
In both pictures, the coset above the middle is $z$, while the coset below is $m$ on the left, and $n$ on the right. Note that $\mi{n} = \id$. Meanwhile, $\mi{m}$ is not the identity, but it is the identity away from the most interesting $11$ strands.

In \cite[Section 4.10]{EKo}, techniques are given for discussing double cosets in products of Coxeter systems. One can view both $m$ and $n$ as an ``external product'' of double cosets $m_C \subset n_C$ for each connected component $C$ of $Js$. For only one of these connected components is there a difference between $m_C$ and $n_C$, namely the component $C_s$ containing $s$. On this component, the coset $m_{C_s}$ agrees with the coset in Example \ref{ex:runningstart}. It is $m_{C_s}$ which is a Grassmannian permutation in the usual sense of the name.

For other connected components, the cosets $m_C = n_C$ contain the identity but need not be trivial (i.e. length zero). For example, if $C$ is the leftmost connected component generating $S_4$ then $m_C=n_C$ has a reduced expression $[\{s_2\}+s_1+s_3]$.
Thus both $m$ and $n$ have a reduced expression beginning with $[\rightq + s_1 + s_3]$, merging the initial $S_1 \times S_2 \times S_1$ into $S_4$.
\end{ex}


The additional facts given in the next proposition are needed in the construction of the double leaves basis \S\ref{sec:LightLeavesConstruction}.

\begin{prop} \label{thmBconti} Let us continue with the terminology of  Theorem \ref{thmB}. 
The coset $n$ has a reduced expression $[[\rightq \subset Js]]$, and $\ma{n} = w_{Js}$. The coset $z$ satisfies $\ma{z} = w_I . \mi{z}$.
We have
\begin{equation} \label{eq:pqmnredundancies} \leftred(n) = \rightred(n) = \rightred(z) = \rightred(q) = \rightq, \qquad \rightred(m) = \rightred(p). \end{equation}
The left redundancy of $p$ lives inside $I$, while the left redundancy of $m$ lives inside $\rightq$; instead of being equal, these redundancies are conjugate:
\begin{equation} \label{eq:leftredn}\leftred(m) = \mi{q}^{-1} \leftred(p) \mi{q}. \end{equation}
\end{prop}

\subsection{Proof of reduction to Grassmannian pairs}\label{ssec:proofGrass}

Several times in the proof we tacitly use that $x \in W$ is minimal in its $(I,J)$-coset if and only if its left descent set $\leftdes(x)$ intersects $I$ in the empty set, and its right descent set $\rightdes(x)$ intersects $J$ in the empty set (see \cite[Lemma 2.12(5)]{EKo}).

\begin{proof}[Proof of Theorem \ref{thmB} and Proposition \ref{thmBconti}] That $\ma{n} = w_{Js}$ and that $[[\rightq \subset Js]]$ is a reduced expression for $n$ follow immediately from the definitions.  Since $\mi{n}$ is the identity and $\rightq \subset Js$, it is obvious that $\leftred(n) = \rightred(n) = \rightq$.

Let $\leftq = \leftred(q)$. Let $z$ be the $(I,\rightq)$-coset containing $\mi{q}$. Then $\mi{q}$ is minimal in $z$, because its left descent does not intersect $I$, and its right descent does not intersect $Js$, and $\rightq \subset Js$. So $\mi{q} = \mi{z}$. Therefore we have
\begin{equation} \rightred(z) = \rightq \cap \mi{z}^{-1} I \mi{z}= (\rightq \cap Js) \cap \mi{q}^{-1} I \mi{q} = \rightq \cap (Js \cap \mi{q}^{-1} I \mi{q}) = \rightq \cap \rightq = \rightq. \end{equation}
By \eqref{eq:mapmip} we deduce that
\begin{equation} \label{eq:mazIs} \ma{z} = w_I.\mi{z}.(w_{\rightq}^{-1} w_{\rightq}^{\phantom{-1}}) = w_I . \mi{z}. \end{equation}

Let $n$ be the $(\rightq,Js)$-coset containing $\id$, with underlying set $W_{Js}$. Using \eqref{eq:mapmip} and $\mi{z} = \mi{q}$ and \eqref{eq:mazIs} we have
\begin{equation} \ma{q} = w_I . \mi{q} . (w_{\rightq}^{-1} w_{Js}) = (w_I . \mi{z}) . (w_{\rightq}^{-1} w_{Js}) = \ma{z} . (w_{\rightq}^{-1} w_{Js}).\end{equation} 
By \cite[Proposition 4.3] {EKo}
, this implies $q = z.n$.

By \cite[Lemma 2.15]{EKo}, there exists $y \in W_{Js}$ such that 
\begin{equation} \mi{p} = \mi{q} . y. \end{equation}
We claim that $y$ is minimal in its $(\rightq,J)$-coset. 
Note that $\rightdes(y) \subset \rightdes(\mi{p})$, see \cite[Lemmas 2.3 and 2.5]{EKo}. But $\rightdes(\mi{p}) \cap J = \mt$, and thus the same is true for $y$. (In particular, $\rightdes(y) \subset \{s\}$.) Now we show that $\leftdes(y) \cap \rightq = \mt$.
If $t \in \leftdes(y) \cap \rightq$ then 
\begin{equation} \mi{q} t \mi{q}^{-1} \in \leftdes(\mi{q}. y) \cap \mi{q} \rightq \mi{q}^{-1} = \leftdes(\mi{p}) \cap \leftq. \end{equation}
But $\leftdes(\mi{p}) \cap \leftq$ is empty, since $\leftq \subset I$ and $\mi{p}$ is minimal in its $(I,J)$-coset.

Let $m$ be the $(\rightq,J)$-coset containing $y$. Since $y \in W_{Js}$ and $\rightq, J \subset Js$, we have $m \subset W_{Js} = n$. By the previous paragraph, we have $\mi{m} = y$. We claim that $z.m = p$. By \cite[Proposition 4.3] {EKo}, it suffices to prove the equality
\begin{equation} \label{eq:desiredmap} \ma{p} = \ma{z} . (w_{\rightq}^{-1} \ma{m}). \end{equation}
Using \eqref{eq:mazIs} to rewrite $\ma{z}$ and \eqref{eq:mapmip} to rewrite $\ma{p}$ and $w_{\rightq}^{-1} \ma{m}$, \eqref{eq:desiredmap} is equivalent to
\begin{equation} w_I . \mi{p} . (w_{\rightred(p)}^{-1} w_J) = w_I . \mi{q} . \mi{m} . (w_{\rightred(m)}^{-1} w_J). \end{equation}
Since $\mi{p} = \mi{q}.\mi{m}$, it remains to prove that
\begin{equation} w_{\rightred(p)}^{-1} w_J = w_{\rightred(m)}^{-1} w_J. \end{equation}
This is equivalent to $\rightred(m) = \rightred(p)$, which states that
\begin{equation} \label{thisoneisthat} J \cap \mi{m}^{-1} \rightq \mi{m} = J \cap (\mi{m}^{-1} \mi{q}^{-1} I \mi{q} \mi{m}). \end{equation}

Kilmoyer's theorem (see \cite[Lemma 2.14]{EKo}) states that \begin{equation} \rightq = Js \cap \mi{q}^{-1} I \mi{q} = W_{Js} \cap \mi{q}^{-1} I \mi{q}. \end{equation}
Thus 
\begin{equation} \label{thisoneimpliesthat} \mi{m} J \mi{m}^{-1} \cap \rightq \;\; = \;\; \mi{m} J \mi{m}^{-1} \cap W_{Js} \cap \mi{q}^{-1} I \mi{q} \;\; = \;\; \mi{m} J \mi{m}^{-1} \cap \mi{q}^{-1} I \mi{q}. \end{equation}
The last equality holds since $\mi{m} J \mi{m}^{-1}$ is already contained in $W_{Js}$. Conjugating by $\mi{m}$, \eqref{thisoneimpliesthat} implies \eqref{thisoneisthat}.

It remains to prove \eqref{eq:leftredn}. Above we proved that $\rightred(p) = \rightred(m)$. Thus,
\begin{equation} \leftred(p) = \mi{p} \rightred(p) \mi{p}^{-1} = \mi{q} \mi{m} \rightred(m) \mi{m}^{-1} \mi{q}^{-1} = \mi{q} \leftred(m) \mi{q}^{-1}, \end{equation}
as desired.
\end{proof}


\subsection{The core of a double coset}

To place Theorem \ref{thmB} in context, we recall the notion of the core. Let $p$ be an $(I,J)$-coset with left redundancy $\leftred$ and right redundancy $\rightred$. In \cite[Proposition 4.28]{EKo} it is proven that 
\begin{equation} p = [[I \supset \leftred]] \; . \;p^{\core} \; . \; [[\rightred \subset J]] \end{equation}
for some $(\leftred,\rightred)$-coset called the \emph{core} of $p$. An illustrative example is below.

Said another way, $p$ has a singlestep reduced expression
\begin{equation} \label{rexthrucoreintro} [I, \ldots, \leftred, \ldots, \rightred, \ldots, J] \end{equation}
which begins by removing the elements of $I \setminus \leftred$ one by one and ends by adding the elements of $J \setminus \rightred$. The internal subword $[\leftred, \ldots, \rightred]$ is a reduced expression for $p^{\core}$. 
By \cite[Def. 4.26 and Prop. 4.28]{EKo} we have that $\leftred(p^{\core}) = \leftred$, $\rightred(p^{\core}) = \rightred$, and $\mi{p} = \mi{p^{\core}}$.

One also has
\begin{equation} \label{maxqcore} \ma{p^{\core}} = w_{\leftred} . \mi{p^{\core}}  = \mi{p^{\core}}. w_{\rightred}, \end{equation}
\begin{equation}  \ma{p} = (w_I w_{\leftred}^{-1}) . \ma{p^{\core}} . (w_{\rightred}^{-1} w_J). \end{equation}
The equalities above are not found in \cite{EKo} but can be proven quickly with \eqref{eq:mapmip}.

Returning to the setup of Theorem \ref{thmB}, where $n$ has reduced expression $[[\rightq \subset Js]]$, it is easy to deduce that
\begin{equation} q = [[I \supset \leftq]] \; . \; q^{\core} \; . \; n, \quad z = [[I \supset \leftq]] \; . \; q^{\core}. \end{equation}
To conclude, the factorization of the coset pair $p \subset q$ into $z$ times a Grassmannian coset pair $m \subset n$ is compatible with the reduced expression of $q$ which factors through its core. The existence of reduced expressions factoring through the core will also play a crucial role in the construction of light leaves in \S\ref{sec:LightLeavesConstruction}.

\begin{ex} \label{ex:coreA} We continue \Cref{ex:runningstart}.  Recall that $m$ is the double coset with minimal element depicted below on the left.
\begin{equation} \igm{exampleGrasscosetA} \qquad \igm{exampleGrasscore} \end{equation}
The picture on the right represents $m^{\core}$.

A reduced expression for $m^{\core}$ is (see Remark \ref{teaser})
\begin{equation} \label{expressionforpcoretypeA} m^{\core} \expr [\leftred(p) + s_2 - s_4 + s_7 - s_6 + s_5 - s_5 + s_8 - s_8 + s_6 - s_7 + s_{10} - s_9 + s_8 - s_8]. \end{equation}
Thus a reduced expression for $m$ is
\begin{align} \nonumber m \expr & [I - s_1 - s_5] \circ \\
\nonumber  & [+ s_2 - s_4 + s_7 - s_6 + s_5 - s_5 + s_8 - s_8 + s_6 - s_7 + s_{10} - s_9 + s_8 - s_8] \circ  \\ & [+ s_1 + s_4 + s_5 + s_7 + s_9]. \end{align}
\end{ex}

\begin{rem}\label{teaser} In a subsequent paper \cite{EKLP3} we explain how to find reduced expressions for core cosets in type $A$ (see also \cite[Section 3]{paper6}). Here is a teaser. For the coset $m^{\core}$ from Example \ref{ex:coreA}, we can express it as a reduced composition of other cosets as follows (colored for emphasis).
\begin{equation} \igm{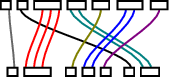} \qquad \igm{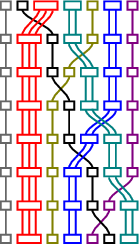} \end{equation}
Each crossing of two groups of strands becomes an instance of $[+ s_i - s_j]$ in \eqref{expressionforpcoretypeA}, where $i$ locates the top of the crossing and $j$ locates the bottom.
\end{rem}

 \section{The Hecke algebroid}\label{HA}

Most results in this chapter come directly from Geordie Williamson's thesis \cite{Wthesis}, which was adapted to the article \cite{SingSb}. The thesis contains some additional exposition and useful details which are omitted in the later article. We refer to \cite{SingSb} whenever possible. Starting with Definition \ref{defn:singulardefect}, the material is new.

 \subsection{Basics and bases}\label{Basics}

Let $H$ be the Hecke algebra of the Coxeter system $(W,S)$ with standard basis $\{h_w\, \vert \, w\in W\}$ and Kazhdan--Lusztig basis $\{b_w\, \vert \, w\in W\}$ over the ring $\Z[v,v^{-1}]$. For finitary $I\subseteq S$ recall that $w_I$ is the longest element of the parabolic subgroup $W_I.$ We define two polynomials in $\Z[v,v^{-1}]$ associated to $I$.

\begin{equation}\label{Poincpoly}\pi^+(I):=\sum_{w\in W_I}v^{2 \ell(w)},\qquad \pi(I):=v^{-\ell(w_I)}\pi^+(I).\end{equation}
If $I\subset K$ we have  \cite[Eq 2.9]{SingSb} \begin{equation} \label{bIsquared} b_{w_I}b_{w_K}=\pi(I)b_{w_K}.\end{equation}

Define the $\Z[v,v^{-1}]$-module $$\leftup{I}H^J:=b_{w_I}H\cap Hb_{w_J}.$$ If $p$ is an $(I,J)$-coset, define
$$h_p = \leftup{I}h_p^J:=\sum_{x\in p}v^{\ell(\overline{p})-\ell(x)}h_x.$$
Then $h_p \in \leftup{I}H^J$. Remember that the notation $p$ is shorthand for the triple $(p, I, J)$, and henceforth we only write the superscripts $I$ and $J$ for emphasis. If $p$ is an $(I,J)$-coset with maximal element $\ma{p}$, then the ordinary Kazhdan--Lusztig basis element $b_{\ma{p}}$ is also in $\leftup{I}H^J$. We denote it as
$$ b_p = \leftup{I}b_p^J := b_{\ma{p}}.$$

The \emph{Hecke algebroid} $\HA$ is the $\mathbb{Z}[v,v^{-1}]-$linear category defined as follows. The objects are finitary subsets $I\subseteq S$. The set of morphisms from $J$ to $I$, denoted $\HA(J,I)$, is the module $\leftup{I}H^J$. Composition $\leftup{I}H^J\times \leftup{J}H^K\rightarrow \leftup{I}H^K$  sends the pair $(h_1,h_2)$ to 
\begin{equation} h_1*_Jh_2:=\frac{1}{\pi(J)}h_1h_2.\end{equation}
Note that $b_{w_I}$ is the identity map in $\HA(I,I)$, by \eqref{bIsquared}.

The $\mathbb{Z}[v,v^{-1}]-$module $\leftup{I}H^J$ also has two bases. 
The \emph{standard basis} of $\leftup{I}H^J$ is the set $\{h_p \}$ and the \emph{Kazhdan--Lusztig basis} is the set $\{b_p\}$, as $p$ ranges amongst all $(I,J)$-cosets. 


\begin{rem} Unlike the standard basis of the ordinary Hecke algebra, standard basis elements $h_p$ are not invertible. As yet, there is no reasonable notion of a ``singular braid group.'' \end{rem}

\subsection{Generators of the Hecke algebroid}


Suppose that $I \subset J$. If $p_{\id}$ is the $(I,J)$-coset containing the identity then $\ma{p_{\id}} = w_J$ and $h_{p_{\id}} = b_{p_{\id}}.$ The same statement holds if $p_{\id}$ is the $(J,I)$-coset containing the identity. We define
\begin{equation} \leftup{I}\genha^J  :=\leftup{I}h_{p_{\id}}^J, \qquad \leftup{J}\genha^I :=\leftup{J}h_{p_{\id}}^I.\end{equation}

The elements $\leftup{I}\genha^{Is}$ and $\leftup{Is}\genha^I$ are generators of the Hecke algebroid, similar to the Kazhdan--Lusztig generators $b_s$ of the Hecke algebra. Here is notation for iterated products of these generators, which matches the fact that they will be categorified by singular Bott--Samelson bimodules.

\begin{notation} Let $I_{\bullet} = [I_0, \ldots, I_d]$ be a singlestep expression. Then
\begin{equation} \habs(I_{\bullet}) = \leftup{I_0}\genha^{I_1} *_{I_1} \leftup{I_1} \genha^{I_2} *_{I_2} \cdots *_{I_{d-1}} \leftup{I_{d-1}}\genha^{I_d}. \end{equation}
Similarly, when $L_{\bullet} = [[I_0 \subset K_1 \supset \cdots \subset K_m \supset I_m]]$ is a multistep expression, then
\begin{equation} \label{habsLbullet1} \habs(L_{\bullet}) = \leftup{I_0}\genha^{K_1} *_{K_1} \leftup{K_1}\genha^{I_1} *_{I_1} \cdots *_{K_m} \leftup{K_m}\genha^{I_m}.\end{equation}
\end{notation}

\begin{lem} When  $L_{\bullet} = [[I_0 \subset K_1 \supset \cdots \subset K_m \supset I_m]]$ is a multistep expression, we have
\begin{equation} \label{habsLbullet2} \habs(L_{\bullet}) = b_{K_1} *_{I_1} b_{K_2} *_{I_2} \cdots *_{I_{m-1}} b_{K_m}, \end{equation}
where we write $b_K$ for $b_{w_K}$.
\end{lem}

\begin{proof} For any $1 \le i \le m$ we have $\leftup{I_{i-1}}\genha^{K_i} = \leftup{K_i}\genha^{I_i} = b_{K_i}$. Thus 
$$ \leftup{I_{i-1}}\genha^{K_i} *_{K_i} \leftup{K_i}\genha^{I_i} = b_{K_i}.$$
Applying this relation many times within \eqref{habsLbullet1} we obtain \eqref{habsLbullet2}.
\end{proof} 

\subsection{The singular Deodhar formula}\label{singdeodhar}

Now we enunciate the key equations in the Hecke algebroid that lead to the singular Deodhar formula, and which will be categorified by the graded degrees of single-step singular light leaves. The following is \cite[Prop. 2.8]{SingSb}.

\begin{prop}\label{singDeodhar}
Let $I,J,K\subseteq S$ be finitary with $J \subset K$.
\begin{enumerate}
    \item When right-multiplying by $\leftup{K}\genha^{J}$ we have
\begin{equation} \leftup{I}h_q^K*_K \leftup{K}\genha^J=\sum_{p \subset q} v^{\ell(\overline{q})-\ell(\overline{p})}\ \leftup{I}h_p^J.\end{equation}
The sum is over all $(I,J)$-cosets $p$ contained in $q$.
    \item When right-multiplying by $\leftup{J}\genha^{K}$ we have
\begin{equation} \label{rightmultgoup} \leftup{I}h_p^J*_J \leftup{J}\genha^K= v^{\ell(\underline{q})-\ell(\underline{p})}\frac{\pi(\leftred(q))}{\pi(\leftred(p))}\ \leftup{I}h_q^K, \end{equation} where $q$ is the $(I,K)$-coset containing $p$.
\end{enumerate}
\end{prop}

Note  that \eqref{rightmultgoup} can be rewritten as
\begin{equation} \leftup{I}h_p^J*_J \leftup{J}\genha^K=
v^{\ell(\underline{q})-\ell(\underline{p})- \ell(\leftred(q)) + \ell(\leftred(p))} \frac{\pi^+(\leftred(q))}{\pi^+(\leftred(p))}\ \leftup{I}h_q^K. \end{equation}

Let us elaborate. An $(I,K)$-coset $q$ splits into a disjoint union of smaller $(I,J)$-cosets $p$. Multiplying by $\leftup{K}\genha^J$ sends $h_q$ to a sum over $h_p$, with each $h_p$ appearing once with a power of $v$ as its coefficient. When right multiplying to make the parabolic bigger, an $(I,J)$-coset $p$ is contained in a single $(I,K)$-coset $q$. Multiplication by $\leftup{J}\genha^K$ sends $h_p$ to a multiple of $h_q$. This time the coefficient of $h_q$ is not a monomial but a rescaled ratio of Poincar\'{e} polynomials. This ratio will end up being the rank of $R^{\leftred(p)}$ as a free module over $R^{\leftred(q)}$, these being two invariant subrings of a polynomial ring $R$, see \Cref{sec:realization}.

\begin{defn} \label{defn:singulardefect} Let $p \subset q$ be a coset pair (see Notation \ref{notation:cosetpair}) for $(I,J,s)$. 
We define the \emph{defect} and the \emph{polynomial factor} of the pairs $[p,q]$ and $[q,p]$ as follows.
\begin{alignat*}{3}
    &\defect([p,q])=\ell(\mi{q})-\ell(\mi{p})-\ell(\leftred(q))+\ell(\leftred(p)),\quad && \poly([p,q])=\displaystyle \frac{\pi^+(\leftred(q))}{\pi^+(\leftred(p))},\\
    &\defect([q,p])=\ell(\ma{q})-\ell(\ma{p}),&& \poly([q,p])=1.
\end{alignat*}
\end{defn}

Proposition \ref{singDeodhar} explains what happens at each step in an expression. Just as in the Deodhar formula, we can extrapolate this into a formula for $\habs(I_{\bullet})$ in terms of the standard basis.

\begin{defn}
Given a path $t_{\bullet}$ subordinate to $I_{\bullet}$ (see Definition \ref{defn:subordinatepath}), we define its \emph{defect} and \emph{polynomial factor} as  
\begin{equation} \defect(t_{\bullet}):=\sum_k \defect([t_k,t_{k+1}]), \qquad  \poly(t_{\bullet}):=\prod_k \poly([t_k,t_{k+1}]). \end{equation}
\end{defn}

\begin{cor} \label{cor:singDeodhar} Let $I_{\bullet}$ be a singlestep expression. Then
\begin{equation} \habs(I_{\bullet}) = \sum_{t_{\bullet} \subset I_{\bullet}} \poly(t_{\bullet}) v^{\defect(t_{\bullet})} h_{\term(t_{\bullet})}. \end{equation}
\end{cor}

\begin{proof} This is no more than an elaboration on the iteration of Proposition \ref{singDeodhar}. \end{proof}

We make two further observations on defects.

\begin{lem} Let $p \subset q$ be a coset pair for $(I,J,s)$.  Then
\begin{equation}
    \label{eq:defqpdefpq}
\defect([q,p])=\defect([p,q])+\ell(Js)-\ell(J).
\end{equation}
\end{lem}

\begin{proof}
By \eqref{eq:mapmip} we have
\[ \ell(\ma{p}) = \ell(\mi{p}) + \ell(I) + \ell(J) - \ell(\leftred(p)), \quad \ell(\ma{q}) = \ell(\mi{q}) + \ell(I) + \ell(Js) - \ell(\leftred(q)). \]
From this it is easy to deduce that $\defect([q,p]) - \defect([p,q]) = \ell(Js) - \ell(J)$. \end{proof}

\begin{lem}\label{degmnvspq}
Let $p \subset q$ be an $(I,J,s)$ coset pair, and $m \subset n$ be the associated Grassmannian coset pair from Theorem \ref{thmB}. Then $\defect([p,q]) = \defect([m,n])$ and $\defect([q,p]) = \defect([n,m])$. Also, $\poly([p,q]) = \poly([m,n])$ and $\poly([q,p]) = \poly([n,m])$. 
\end{lem}

\begin{proof}
From \Cref{thmBconti} we get that $\leftred(n) = \rightred(q)$, which is conjugate to $\leftred(q)$, and we get that $\leftred(m)$ is conjugate to $\leftred(p)$. Thus $\ell(\leftred(n)) = \ell(\leftred(q))$ and $\pi^+(\leftred(n)) = \pi^+(\leftred(q))$, and similarly for $m$ and $p$. Now one can easily verify the $\poly$ equalities.
From \eqref{mnzstuff} we also know that $\ell(\mi{q})+\ell(\mi{m})=\ell(\mi{p})+\ell(\mi{n})$. Now it is easy to verify that $\defect([m,n]) = \defect([p,q])$. Applying \eqref{eq:defqpdefpq} to both $p \subset q$ and $m \subset n$ we deduce that $\defect([q,p]) = \defect([n,m])$.
\end{proof}

\section{Singular Soergel bimodules} \label{sec:SSB}

Most definitions and results in this chapter come from Geordie Williamson's article \cite{SingSb}. There was an error in Williamson's definition of duality, which is now fixed in an erratum. We also needed some results on the ideal of lower terms which could not be found in Williamson's work. We developed these results in \cite{EKLP2}, and recall them in \Cref{sec:lowerterms}.

\subsection{Assumptions and notation}\label{sec:realization} Let $\Bbbk$ be a field. We fix a realization $V$ of $W$ over $\Bbbk$ (see \cite[Definition 3.1]{Soergelcalculus}). As in \cite[\S 3.1]{EKLP1}, we make some technical assumptions on our realizations. We require that $V$ is \emph{faithful}, \emph{balanced} and that \emph{generalized Demazure surjectivity} holds. We further require that the realization is \emph{reflection faithful}, so that the category of singular Soergel bimodules is well behaved (see \cite[\S 4.1]{SingSb}). 
We can always find a realization with $\Bbbk = \R$ satisfying all these properties (cf. \cite[Prop. 2.1]{Soe07}). 

\begin{rem} If the group $W$ is infinite, such a realization may not exist over a field of characteristic $p$. When the realization is not faithful or reflection faithful, singular Soergel bimodules need not be well-behaved. The diagrammatic version of the Hecke category should still behave well in characteristic $p$ (or even over $\Z$ if the group $W$ is crystallographic), once it is properly defined. This is one of the main motivations why such a diagrammatic description, of which the present work is propaedeutic, is desirable.  \end{rem}

\begin{notation} Let $R = \Sym(V)$ be the symmetric algebra of $V$, graded so that $\deg V = 2$. For each $I \subset S$ finitary, one can consider the subring $R^I$ of $W_I$-invariant polynomials. 
For any element $w\in W$, the corresponding Demazure operator $\pa_w$ is a $\Bbbk$-linear map of degree $-2\ell(w)$.  \end{notation}

\begin{lem}\label{lem:grankIJ}
    Let $I\subset J$ be finitary subsets of $S$. Then $R^I$ is a free $R^J$-module of graded rank $\frac{\pi^+(J)}{\pi^+(I)}$.
\end{lem}
\begin{proof}
Our assumptions ensure that the extension $R^J\subset R^I$ is Frobenius (see \cite[Theorem 4.3]{EKLP2}). The graded rank of the extension is computed in \cite[Corollary 2.1.4]{Wthesis}.
\end{proof}


\begin{notation} For $n \in \mathbb{Z}$, $(n)$ denotes the grading shift of a graded vector space by $n$. If $V$ is concentrated in degree $0$, then $V(n)$ is concentrated in degree $-n$.
\end{notation} 

\begin{notation}
 Given a Laurent polynomial with positive coefficients $p=\sum a_iv^i\in \mathbb{N}[v,v^{-1}]$ 
and a $\mathbb{Z}$-graded module $M$, we define\footnote{We remark that this convention is opposite to \cite{SingSb}.} 
\[p\cdot M:=\sum M(-i)^{\oplus a_i}.\]
\end{notation}

If $V$ is a $\Z$-graded vector space, we denote by $\gdim(V)$ its graded dimension. This is the Laurent polynomial defined by
\[ \gdim(V):= \sum \dim(V_i) v^i \in \mathbb{N}[v,v^{-1}].\]
If $p$ is a Laurent polynomial, we have
\begin{equation}\label{eq:gdimp} \gdim(p \cdot V)=p\cdot  \gdim(V).\end{equation}

\subsection{Definitions}

\begin{defn} Let $(W,S)$ be a Coxeter system and $R$ the polynomial ring of a realization. Given a singular multistep expression
\[ L_{\bullet}=[[ I_0\subset K_1\supset I_1\subset K_2 \supset \cdots \subset K_m\supset I_m]], \]
the corresponding \emph{(singular) Bott--Samelson bimodule} is the graded $(R^{I_0}, R^{I_m})$-bimodule
\begin{equation} \BS(L_{\bullet}) := R^{I_0} \ot_{R^{K_1}} R^{I_1} \ot_{R^{K_2}} \cdots \ot_{R^{K_m}} R^{I_m}(\sigma). \end{equation}
where
\begin{equation} \sigma=\ell(K_1)+\ldots-\ell(I_{m-1})+\ell(K_m) -\ell(I_m)=\frac{1}{2}(\ell(L_\bullet)+\ell(I_0)-\ell(I_m)). \end{equation}
\end{defn}

The length of an expression $\ell(L_\bullet)$ was defined in   \eqref{explength}.

\begin{notation}\label{restB}
Let $B$ be a graded $(R^K,R^L)$-bimodule. Let $K\subset I$ and $ L \subset J$. We denote by 
\[\restrict{R^I}{B}{R^J}:=R^I\otimes_{R^K}B\otimes_{R^L}{R^J}\]
the restriction of $B$ to a graded $(R^J,R^I)$-bimodule.
\end{notation}

The Bott--Samelson bimodule $\BS(L_{\bullet})$ is an iterated tensor product of two kinds of bimodules: the \emph{induction bimodule}
\[ \BS([[I \subset K]]) =
\restrict{R^I}{R^I}{R^K} \]
and the \emph{(shifted) restriction bimodule}
\[ \BS([[K \supset I]]) = \restrict{R^K}{R^I}{R^I}(\ell(K) - \ell(I)). \]
For $I \subset J \subset K$, the tensor product of two induction bimodules (resp. restriction bimodules) is naturally isomorphic to another induction bimodule (resp. restriction bimodule). Because of this, the Bott--Samelson bimodule $\BS(I_{\bullet})$ associated to a singlestep expression is naturally isomorphic to the Bott--Samelson bimodule for the associated multistep expression (obtained by remembering the maxima and minima).

Given a multistep expression $L_\bullet$, we denote by $\onetensor_{L_\bullet}$ (or simply by $\onetensor$) the element 
\begin{equation}\label{def:cbot}
\onetensor_{L_\bullet} :=1\otimes 1 \otimes \ldots \otimes 1\in \BS(L_\bullet).    
\end{equation}
Note that  $\deg(\onetensor_{L_\bullet})=-\frac12\left(\ell(L_\bullet)+\ell(I_0)-\ell(I_m)\right)$. It is an element of lowest degree in $\BS(L_\bullet)$, and spans that degree of $\BS(L_{\bullet})$ as a one-dimensional vector space.

\begin{defn}
    We denote by $\Bim$ the bicategory defined as follows. The objects in $\Bim$ are the finitary subsets $I\subset S$, identified with the graded algebras $R^I$. The category $\Bim(J,I)$ is the category of graded $(R^I,R^J)$-bimodules. The composition of $1$-morphisms $\Bim(J,I)\times \Bim(K,J)\to \Bim(K,I)$ is given by the tensor product over $R^J$.

    Given $M,N\in\Bim(J,I)$, the morphism space is the graded $(R^I,R^J)$-bimodule
\begin{equation}\label{eq.Hom}
    \Hom(M,N):=\bigoplus_{i\in \mathbb Z}\Hom^i(M,N), \qquad \Hom^i(M,N):=\Hom^0(M,N(i)).   
    \end{equation}
    Here $\Hom^0(M,N(i))$ denotes the space of degree zero $(R^I,R^J)$-bimodule maps from $M$ to $N(i)$. 
     The $(R^I,R^J)$-bimodule structure on $\Hom(M,N)$ is the obvious one.
\end{defn}

If $I_{\bullet} = [I_0, \ldots, I_d]$ is an expression, then $\BS(I_{\bullet})$ is an $(R^{I_0},R^{I_d})$-bimodule and hence an object of $\Bim(I_d,I_0)$.

\begin{defn}
The \emph{bicategory of singular Bott--Samelson bimodules}
$\SBSBim$ (which in the introduction we call $\SHC_{\BS}$\footnote{Up to equivalence. See the next footnote.}) is the subbicategory of $\Bim$ with the same objects, whose 1-morphisms have the form $\BS(I_{\bullet})$ for various expressions $I_{\bullet}$. Equivalently, $\SBSBim$ is the 2-full subbicategory in $\Bim$ generated by the 1-morphisms $\BS([I,Is])$ and $\BS([Is,I])$, for $Is\subset S$ finitary.

\end{defn}

\begin{defn}
    The \emph{bicategory  of singular Soergel bimodules} $\SSBim$ (which in the introduction we called $\SHC$\footnote{To be precise, what we denote by $\SHC$ in the introduction is a 2-category (bi)equivalent to $\SSBim$. 
    Note that the bicategory $\SSBim$, as well as the bicategory $\SBSBim$, is not a 2-category because the tensor product of bimodules does not satisfy associativity strictly but only up to natural isomorphism. However, the difference between bicategory and 2-category is not important, as the coherence theorem says that every bicategory is equivalent to a 2-category. In our case, such a 2-category $\SHC_{\BS}$ for $\SBSBim$ can be defined as a quotient of $\Frob$, by identifying the 2-morphisms mapped to isomorphic bimodule morphisms under the evaluation functor $\evaluation:\Frob\to \SBSBim$ (see Sections \ref{ss.Frob} and \ref{ss.FfromFrob}). 
    Then the additive Karoubi envelope of $\SHC_{\BS}$, which we denote by $\SHC$, is a 2-category equivalent to $\SSBim$.}) 
    is the additive closure of $\BSBim$ in $\Bim$, that is, $\SSBim$ has objects finitary subsets $I\subset S$ and the category $\SSBim(J,I)$ is the  category of graded $(R^I, R^J)$-bimodules consisting of the direct sums of shifts of direct summands of Bott--Samelson bimodules.
\end{defn}


\subsection{Classification of singular Soergel bimodules}

Singular Soergel bimodules provide a categorification of the Hecke algebroid. That is, there is an isomorphism between the split Grothendieck group of $\SSBim(J,I)$ and $\leftup{I}H^J$; for a singular Soergel bimodule $B$, we let $[B]$ denote the corresponding element of the Hecke algebroid. If $B\in \Hom(J,I)$ and $B'\in \Hom(K,J)$, we have 
\begin{equation}\label{tensorchar}
     [B\otimes_{R^J} B'] = [B]\ast_J [B'] \in {}^IH^K.
\end{equation}
Moreover,
\begin{equation}\label{BSel}
    \left[\BS([[I\subset K]])\right]={}^Ib^K\quad\text{ and }\quad
 [\BS([[K\supset I]])] ={}^Kb^I.
\end{equation} 
It immediately follows that
\begin{equation}\label{chBS}
    [\BS(I_\bullet)]=\habs(I_\bullet)
\end{equation} for every singlestep expression $I_\bullet$.

\begin{rem} The isomorphism between the Grothendieck group and the Hecke algebroid is embodied in the \emph{character map} defined by Williamson. 
For the definition of the character map and more details on the above we refer to \cite[\S 3.4]{EKLP2} and \cite[\S 6]{SingSb}.
\end{rem}

If $M$ is a $(R^I,R^J)$-bimodule, 
then its dual
$\duality M :=\Hom_{R^I}(M,R^I)$, the space of morphisms as left $R^I$-modules, is an $(R^I,R^J)$-bimodule. This defines the \emph{duality functor}
\[\duality:= \Hom_{R^I}(-,R^I):\Bim(J,I)\to\Bim(J,I).\]
The functor $\duality$ preserves Bott--Samelson bimodules by \cite[Prop. 6.15] {SingSb} and, as a consequence, the category of singular Soergel bimodules $\SSBim(J,I)$. Moreover, 
\[ [\duality B] = \overline{[B]},\]
where the overline represents the bar involution on the Hecke algebroid.

\begin{rem} \label{rmk:intersectionform}
Singular Bott--Samelson bimodules are isomorphic to their duals, but it is often important to fix this isomorphism. This is equivalent to constructing a non-degenerate bilinear form on $\BS(I_{\bullet})$. For ordinary Bott--Samelson bimodules, the \emph{intersection form} defined in \cite[\S 3.6]{EWhodge} is such a bilinear form. An intersection form for one-sided singular Soergel bimodules (with $I=\emptyset$) in \cite[Appendix A.1]{PatGrass}. The technology developed in this paper makes it possible to define an intersection form for singular Soergel bimodules in general, and we hope to address this topic in future work. \end{rem} 


Moreover, on $\SSBim(J,I)$ the functor $\duality$ is a anti-involution, i.e. for any $B,B'\in \SSBim(J,I)$  we have  $\duality^2 B \cong B$ and $\Hom(B,B')\cong \Hom(\duality B',\duality B)$. 


\begin{thm}
\label{thm::indec}
For every $(I,J)$-coset $p$, there exists a unique self-dual indecomposable singular Soergel bimodule $B_p$ such that 
\begin{equation}\label{[B_p]} [B_p] \in  h_p + \sum_{p'<p}\mathbb{N}[v^\pm] h_{p'},\end{equation}
where the sum is indexed by $(I,J)$-cosets $p'$ which are smaller in the Bruhat order.

Every indecomposable Soergel bimodule in $\SSBim(J,I)$ is isomorphic to $B_p(m)$, for some $(I,J)$-coset $p$ and some $m\in \Z$.

Moreover, for any reduced expression $I_{\bullet}$ for $p$ and for any decomposition of $\BS(I_{\bullet})$ into indecomposable Soergel bimodules, there is a unique summand containing $\onetensor_{I_\bullet}$, and this summand is isomorphic to $B_p$.
\label{thm:classification}\end{thm}

\begin{proof}
    The first two claims follow by {\cite[Theorem 7.10]{SingSb}}, while the last claim is proved in \cite[Prop. 3.3]{EKLP2}.
\end{proof}

\begin{cor}\label{cor:indecBS}
Let $I_\bullet$ be a singlestep expression, and $p$ be an $(I_0, I_d)$-coset. If $B_p(k)$ is a direct summand of $\BS(I_{\bullet})$ for some $k \in \Z$, then $p$ is the terminus of some path subordinate to $I_\bullet$.
\end{cor}

\begin{proof}
If $B_p(k)$ is a direct summand of $B$, then $v^{-k} h_p$ appears in $B$ with positive coefficient, by \eqref{[B_p]}. By 
Corollary \ref{cor:singDeodhar} and \eqref{chBS}, if $h_p$ appears with nonzero coefficient in $[\BS(I_{\bullet})]$ then $p$ is the terminus of a path subordinate to $I_{\bullet}$.
\end{proof}

\subsection{Standard bimodules and the Soergel--Williamson Hom formula}

\begin{defn} [{\cite[Definition 4.4]{SingSb}}]
Let $p$ be a $(I,J)$-coset. The \emph{(singular) standard bimodule} $R_p$ is the graded $(R^I, R^J)$-bimodule defined as follows. Let $K = \leftred(p) \subset J$ and $L = \rightred(p) \subset I$. As a left $R^I$-module, $R_p$ is $R^K$. The right action of $f \in R^J$ on $m \in R_p$ is given by
\begin{equation}\label{mf=pfm} m \cdot f = (\mi{p} f) m.
\end{equation}
In words, we say that the right action is twisted by $\mi{p}$. We also view $R_p$ as a ring via its identification with $R^K$.
\end{defn}

By \cite[Cor. 4.13]{SingSb} we have $\End(R_p)\cong R_p$ as both an $(R^I, R^J)$-bimodule and a ring.

We rephrase now the Soergel--Williamson Hom formula \cite[Theorem 7.9]{SingSb} in more convenient terms for our purposes.

\begin{thm}\label{SWhom} (Soergel--Williamson Hom Formula)
Let $B,B'\in \SSBim(J,I)$ and assume that $B'$ is self-dual. If 
\[[B] = \sum c_p h_p, \quad \mathrm{and} \quad [B'] = \sum c'_p h_p,\]
 then the graded dimension of their morphism space is given by the formula
\begin{equation}\label{singhom}  \gdim(\Hom(B,B')) = \sum c_p c'_p \cdot \gdim(R_p). \end{equation}
\end{thm}


\begin{proof}


For $p$ an $(I,J)$-coset, we define \[\pi(p):=v^{\ell(\ma{p})+\ell(\mi{p})}\sum_{x\in p}v^{-2\ell(x)}.\]
By \cite[Theorem 7.9]{SingSb} we have an isomorphism 
\begin{equation}\label{hom1}
\Hom(B,B')(-\ell(J))\cong \left\langle [B],\overline{[B']} \right\rangle \cdot R^I\cong  \langle [B],[B'] \rangle\cdot R^I
\end{equation}
of graded left $R^I$-modules, where the pairing $\langle-,-\rangle$ is defined in \cite{SingSb},  after Remark 2.2.
We can use \cite[Lemma 2.13]{SingSb} to expand the pairing,  and obtain
\begin{equation}\label{Bpairing}\langle [B],[B'] \rangle=\sum_p c_p c'_p v^{\ell(\ma{p})-\ell(\mi{p})-\ell(J)}\frac{\pi(p)}{\pi(J)}.
\end{equation}
Let $K$ be the left redundancy of $p$.
By \cite[(2.1) and (2.3)]{SingSb} we have
\begin{equation}\label{redlength}\ell(\mi{p})-\ell(\ma{p})+\ell(J)=\ell(K)-\ell(I) \quad \ \text{and}\quad \frac{\pi(p)}{\pi(J)}=\frac{\pi(I)}{\pi(K)}.\end{equation}
Moreover, by  \Cref{lem:grankIJ} we also have an isomorphism 
\begin{equation}\label{Rpdec}
R_p\cong R^K\cong v^{\ell(I)-\ell(K)}\frac{\pi(I)}{\pi(K)}\cdot R^I
\end{equation}
of graded left $R^I$-modules.
So, we conclude that
\begin{equation}\label{hom3}  \Hom(B,B')\cong \sum_p c_p c'_pv^{\ell(I)-\ell(K)}\frac{\pi(I)}{\pi(K)}\cdot R^I \cong \sum_p c_pc'_p\cdot R_p
\end{equation}
and the claim follows.
\end{proof}

\begin{cor}\label{BShom} 
Let $I_{\bullet}=[I_0,\ldots, I_d],I'_{\bullet}=[I'_0,\ldots, I'_r]$ be singlestep expressions with $I=I_0=I'_0$ and $J=I_d=I'_r$. Then, the morphism space $\Hom(\BS(I_\bullet),\BS(I'_\bullet))$ has graded dimension given by the formula
\begin{equation}\label{gradeddimhom}   \sum_{p\in W_I\backslash W/W_J
}  \sum_{
\substack{t_{\bullet}\subset I_{\bullet}, 
t'_{\bullet}\subset I'_{\bullet}\\
\mathrm{term}(t_{\bullet}) =\mathrm{term}(t'_{\bullet})=p
 }
} \mathrm{poly}(t_{\bullet})\mathrm{poly}(t'_{\bullet})v^{\mathrm{def}(t_{\bullet})+\mathrm{def}(t'_{\bullet})}\cdot \gdim\left(R_p \right). \end{equation}
\end{cor}

\begin{rem} It follows from \cite[Theorem 1.4]{SingSb} (thanks to the fact that Soergel's conjecture is now proven \cite{EWhodge}) 
 that for the reflection realization of $W$ over $\R$, Soergel's conjecture holds. That is, we have
\[ [B_p] ={}^Ib_p^J\]
for any $(I,J)$-coset $p$. In this case, it follows from Theorem \ref{SWhom} that the only endomorphisms of $B_p$ in degree $0$ are multiplication by scalars. We will not need this result in this paper. \end{rem}

\subsection{Lower terms}\label{sec:lowerterms}

Let $B$ and $B'$ be two singular Soergel bimodules. For a double coset $p$ we denote by $\Hom_{<p}(B,B')$ the ideal of  morphisms which factor through a direct sum of shifts of bimodules $B_q$, for $q<p$. Let $\Hom_{\not <p}(B,B')$ denote the quotient $\Hom(B,B')/\Hom_{<p}(B,B')$.

We recall two propositions from \cite{EKLP2} that characterize lower terms. 


\begin{prop}[{\cite[Prop 3.33]{EKLP2}}]\label{prop::kill1tensor}
Let $I_\bullet$ and $I'_\bullet$ be reduced expressions for a $(I,J)$-coset $p$. Then 
  \[\Hom^0_{<p}(\BS(I_\bullet),\BS(I'_\bullet))=\{ f\in  \Hom^0(\BS(I_\bullet),\BS(I'_\bullet))\mid f(\onetensor_{I_\bullet})=0\}\]
  where $\onetensor_{I_\bullet}$ is as in \eqref{def:cbot}.
\end{prop}

\begin{prop}[{\cite[Prop 3.34]{EKLP2}}]\label{prop:lowertermsR}
Let $I_\bullet$ and $I'_\bullet$ be reduced expressions for a $(I,J)$-coset $p$. Assume further that $I'_\bullet$ is of the form
\[I'_\bullet =[[I\supset \leftred(p)]]\circ K_\bullet\] and consider the action of $R^{\leftred(p)}$ on $\BS(I'_\bullet)$ restricted from that on $\BS(K_\bullet)$. 
If $f\in \Hom_{<p}(\BS(I_\bullet),\BS(I'_\bullet))$, then  $\im f\cap (R^{\leftred(p)} \cdot \onetensor_{I'_\bullet})= 0$. 
\end{prop}

We conclude this section by using the results in \cite{EKLP2} to determine the size of the space of morphisms modulo lower terms  to a reduced Bott--Samelson bimodule. 

\begin{prop}\label{gdimlowerterms}
 Let $B\in \SSBim$ with $[B] = \sum c_p h_p$. Let $I_\bullet$ be a reduced expression for a $(I,J)$-coset $p$.
 Then we have
\[\Hom_{\not< p}(B,\BS(I_\bullet))\cong \Hom_{\not <p}(B,B_p)\cong c_p\cdot R_p\]
\end{prop}
\begin{proof}
We know that all summands in $\BS(I_\bullet)$ other than $B_p$ are smaller than $p$, so they vanish modulo lower terms. This shows the first isomorphism.

    By \cite[Lemma 3.32]{EKLP2}  we have  $\Hom_{\not <p}(B,B_p)\cong\Hom(B,R_p(\ell(\ma{p})-\ell(J)))$. Finally, by \cite[Theorem 7.9]{SingSb}, we have
\[\Hom(B,R_p(\ell(\ma{p})-\ell(J)))\cong \langle [B_p],h_p\rangle \cdot R_p =c_p \cdot R_p.\qedhere\]
\end{proof}

\section{Singular diagrammatics and Frobenius extensions} \label{sec:diagrammatics}

This chapter contains no new material. We recall the diagrammatic calculus introduced in \cite{ESW}, primarily for the purpose of efficiently describing morphisms between singular Soergel bimodules. We also discuss Demazure operators and Frobenius extensions, which play a role for both the diagrammatics and the proofs. Experts may wish to review the assumptions from \Cref{sec:realization}.

\subsection{Demazure operators} \label{subsec:demazure}

We assume familiarity with the Demazure operators $\pa_w \co R \to R$ associated with $w \in W$.  See \cite[\S 3.2]{EKLP1} for a helpful review.

Let $p$ be an $(I,J)$-coset. As in \cite[Definition 3.8]{EKLP1} we set
\begin{equation} \label{eq:pap} \pa_p := \pa_{\ma{p} w_J^{-1}}. \end{equation}
Though a priori this is a map $R \to R$, \cite[Lemma 3.9]{EKLP1} proves that, when restricted to the subring $R^J$, it has image contained in $R^I$. By convention, $\pa_p$ is always considered as a map $R^J \to R^I$.

A special case is when $J \subset I$, and $p$ has underlying set $W_I$. The operator $\pa_{p} = \pa_{w_I w_J^{-1}} \co R^J \to R^I$ is denoted $\pa^J_I$. We call it a \emph{trace map}.

A second special case is when $I \subset J$ and $p$ has underlying set $W_J$. The operator $\pa_p$ is the inclusion map $\iota^I_J \co R^J \hookrightarrow R^I$.

In \cite[Corollary 3.19]{EKLP1} it is proven  that (for cosets $p, q, r$ which are appropriately composable)
\begin{equation} p.q = r \implies \pa_p \circ \pa_q = \pa_r, \end{equation}
In particular, given a reduced expression $L_{\bullet}$ for $p$ as in \eqref{ex}, one can view the operator $\pa_p$ as an iterated composition of inclusion maps $\iota^{K_i}_{I_{i-1}}$ and trace maps $\pa^{I_i}_{K_i}$.

For any double coset expression $L_{\bullet}$ one can consider the corresponding iterated composition of inclusion and trace maps, and denote it by $\pa_{L_{\bullet}}$. See \cite[Lemma 3.16]{EKLP1} for more details. It is proven in \cite[Proposition 3.17]{EKLP1} that $\pa_{L_{\bullet}} = \pa_p$ if $L_{\bullet}$ is any reduced expression for $p$, and $\pa_{L_{\bullet}} = 0$ if $L_{\bullet}$ is not a reduced expression.

We note one other special case of \eqref{eq:pap}. If $p^{\core}$ is the core of $p$, then
\begin{equation} \label{eq:papcore} \pa_{p^{\core}} = \pa_{\mi{p}} \end{equation}
as a map from $R^{\rightred(p)}$ to $R^{\leftred(p)}$. This is an easy consequence of \eqref{eq:mapmip}.

\subsection{Cubes of Frobenius extensions} \label{subsec:cubes}

\begin{defn} Let $S$ be a finite set. Let $P(S)$ be the power set of $S$, the set of all subsets of $S$. Then $P(S)$ is a poset under inclusion. Let $F \subset P(S)$ be a subset which is downward-closed: if $I \in F$ and $J \subset I$ then $J \in F$. Such an $F$ is called a \emph{partial cube}. \end{defn}

\begin{defn} A \emph{(partial) cube of (commutative) graded Frobenius extensions} indexed by a partial cube $F$ is the data of: \begin{itemize} \item for each $I \in F$, a graded commutative ring $R^I$, and an integer $\ell(I)$, \item for each $I \subset J \in F$, an injective ring homomorphism $\iota^I_J : R^J \rightarrow R^I$, and a homogeneous $R^J$-linear map $\pa^I_J \co R^I \to R^J$. 
\end{itemize} This data satisfies: \begin{enumerate} \item For $I \subset J \in F$, $\pa^I_J$ is a Frobenius trace map, making $R^J \subset R^I$ a Frobenius extension. \item For $I \subset J \subset K \in F$, $\iota^I_K = \iota^I_J \circ \iota^J_K$ and $\pa^I_K = \pa^J_K \circ \pa^I_J$. \end{enumerate}
\end{defn}

We recall Frobenius extensions below, but first we give the main example.

\begin{defn} \label{defn:Soergelcube} Fix a Coxeter system $(W,S)$.
Let $\Bbbk$ and $V$ and $\{\alpha_s, \alpha_s^\vee\}$ and $R$ be as in \Cref{sec:realization}.
Let $F \subset P(S)$ be the set of all finitary subsets of $S$. The \emph{Soergel cube} is the cube of graded Frobenius extensions where \begin{itemize}
\item For $I \subset S$ finitary, $R^I$ is the invariant subring, and $\ell(I)$ is the length of the longest element.
\item For $J \subset I \subset S$ finitary, $\iota^J_I \co R^I \to R^J$ is the inclusion map, and $\pa^J_I=\pa_{w_Iw_J^{-1}} \co R^J \to R^I$ is the Frobenius trace map. \end{itemize}
\end{defn}

That the Soergel cube satisfies the requirements for a cube of Frobenius extensions follows from \cite[\S 4.1]{EKLP1}.

Let $I \subset J \in F$. That $R^J \subset R^I$ is a Frobenius extension 
with trace $\pa^I_J$, by definition, says that 
there exist \emph{dual bases} $\{c_i\}$ and $\{d_i\}$ for $R^I$ as a free $R^J$-module, such that
\begin{equation} \pa^I_J(c_i \cdot d_j) = \delta_{ij}. \end{equation}
The \emph{coproduct element}
\begin{equation} \label{dedef} \De^I_J := \sum_i c_i \ot d_i \in R^I \ot_{R^J} R^I \end{equation}
is independent of the choice of dual bases, so we sometimes write this element in Sweedler notation as
\[ \De^I_J = \De^I_{J(1)} \ot \De^I_{J(2)}. \]
Because it satisfies the property
\begin{equation} f \cdot \De^I_J = \De^I_J \cdot f \end{equation}
for any $f \in R^I$, the map
\begin{equation} \label{coproddef} \De \co R^I \to R^I \ot_{R^J} R^I, \qquad \De(f) = f \cdot \De^I_J \end{equation}
is a morphism of $R^I$-bimodules.

The element
\begin{equation} \mu^I_J = \sum_i c_i d_i \in R^I \end{equation}
is called the product-coproduct element for the extension. For the Soergel cube it has a closed formula:
\begin{equation} \mu^I_J = \prod_{\alpha \in \Phi^+_J \setminus \Phi^+_I} \alpha. \end{equation}
Here, $\Phi^+_I$ represents the set of positive roots of the form $w(\alpha_s)$ for $w \in W_I$ and $s \in I$. 
That is, $\mu^I_J$ is the product of all the positive roots for $J$ which are not positive roots for $I$. We write $\mu_I := \mu^{\emptyset}_{I}$.

In the Soergel context, for $I \subset J$ finitary, note that
\begin{subequations}
\begin{equation} \BS([[J \supset I \subset J]]) =  R^I(\ell(J) - \ell(I)) \text{ as an $(R^J, R^J)$ bimodule,} \end{equation}
\begin{equation} \BS([[I \subset J \supset I]]) =  R^I \ot_{R^J} R^I(\ell(J) - \ell(I)) \text{ as an $(R^I, R^I)$ bimodule.} \end{equation}
\end{subequations}

The data of a graded Frobenius extension gives rise to four homogeneous bimodule maps for each $I \subset J \in F$, which are independent of the choice of dual bases. Two of these are $(R^J, R^J)$-bimodule maps:
\begin{subequations} \label{fourstructuremaps}
\begin{equation} \iota^I_J \co R^J \to \BS([[J \supset I \subset J]]), \qquad \pa^I_J \co \BS([[J \supset I \subset J]]) \to R^J. \end{equation}
The grading shifts above are chosen so that these maps are both homogeneous of (negative) degree $\ell(I) - \ell(J)$. 

The other two are $(R^I, R^I)$-bimodule maps:
\begin{equation} m : \BS([[I \subset J \supset I]]) \to R^I, \qquad \Delta : R^I \to \BS([[I \subset J \supset I]]). \end{equation}
\end{subequations}
These are both homogeneous of (positive) degree $\ell(J) - \ell(I)$. Above, $m$ is the multiplication map $m(f \ot g) = fg$, and $\Delta$ was defined in \eqref{coproddef}.

\begin{rem} The functors of induction, i.e. tensoring with $R^I = \BS([[I \subset J]])$ as an $(R^I, R^J)$-bimodule, and shifted restriction, i.e. tensoring with $R^I(\ell(J) - \ell(I)) = \BS([[J \supset I]])$ as an $(R^J,R^I)$-bimodule, are biadjoint up to shift. The four maps above are the units and counits of this biadjunction.  \end{rem}

Later in the paper we will need the following concept.

\begin{defn} \label{defn:almostdual} Let $I \subset J \in F$. Homogeneous bases $\{a_i\}$ and $\{b_i\}$ for $R^I$ as a free $R^J$-module are called \emph{almost dual bases} if
\begin{equation} \pa^I_J(a_i \cdot b_j) \equiv  \delta_{ij} \text{ modulo } \mg, \end{equation}
where $\mg$ represents the homogeneous maximal ideal of $R^J$. \end{defn}

Given almost dual bases $\{a_i\}$ and $\{b_i\}$, one can construct dual bases $\{a_i\}$ and $\{b'_i\}$, where $\{b_i\}$ and $\{b'_i\}$ are related by unitriangular change of basis (with respect to an ordering by degree), see \cite[Lemma 4.7]{EKLP1}. 

\subsection{Diagrammatics for the Soergel cube: remarks}

In \cite{ESW} they describe a 2-category $\Frob$ by generators and relations, associated to any cube of Frobenius extensions. It comes equipped with a 2-functor $\evaluation$ to singular Bott--Samelson bimodules (in the case of the Soergel cube).

In general, there is no guarantee that $\evaluation$ should be full. However, all the singular double leaves we construct in \S\ref{sec:LightLeavesConstruction} will be in the image of $\evaluation$. We prove that double leaves form a basis under reasonable assumptions (e.g. for faithful realizations over a field), and thus $\evaluation$ is full in this case.

The functor $\evaluation$ is rarely expected to be faithful; there are more relations one must impose upon $\Frob$ to obtain a presentation for $\SBSBim$. These relations are unknown in general. See \cite[Chapter 24]{GBM} for remarks on the status of singular Soergel diagrammatics. We will not need any of the additional relations in this paper.


\subsection{Diagrammatics for the Soergel cube}\label{ss.Frob}

We now define the graded 2-category $\Frob$ following \cite{ESW}.

The objects of $\Frob$ are finitary subsets of $S$; therefore, finitary subsets will label the regions of our string diagrams. The $1$-morphisms are singlestep expressions $I_{\bullet}$. We encode (the identity morphism of) an object with a sequence of oriented strands, colored by simple reflections. An $\sberry$-colored strand separates $I$ and $I\sberry$ for some $I \subset S$, and the orientation is such that $I\sberry$ is to the right of the strand.

\begin{ex} Let $I_{\bullet} = [J\sberry \urial,J\urial, J\teal\urial, J\sberry\teal\urial, J\sberry\teal]$. It is drawn as 
\begin{equation} \id_{I_{\bullet}} = \quad {
\labellist
\tiny\hair 2pt
 \pinlabel {$J\sberry\urial$} [ ] at 0 31
\endlabellist
\centering
\igm{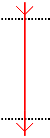}
}{
\labellist
\tiny\hair 2pt
 \pinlabel {$J\urial$} [ ] at 0 31
\endlabellist
\centering
\igm{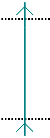}
}{
\labellist
\tiny\hair 2pt
 \pinlabel {$J\teal\urial$} [ ] at 0 31
\endlabellist
\centering
\igm{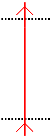}
}{
\labellist
\tiny\hair 2pt
 \pinlabel {$J\sberry\teal\urial$} [ ] at 0 31
 \pinlabel {$J\sberry\teal$} [ ] at 25 31
\endlabellist
\centering
\igm{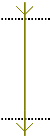}
} \end{equation}

We often use $\pm$ notation for singlestep expressions, as in \cite[Notation 5.1]{EKo}. Continuing the above example, we have
\begin{equation} [(J \sberry\urial) - \sberry + \teal + \sberry - \urial] := [J\sberry \urial,J\urial, J\teal\urial, J\sberry\teal\urial, J\sberry\teal]. \end{equation}
The notation $[J\sberry \urial,J\urial, \ldots]$ records the region labels, while the notation $[(J \sberry\urial) - \sberry + \ldots]$ emphasizes the sequence of oriented colored strands. \end{ex}

\begin{rem} As a helpful mnemonic for the color-sighted, $\sberry$ is for ``strawberry'' and $\teal$ is for ``teal.'' \end{rem}

Note that an upward $\sberry$-colored strand means different things in different contexts: it could be a $1$-morphism from $I$ to $I\sberry$, or a $1$-morphism from $I'$ to $I'\sberry$, for $I, I'$ distinct. This is disambiguated by the region labels. Note that only one region requires a label, as the labels on other regions are determined by the coloring of the strands. Even the orientation on the strands is redundant information given one region label, though it helps one's visual understanding of a picture.


The $2$-category $\Frob$ will be cyclic, see \cite[p.131]{GBM}. This means that the category possesses oriented cups and caps, denoted as follows. Under the evaluation functor, they will go to the units and counits of adjunction with the same name from \eqref{fourstructuremaps}. By convention, the bottom boundary of a diagram is the source of the morphism, and the top boundary is the target.
\begin{equation} \label{frobcupcap} m = {
\labellist
\tiny\hair 2pt
 \pinlabel {$I$} [ ] at 31 47
 \pinlabel {$I\sberry$} [ ] at 31 17
\endlabellist
\centering
\igs{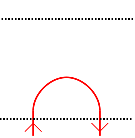}
}, \qquad \Delta = {
\labellist
\tiny\hair 2pt
 \pinlabel {$I\sberry$} [ ] at 31 47
 \pinlabel {$I$} [ ] at 31 17
\endlabellist
\centering
\igs{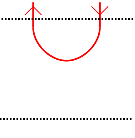}
}, \qquad \pa = {
\labellist
\tiny\hair 2pt
 \pinlabel {$I\sberry$} [ ] at 31 47
 \pinlabel {$I$} [ ] at 31 17
\endlabellist
\centering
\igs{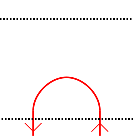}
}, \qquad \iota = {
\labellist
\tiny\hair 2pt
 \pinlabel {$I$} [ ] at 31 47
 \pinlabel {$I\sberry$} [ ] at 31 17
\endlabellist
\centering
\igs{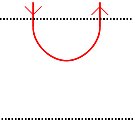}
}. \end{equation}
The degree of clockwise cups and caps is $\ell(Is) - \ell(I) > 0$, and the degree of counterclockwise cups and caps is $\ell(I) - \ell(Is) < 0$.

All $2$-morphisms will be cyclic with respect to these cups and caps, so that isotopic diagrams represent the same $2$-morphism. With this assumption, we now give an isotopy presentation (see \cite[Proposition 7.18]{GBM}) of the category.

The $2$-morphisms will be monoidally generated by two kinds of maps (in addition to cups and caps). For each $\sberry, \teal \in S$ and $I \subset S$ such that $I\sberry\teal$ is finitary, there is the \emph{(up-facing) crossing map} from $[I,I\teal,I\sberry\teal]$ to $[I,I\sberry,I\sberry\teal]$.
\begin{equation} \label{upcross} {
\labellist
\tiny\hair 2pt
 \pinlabel {$I\sberry\teal$} [ ] at 55 32
 \pinlabel {$I\teal$} [ ] at 31 15
 \pinlabel {$I\sberry$} [ ] at 31 45
 \pinlabel {$I$} [ ] at 7 32
\endlabellist
\centering
\igs{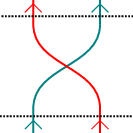}
} \end{equation}
The degree of the crossing is zero. For each $f \in R^I$ homogeneous with $I \subset S$ finitary, there is the \emph{polynomial map}, whose degree is the degree of $f$.
\begin{equation} {
\labellist
\tiny\hair 2pt
 \pinlabel {\small $f$} [ ] at 12 33
 \pinlabel {$I$} [ ] at 20 17
\endlabellist
\centering
\igs{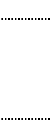}
} \end{equation}

By cyclicity, one can define the left-facing crossing $[I\sberry,I,I\teal] \to [I\sberry,I\sberry\teal,I\teal]$:
\begin{equation} \label{defsideways} \igs{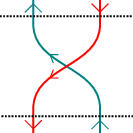} := \igs{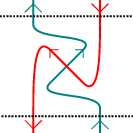}. \end{equation}
Right-facing (sideways) and down-facing crossings are defined similarly. Down-facing crossings also have degree zero. The degree of a sideways crossing (right- or left-facing) is $\ell(I) + \ell(I\sberry\teal) - \ell(I\sberry) - \ell(I\teal)$, which one can remember as ``big plus small minus mediums.''

\begin{rem} \label{rmk:samecolor} The advantage of using the same picture (ignoring region labels) for both the $1$-morphisms $I \to I\sberry$ and $I\teal \to I\sberry\teal$ is that a crossing looks like the transverse union of two differently-colored $1$-manifolds (with boundary). Indeed, any diagram in $\Frob$ is a tranverse union of colored $1$-manifolds, decorated with polynomials and region labels. \end{rem}

The relations between $2$-morphisms in $\Frob$ are the following.
First, there are relations stating that polynomials add and multiply as expected.
Next is the polynomial sliding relation, which says for $f \in R^{I\sberry} \subset R^I$ that
\begin{equation} \label{polyslide} {
\labellist
\tiny\hair 2pt
 \pinlabel {\small $f$} [ ] at 12 33
 \pinlabel {$I$} [ ] at 20 17
\endlabellist
\centering
\igs{space}
}\igs{ups}{
\labellist
\tiny\hair 2pt
 \pinlabel {$I\sberry$} [ ] at 20 17
\endlabellist
\centering
\igs{space}
} = {
\labellist
\tiny\hair 2pt
 \pinlabel {$I$} [ ] at 20 17
\endlabellist
\centering
\igs{space}
}\igs{ups}{
\labellist
\tiny\hair 2pt
 \pinlabel {\small $f$} [ ] at 12 33
 \pinlabel {$I\sberry$} [ ] at 20 17
\endlabellist
\centering
\igs{space}
}. \end{equation}
As a consequence, polynomials slide from region to region within $\End(I_{\bullet})$ precisely as they slide between tensor factors in $\BS(I_{\bullet})$.

Associated to the Frobenius extension for $I \subset I\sberry$ we have the circle evaluation relations
\begin{equation} \label{circrelns} {
\labellist
\tiny\hair 2pt
 \pinlabel {$Is$} [ ] at 36 26
 \pinlabel {$I$}  [ ] at 55 18
\endlabellist
\centering
\igs{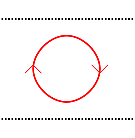}
} = {
\labellist
\small\hair 2pt
 \pinlabel {$\mu^I_{Is}$} [ ] at 12 33
\endlabellist
\centering
\igs{space}
}, \qquad {
\labellist
\tiny\hair 2pt
 \pinlabel {$I$} [ ] at 38 25
 \pinlabel {\small $f$} [ ] at 31 33
 \pinlabel {$Is$}  [ ] at 55 18
\endlabellist
\centering
\igs{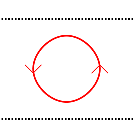}
} = \igs{space} {
\labellist
\small\hair 2pt
 \pinlabel {$\pa^I_{Is}(f)$} [ ] at 5 33
\endlabellist
\centering
\igs{space}
}\;\;. \end{equation}
for any $f \in R^I$, and the idempotent decomposition relation
\begin{equation} \label{idempdecomp} \igs{downs} \igs{ups} = {
\labellist
\tiny\hair 2pt
 \pinlabel {$\Delta_{(1)}$} [ ] at 32 16
 \pinlabel {$\Delta_{(2)}$} [ ] at 32 47
\endlabellist
\centering
\igs{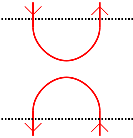}
}. \end{equation}
Here, $\Delta = \Delta^I_{Is} \in R^I \ot_{R^{Is}} R^I$ is the coproduct element as in \eqref{dedef}, and Sweedler notation means that the right-hand side is a sum of diagrams with polynomials. It makes sense to put elements of this tensor product in the diagram because of the polynomial sliding relation, see the discussion around \eqref{polyslide}.

Associated to a square of Frobenius extensions $I \subset I\sberry\teal$ we have the easy Reidemeister II relation
\begin{equation} \label{R2easy} \igs{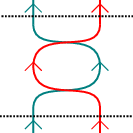} = \igs{upt}\igs{ups} \end{equation}
and the hard Reidemeister II relations
\begin{equation} \label{R2hard} \igs{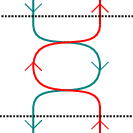} = \igs{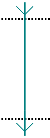} {
\labellist
\tiny\hair 2pt
 \pinlabel {\small $\mu^I_{I\sberry\teal}$} [ ] at 12 33
\endlabellist
\centering
\igs{space}
} \igs{ups}, \qquad 
{
\labellist
\small\hair 2pt
 \pinlabel {$g$} [ ] at 31 32
\endlabellist
\centering
\igs{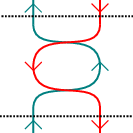}
} = \pa\Delta^I_{I\sberry\teal}(g)_{(1)} \igs{upt} \igs{downs} \pa\Delta^I_{I\sberry\teal}(g)_{(2)}. \end{equation}
Above we have
\begin{equation} \mu^I_{I\sberry\teal} = \frac{\mu_I \mu_{I\sberry\teal}}{\mu_{I\sberry}\mu_{I\teal}} =  \prod_{\alpha \in \Phi^+_{I\sberry\teal} \setminus (\Phi^+_{I\sberry} \cup \Phi^+_{I\teal})} \alpha, \end{equation}
while  the element $\pa\De^I_{I\sberry\teal}(g) \in R^{I\sberry} \ot_{R^{I\sberry\teal}} R^{I\teal}$ can be described by one of the following four equivalent formulas (see \cite[page 11]{ESW}).
\begin{eqnarray} \label{eq:leftcrossingnotsym}\pa\De^I_{I\sberry\teal}(g) & = &  \De^{I\sberry}_{I\sberry\teal(1)} \ot \pa^I_{I\teal}(g \cdot \De^{I\sberry}_{I\sberry\teal(2)}) \\ & =  & \pa^I_{I\sberry}(g \cdot \De^{I\teal}_{I\sberry\teal(1)}) \ot \De^{I\teal}_{I\sberry\teal(2)} \nonumber \\ & = &  \pa^I_{I\sberry}(\De^I_{I\sberry\teal(1)}) \ot \pa^I_{I\teal}(g \cdot \De^I_{I\sberry\teal(2)}) \nonumber\\ & =  & \pa^I_{I\sberry}(g \cdot \De^I_{I\sberry\teal(1)}) \ot \pa^I_{I\teal}(\De^I_{I\sberry\teal(2)}).\label{eq:leftcrossing} 
\end{eqnarray}

Associated to a 3D cube $I \subset I\sberry\teal\urial$ of Frobenius extensions we have the Reidemeister III relations (easy and hard)
\begin{equation} \label{eq:R3} \igs{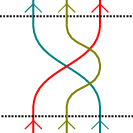} = \igs{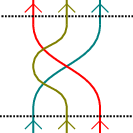}, \qquad  \igs{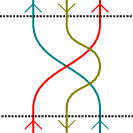} = {
\labellist
\small\hair 2pt
 \pinlabel {$*$} [ ] at 22 33
\endlabellist
\centering
\igs{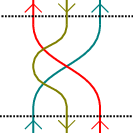}
}. \end{equation}
Above, $*$ represents the polynomial $\mu^I_{I\sberry\teal\urial} \in R^I$, where
\begin{equation} \mu^I_{I\sberry\teal\urial} = \prod_{\alpha \in \Phi^+_{I\sberry\teal\urial} \setminus (\Phi^+_{I\sberry\teal} \cup \Phi^+_{I\sberry\urial} \cup \Phi^+_{I\teal\urial})} \alpha. \end{equation}

This ends the definition of $\Frob$.

\subsection{Evaluation}\label{ss.FfromFrob}

\begin{defn} Let $\evaluation$ denote the graded 2-functor $\Frob \to \SBSBim$ defined below. On objects, it sends $I$ to $I$, for $I \subset S$ finitary. On $1$-morphisms, it sends $I_{\bullet} \mapsto \BS(I_{\bullet})$. Cups and caps are sent to the corresponding structure maps for the Frobenius extension, see \eqref{fourstructuremaps} and \eqref{frobcupcap}. Finally, the crossing from \eqref{upcross} is sent to the natural isomorphism
\begin{equation} R^I \ot_{R^{Is}} R^{Is}_{R^{Ist}} \to R^I_{R^{Ist}} \to R^I \ot_{R^{It}} R^{It}_{R^{Ist}}, \qquad 1 \ot 1 \mapsto 1 \mapsto 1 \ot 1. \end{equation}
\end{defn}

It helps to know what happens to the other crossings under $\evaluation$. The down-facing crossing also goes to a natural isomorphism which preserves the 1-tensor. The right-facing sideways crossing goes to a map $R^{Is} \ot_{R^{Ist}} R^{It} \to R^I$, which is just multiplication.
\begin{equation} \evaluation \co \igs{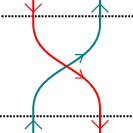} \mapsto \left( f \ot g \mapsto fg \right). \end{equation}
The left-facing sideways crossing goes to the map $R^I \to R^{Is} \ot_{R^{Ist}} R^{It}$ from \eqref{eq:leftcrossing}:
\begin{equation} \evaluation \co \igs{leftcross} \mapsto \left( f \mapsto \pa\De^I_{Ist}(f) \right). \end{equation}

\subsection{Duality} \label{subsec:duality}

\begin{defn} \label{defn:duality} The \emph{duality functor} $\duality$ is a contravariant, monoidally-covariant autoequivalence of $\Frob$. It fixes all objects and $1$-morphisms. It flips all $2$-morphisms upside-down, and then reverses the orientation. Polynomials are fixed by $\duality$. \end{defn}

\begin{lem} The duality functor is well-defined, and preserves degree. \end{lem}

\begin{proof} We need only confirm that $\duality$ preserves each of the relations above. This is actually a trivial exercise. The only relation which is not obviously fixed by $\duality$ is \eqref{idempdecomp}, where $\De_{(1)}$ and $\De_{(2)}$ are swapped. However, the element $\De \in R^I \ot_{R^{Is}} R^I$ is independent of the choice of dual bases, and by swapping the two bases one swaps these two Sweedler symbols. \end{proof}

\begin{rem} One would like a result stating that duality on $\Frob$ intertwines with duality on $\SSBim$, via the evaluation functor $\evaluation$. However, the objects of $\Frob$ are self-dual by construction, whereas singular Bott--Samelson bimodules are only isomorphic to their duals. To construct a natural isomorphism $\duality \circ \evaluation \to \evaluation \circ \duality$, one would need to fix isomorphisms between singular Bott--Samelson bimodules and their duals, see Remark \ref{rmk:intersectionform}. Again, we hope to address this in future work. \end{rem}

\part{Construction of double leaves}

\section{Rex moves}\label{ch.rexmove}


As in \cite{EKo}, we write $I_{\bullet} \expr I'_{\bullet}$ whenever these are both reduced expressions for the same double coset $p$. It was proven in \cite[Theorem 5.30]{EKo} that any two such expressions are related by a sequence of \emph{double coset braid relations}.

In this chapter we recall the double coset braid relations and construct degree zero morphisms in $\Frob$ associated to each. Compositions of these we call \emph{rex moves}. We also prove various properties of rex moves.


\subsection{Elementary rex moves}\label{sec:rexmoves}

There are three kinds of braid relation found in \cite{EKo}: the up-up relations, the down-down relations, and the switchback relations. Let us construct a diagram for each relation.

Whenever $I\sberry \teal$ is finitary, the \emph{up-up relation} states that 
\[[I+\sberry+\teal]\expr [I+\teal+\sberry].\]
To this relation we associate the pair of morphisms
\begin{equation} {
\labellist
\tiny\hair 2pt
 \pinlabel {$I$} [ ] at 7 32
 \pinlabel {$Ist$} [ ] at 55 32
\endlabellist
\centering
\igs{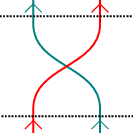}
} \qquad {
\labellist
\tiny\hair 2pt
 \pinlabel {$I$} [ ] at 7 32
 \pinlabel {$Ist$} [ ] at 55 32
\endlabellist
\centering
\igs{upcross}
}.\end{equation} The first is a morphism from the left-hand side to the right-hand side, and the other is vice versa. Note that these are inverse isomorphisms by \eqref{R2easy}.

Whenever $I\sberry \teal$ is finitary, the \emph{down-down relation} states that
\[[I\sberry \teal-\sberry-\teal] \expr [I\sberry \teal-\teal-\sberry].\]
Similarly, we associate to this the pair of inverse isomorphisms 
\begin{equation} \igs{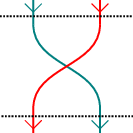} \qquad \igs{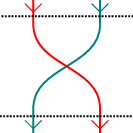}.\end{equation}




A switchback relation has the form 
\begin{equation}\label{eq:ss}
[I +{\color{red} u_0} -{\color{teal} u_d}]  \expr [I-{\color{Brown} u_1}+{\color{red} u_0} -u_2 +{\color{Brown} u_1} -u_3+u_2 \cdots -u_{d-1}+u_{d-2} -{\color{teal} u_d}+u_{d-1}],    
\end{equation}
for some sequence $(u_0, \ldots, u_d)$ called a \emph{rotation sequence}. There is one such relation for each finitary subset $L=I {\color{red} u_0}$ and each pair ${\color{red} u_0}, {\color{teal} u_d} \in L$ with ${\color{teal} u_d} \ne w_L {\color{red} u_0} w_L$. For details, see \cite[\S 5, \S 6]{EKo}. Note that $\{u_i\}$ need not be distinct, and possibly $u_0 = u_d$ which will force us to change our color scheme.

To this relation we associate a pair of degree zero morphisms, whose construction is clear from the following examples.
To put it in words anyway, the strands
$+ {\color{red} u_0}$ and $- {\color{teal} u_d}$
appear in both the source and target, and correspond to vertical lines in the diagram. The remaining inputs or outputs are paired with each other in counterclockwise caps or cups. Unlike the previous relations, these are \textbf{not} inverse isomorphisms.

\begin{ex} (see \cite[Example 5.20]{EKo}) Let $W = S_{10}$ and $I = S \setminus s_3$. Set ${\color{red} u_0} =  s_3$ and ${\color{teal} u_d} = s_6$. Then $d=2$ and ${\color{Brown} u_1} = s_9$. The switchback relation is
\begin{equation} [I + {\color{red} s_3} - {\color{teal} s_6}] \expr [I - {\color{Brown} s_9} + {\color{red} s_3} - {\color{teal} s_6} + {\color{Brown} s_9}], \end{equation}
with corresponding elementary rex moves
\begin{equation} \label{rexmoveA} \igm{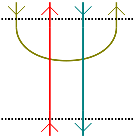} \qquad \qquad \igm{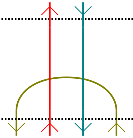}. \end{equation}
\end{ex}

\begin{ex} (see \cite[\S 6.2]{EKo}) Let $W$ be the dihedral group $I_2(5)$ with simple reflections $S = \{\sberry, \teal\}$. Let $u_0 = u_d = \sberry$ (so $u_d$ is not teal in this example), and $I = \{\teal\}$. Then $d = 4$ and the switchback relation is
\begin{equation} [I + \sberry - \sberry] \expr [I - \teal + \sberry - \sberry + \teal - \teal + \sberry - \sberry + \teal], \end{equation}
with corresponding elementary rex moves
\begin{equation} \label{rexmoveI} \igm{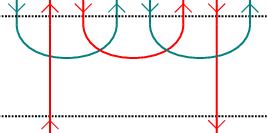} \qquad \qquad \igm{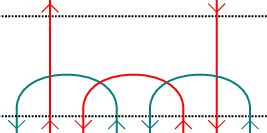}. \end{equation}
\end{ex}

\begin{ex} \label{ex:E7} (see \cite[Example 6.4]{EKo}) Let $W$ have type $E_7$. We number the simple reflections $S$ as in \cite[\S 6.5.2]{EKo}. Let $I = S \setminus s_2$, with ${\color{red} u_0} = s_2$ and ${\color{red} u_d} = s_3$. Then $d = 4$ and the switchback relation is
\begin{equation} [I + {\color{red} s_2} - {\color{teal} s_3}] \expr [I - {\color{Brown} s_6} + {\color{red} s_2} - {\color{green} s_5} + {\color{Brown} s_6} - {\color{purple} s_7} + {\color{green} s_5} - {\color{teal} s_3} + {\color{purple} s_7}], \end{equation}
with corresponding elementary rex moves
\begin{equation} \label{rexmoveE} \igm{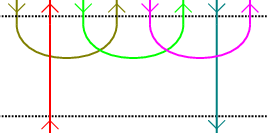} \qquad \qquad \igm{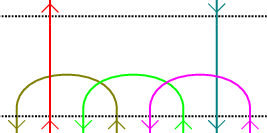}. \end{equation}
\end{ex}

The following lemma should not be obvious to the reader, and is proven as part of Theorem \ref{thm:braidmovespreservecbot}.

\begin{lem}\label{lem:rexmovesdegzero} Elementary rex moves have degree zero. \end{lem}

\subsection{Rex moves}

\begin{defn} \label{defn:rexmove} Let $I_{\bullet}$ be a singlestep expression. Apply a sequence of braid relations to $I_{\bullet}$ to obtain $I'_{\bullet}$. Composing the corresponding sequence of elementary rex moves, we get a morphism $I_{\bullet} \to I'_{\bullet}$ in $\Frob$ called a \emph{rex move}. Applying $\evaluation$, we get a morphism $\BS(I_{\bullet}) \to \BS(I'_{\bullet})$ which we also call a \emph{rex move}. \end{defn}

\begin{ex} The reduced expressions $[\mt,s,\mt,t,\mt,s,\mt]$ and $[\mt,t,\mt,s,\mt,t,\mt]$ can be related by a sequence of four braid relations, as in \cite[(1.8)]{EKo}. Here is the corresponding rex move in singular calculus. For pedagogical reasons we draw it twice, noting the second time how it is the composition of four elementary rex moves.
\begin{equation} \igm{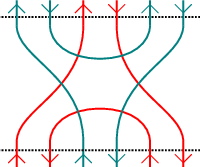} \qquad \qquad \igm{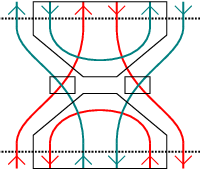} \end{equation}
This is the morphism which is encoded by the 6-valent vertex, an elementary rex  
move in ordinary Soergel calculus, see \cite[Theorem 24.46]{GBM}.
\end{ex}

\subsection{Rex moves and direct summands} \label{subsec:braidpiisone}

The elementary rex moves associated to switchback relations are not isomorphisms, but as we prove here, they are inclusions from and projections to a direct summand.

\begin{thm} \label{thm:braidpiisone} The rex move in $\Frob$ corresponding to the composition 
\begin{equation} \label{braidpiisone} [I +{\color{red} u_0} - {\color{teal} u_d}]  \to [I-{\color{Brown} u_1}+{\color{red} u_0} -u_2 +{\color{Brown} u_1} \cdots -{\color{teal} u_d} +u_{d-1}] \to [I +{\color{red} u_0} - {\color{teal} u_d}] \end{equation}
is the identity morphism. \end{thm}

We prove the theorem later in this section.

\begin{cor} Inside $\SSBim$, the bimodule $\BS([I +{\color{red} u_0} - {\color{teal} u_d}])$ is a direct summand of $\BS([I-{\color{Brown} u_1}+{\color{red} u_0} -u_2 +{\color{Brown} u_1} \cdots -{\color{teal} u_d} +u_{d-1}])$. \end{cor}

\begin{proof} In any additive idempotent-complete category with objects $M$ and $N$, the statement that $M$ is a summand of $N$ is equivalent to the existence of morphisms $i \co M \to N$ and $p \co N \to M$ such that $p \circ i = \id_M$. \end{proof}

Note that all finitary subsets of $S$ in sight are subsets of $I u_0$. For the purpose of proving Theorem~\ref{thm:braidpiisone}, we will assume that $S = Iu_0$.

\begin{notation} \label{circlenotation} Here is the abusive notation which we use for the rest of \S\ref{subsec:braidpiisone}. We let $S = I u_0$. We write $\hat{u_i}$ for $S \setminus u_i$, so that $I = \hat{u_0}$ and $J = \hat{u_d}$. \end{notation}

Using Example \ref{ex:E7} where $d=4$ to illustrate the point, Theorem \ref{thm:braidpiisone} states that
\begin{equation}\label{braidpiisonediagram}
{
\labellist
\tiny\hair 2pt
 \pinlabel {$\hat{u_0}$} [ ] at 6 16
 \pinlabel {\small $S$} [ ] at 62 8
 \pinlabel {$\hat{u_d}$} [ ] at 122 16
 \pinlabel {$\hat{u_1}$} [ ] at 32 29
 \pinlabel {$\hat{u_2}$} [ ] at 63 29
 \pinlabel {$\hat{u_3}$} [ ] at 94 29
\endlabellist
\centering
\igm{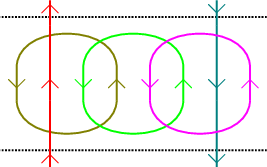}
}
=
{
\labellist
\tiny\hair 2pt
 \pinlabel {$\hat{u_0}$} [ ] at 6 21
 \pinlabel {$S$} [ ] at 23 21
 \pinlabel {$\hat{u_d}$} [ ] at 40 21
\endlabellist
\centering
\igm{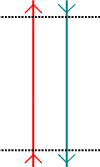}
}. \end{equation}

Theorem \ref{thm:braidpiisone} is a consequence of a more general circle evaluation lemma.

\begin{lem} \label{lem:circleresolution} Let $S$ be finitary. Let $d \ge 2$ and $u_i \in S$ for $0 \le i \le d$, with $u_i$ distinct from $u_{i+1}$ for all $0 \le i \le d-1$. Let $I = S \setminus {\color{red} u_0}$, and $M = S \setminus \{u_{d-1},u_d\}$. Consider the expressions
\begin{equation} K_{\bullet} := [I +{\color{red} u_0} - {\color{teal} u_d}], \qquad L_{\bullet} := [I-{\color{Brown} u_1}+{\color{red} u_0} -u_2 +{\color{Brown} u_1} \cdots -{\color{teal} u_d} +u_{d-1}]. \end{equation}
Both are well-defined, and $L_{\bullet}$ is not necessarily reduced. Let $g$ be any polynomial in $R^M$. Then we have
\begin{equation} \label{eq:circleresformula} {
\labellist
\tiny\hair 2pt
 \pinlabel {$\hat{u_0}$} [ ] at 6 16
 \pinlabel {\small $S$} [ ] at 62 8
 \pinlabel {$\hat{u_d}$} [ ] at 122 16
 \pinlabel {$\hat{u_1}$} [ ] at 32 23
 \pinlabel {$\hat{u_{d-1}}$} [ ] at 94 23
\pinlabel {\small $g$} [ ] at 112 31
\endlabellist
\centering
\igm{doublerexcircles}
} = \quad \pa_{L'_{\bullet}}(g \Delta^{\hat{u_d}}_{S, (1)}) {
\labellist
\tiny\hair 2pt
 \pinlabel {$\hat{u_0}$} [ ] at 6 21
 \pinlabel {$S$} [ ] at 23 21
 \pinlabel {$\hat{u_d}$} [ ] at 40 21
\endlabellist
\centering
\igm{doublerexidentity}
} \Delta^{\hat{u_d}}_{S, (2)},
\end{equation}
as endomorphisms of $K_\bullet$, where the left-hand side factors through $L_\bullet$. On the right-hand side, Sweedler notation indicates the action of an element of $R^{\hat{u_0}} \ot_{R^S} R^{\hat{u_d}}$ which is a sum of pure tensors. The expression $L'_{\bullet}$ equals $[I-{\color{Brown} u_1}+{\color{red} u_0} -u_2 +{\color{Brown} u_1} \cdots -{\color{teal} u_d}]$, so that $L_{\bullet} = L'_{\bullet} \circ [+u_{d-1}]$. See \S\ref{subsec:demazure} for a reminder on what $\pa_{L'_{\bullet}}$ means. Note that $L'_{\bullet}$ is an $(I,M)$-expression, so the input to $\pa_{L'_{\bullet}}$ is an element of $R^M$. Implicitly, we have included $\Delta^{\hat{u_d}}_{S,(1)}$ from $R^{\hat{u_d}}$ into $R^M$.
\end{lem}

We prove Lemma \ref{lem:circleresolution} in Section \ref{cel}.

\begin{proof}[Proof of Theorem \ref{thm:braidpiisone}] Let us apply Lemma \ref{lem:circleresolution} with $g = 1$. Let $K_{\bullet}$ and $L_{\bullet}$ be the two reduced expressions in the statement of the theorem.

While Lemma~\ref{lem:circleresolution} introduces the expression $L'_{\bullet}$, the operators $\pa_{L_{\bullet}}$ and $\pa_{L'_{\bullet}}$ both agree with the application of $\pa_w$ for the same $w \in W$. The only difference is that $\pa_{L_{\bullet}}$ begins with the inclusion map $R^{\hat{u_d}} \hookrightarrow R^M$. Since we have implicitly applied this inclusion map to $\Delta^{\hat{u_d}}_{S,(1)}$, we use $\pa_{L_{\bullet}}$ below instead of $\pa_{L'_{\bullet}}$.

So the right-hand side of \eqref{braidpiisonediagram} when $g=1$ is equal to
\begin{equation} \pa_{L_{\bullet}}(\Delta^{\hat{u_d}}_{S, (1)}) \ot \Delta^{\hat{u_d}}_{S, (2)}. \end{equation}
But $L_{\bullet}$ is a reduced expression for the $(\hat{u_1},\hat{u_d})$-coset containing $w_S$ (not surprising, since $K_{\bullet}$ is also a reduced expression for this coset). By \cite[Prop. 3.17]{EKLP1} we have $\pa_{L_{\bullet}} = \pa^{\hat{u_d}}_{S}$. Then by \cite[Equation (2.2) with $f=1$]{ESW}, we know that
\begin{equation} \pa^{\hat{u_d}}_S(\Delta^{\hat{u_d}}_{S, (1)}) \ot \Delta^{\hat{u_d}}_{S, (2)} = 1 \ot 1, \end{equation}
as desired.
\end{proof}

\subsection{Rex moves and 1-tensors} \label{ss.rexandonetensor}

Let us now discuss the images of the rex moves under the functor $\evaluation$. Recall the definition of the 1-tensor $\onetensor$ from \Cref{def:cbot}.

\begin{thm} \label{thm:braidmovespreservecbot} If $\phi$ is the image of a rex move under $\evaluation$, then $\phi(\onetensor) = \onetensor$. Moreover, rex moves have degree zero. \end{thm}

\begin{proof} For up- or down-facing crossings, this is immediate from the definition of $\evaluation$.

The diagram associated to the elementary rex move 
\[ \phi \co [I +{\color{red} u_0} - {\color{teal} u_d}]  \to [I-{\color{Brown} u_1}+{\color{red} u_0} -u_2 +{\color{Brown} u_1} \cdots -{\color{teal} u_d} +u_{d-1}]\]
is built entirely of right-facing crossings and counterclockwise cups, see e.g. the left-hand sides of \eqref{rexmoveA},\eqref{rexmoveI},\eqref{rexmoveE}. These generators preserve $1^{\ot}$, and thus so does $\phi$.  Note that the source and the target are reduced expressions for the same double coset, so they have the same length. Therefore the 1-tensors in source and target have the same degree, and $\phi$ must have degree zero.

Finally, consider the diagram associated to the elementary rex move (see e.g. the right-hand sides of \eqref{rexmoveA},\eqref{rexmoveI},\eqref{rexmoveE})
\[  \psi \co [I-{\color{Brown} u_1}+{\color{red} u_0} -u_2 +{\color{Brown} u_1} \cdots -{\color{teal} u_d} +u_{d-1}] \to [I +{\color{red} u_0} - {\color{teal} u_d}]. \]
 Since duality preserves degree, $\psi$ also has degree zero. For degree reasons it will send $1^{\ot} \mapsto \lambda 1^{\ot}$ for some scalar $\lambda$. By \eqref{braidpiisonediagram}, the composition $\psi \circ \phi$ is the identity map, sending $1^{\ot} \mapsto 1^{\ot}$. Therefore $\lambda = 1$, as desired. \end{proof}

\begin{cor}\label{cor:rex1tensor} Let $I_{\bullet}$ be a reduced expression for $p$, and consider any rex move $\phi \co I_{\bullet} \to I_{\bullet}$. Then $\phi \equiv \id$ modulo the ideal $\End_{<p}(\BS(I_\bullet))$. \end{cor}

\begin{proof}
By Proposition \ref{prop::kill1tensor}, the ideal $\End_{< p}(\BS(I_\bullet))$ intersects $\End^0(\BS(I_{\bullet}))$ in the set of morphisms which kill $\onetensor$. By Theorem \ref{thm:braidmovespreservecbot}, any rex move preserves $\onetensor$, and thus $\id - \phi$ kills $\onetensor$, and lives in $\End_{<p}(\BS(I_\bullet))$. \end{proof} 

By this corollary, rex moves in $\SSBim$ are isomorphisms modulo lower terms.

\begin{rem} In order for the diagrammatic category to behave correctly, one requires that rex move endomorphisms are equal to the identity modulo lower terms. This is not true in $\Frob$, but requires relations specific to the Soergel setting. For dihedral groups, it is a consequence of the Elias--Jones--Wenzl relation \cite[(6.8)]{Bendihedral}. For non-singular diagrammatics, one needs the Zamolodchikov relations for larger Coxeter groups. We do not address the question of what it takes to prove the analogous result for singular Soergel diagrammatics in this paper. \end{rem}

\section{Construction of light leaves and double leaves}\label{sec:LightLeavesConstruction}

In this chapter, we describe the algorithmic process of constructing light leaf morphisms and the double leaves basis. We will do it in five stages.  It is complicated enough that we prefer not to intersperse too much motivation and lose the flow. Let us just briefly motivate our starting point.

 Recall from  Definition \ref{defn:subordinatepath} the notion of a subordinate path $t_{\bullet} \subset I_{\bullet}$. At the $k$-th step either $[I_k,I_{k+1}] = [J,Js]$ or $[I_k,I_{k+1}] = [Js,J]$ for some $J$ and $s$. Moreover, either $[t_k,t_{k+1}] = [p,q]$ or $[t_k,t_{k+1}] = [q,p]$, for some $(I,J,s)$-coset pair $p \subset q$, where $I=I_0$. 
 Our singlestep light leaf will be some morphism
\[ \ELL([p,q]) \co X_p \ot [J,Js] \to X_q \qquad \text{ or } \qquad \ELL([q,p]) \co X_q \ot [Js,J] \to X_p, \]
where $X_p$ and $X_q$ are some reduced expressions for $p$ and $q$ respectively. Thanks to Theorem \ref{thmB} we need only construct this map when $p \subset q$ is a Grassmannian pair, for reasons the reader will soon see. This is where we begin.

\subsection{First stage: Elementary light leaves for \texorpdfstring{$[m,n]$}{[m,n]}}\label{ELL}

In this section \S\ref{ELL} the running assumption is that $I,Js\subset S$ are finitary subsets such that $I\subset Js$. We let $[m,n]$ be a Grassmannian pair for $(I,J,s)$. To recall, this means that $m$ is an $(I,J)$-coset, $n$ is the $(I,Js)$-coset containing the identity, and $m \subset n$. In particular, a reduced multistep expression for $n$ is $[[I \subset Js]]$.

By \cite[Proposition 4.28]{EKo}, the coset $m$ has a reduced expression which factors through a reduced expression for its core $m^{\core}$, see \S\ref{ssec:proofGrass} for more details.

\begin{notation} \label{notation:ELL} For the rest of section \S\ref{ELL} we fix a reduced expression $K_{\bullet}$ for the $(\leftred(m),\rightred(m))$-coset $m^{\core}$. We also fix enumerations
\begin{equation} I \setminus \leftred(m) = \{t_1, \ldots, t_l\}, \qquad J \setminus \rightred(m) = \{u_1, \ldots, u_r\}, \end{equation}
which will determine an enumeration
\begin{equation} Js \setminus I = \{v_1, \ldots, v_k\} \end{equation}
as in Lemma \ref{lem:enumeratev} below. Associated with the multistep reduced expressions 
\begin{equation} m \expr [[I \supset \leftred(m)]] \circ K_{\bullet} \circ [[\rightred(m) \subset J]], \qquad n \expr [[I \subset Js]], \end{equation} the enumerations above determine singlestep reduced expressions
\begin{equation} m \expr X_m := [I - t_1 - \ldots - t_l] \circ K_{\bullet} \circ [\rightred(m) + u_1 + \ldots + u_r], \;\; n \expr X_n := [I + v_1 + \ldots + v_k].
\end{equation}
Let $I_{\bullet}$ be the (possibly non-reduced) expression $X_m \circ [J,Js]$.
\end{notation}

\begin{lem} \label{lem:enumeratev} For each simple reflection $t \in S$, one of the following is true.
\begin{itemize} \item $t\in I$. Then $t$ appears an even number of times in $I_{\bullet}$, alternating as in $[I \ldots - t \ldots + t \ldots -t \ldots +t \ldots]$. In particular, $-t$ comes before $+t$.
\item $t \in Js \setminus I$. Then $t$ appears an odd number of times in $I_{\bullet}$, alternating as in $[I \ldots + t \ldots - t \ldots +t \ldots -t \ldots]$. In particular, $+t$ comes before $-t$. 
\item $t \notin Js$. Then $\pm t$ does not appear in $I_{\bullet}$.
\end{itemize}
There is a unique enumeration of $Js \setminus I$ as $\{v_1, \ldots, v_k\}$, so that $i < j$ implies that the first appearance of $+v_i$ in $I_{\bullet}$ comes to the left of the first appearance of $+v_j$.
\end{lem}

The proof of this lemma is completely straightforward, but we postpone most proofs from this chapter to \Cref{proofs} to avoid distracting the flow.

\begin{defn} \label{defn:ELL} We now construct a diagram
\begin{equation} \ELL([m,n]) \co X_m \circ [J,Js] = I_{\bullet} \to X_n \end{equation}
in $\Frob$, built entirely from left-facing crossings and counterclockwise caps. For each $t \in Js$, the underlying $t$-colored 1-manifold in the diagram has the following form, depending on whether $t \in I$ or $t \in Js \setminus I$:
\begin{equation} \label{twooptions} t \in I: \;\; \igs{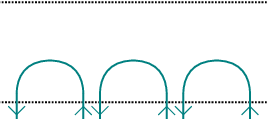}, \qquad t \in Js \setminus I: \;\; \igs{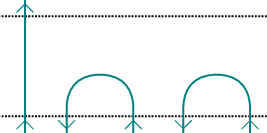}. \end{equation}
These are then superimposed upon each other transversely, so that no two strands cross more than once (and the rules of $\Frob$ are obeyed, e.g. no triple intersections).  The same rules apply to the case $t = s$ as to any other simple reflection $t \in Js$. \end{defn}

\begin{lem} \label{lem:noclock}
There exists a diagram (given by an explicit construction) that meets the criteria of Definition \ref{defn:ELL}. The corresponding morphism in $\Frob$ has degree equal to $\defect([m,n])$. \end{lem}

See \S\ref{sec:sinister} for the proof. We now give several prototypical examples. In all of them, we assume $m \subset n$ is an $(I,J,\sberry)$ Grassmannian coset pair.

\begin{ex} \label{ex:dihedralsinister} In this example $I = \{\sberry\}$ and $J = \{\teal\}$, $J\sberry = \{\sberry, \teal\}$ generates a finite dihedral group with $m_{s,t} \ge 6$. Let $m$ be the $(\sberry,\teal)$-coset with maximal element $ststst$, and note that $\leftred(m) = \rightred(m) = \emptyset$. One can choose $X_m = [\sberry, \emptyset, \teal, \emptyset, \sberry, \emptyset, \teal, \emptyset, \sberry, \emptyset, \teal]$. We have
\begin{equation} \ELL([m,n]) = {
\labellist
\tiny\hair 2pt
 \pinlabel {$I$} [ ] at 6 38
 \pinlabel {$J$} [ ] at 128 17
 \pinlabel {$Js$} [ ] at 144 38
\endlabellist
\centering
\igs{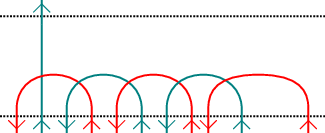}
}. \end{equation}
\end{ex}

\begin{ex} \label{ex:ELLA} We continue Example \ref{ex:coreA}. Recall that $W$ is the symmetric group $S_{11}$, and $W_J \cong S_8 \times S_3$, and $\sberry = s_8$, so that $W_{Js} = W$. We set $W_I = S_2 \times S_5 \times S_1 \times S_2 \times S_1$. Let $m$ be depicted on the left below, with $m^{\core}$ on the right.
\begin{equation} \igm{exampleGrasscosetA} \qquad \igm{exampleGrasscore} \end{equation}
A reduced expression for $m^{\core}$ was given in \eqref{expressionforpcoretypeA}, which we repeat here.
\begin{equation} m^{\core} \expr [\leftred(p) + s_2 - s_4 + s_7 - s_6 + s_5 - s_5 + s_8 - s_8 + s_6 - s_7 + s_{10} - s_9 + s_8 - s_8]. \end{equation}
Choosing the lexicographic order on $I \setminus \leftred(m)$ and $J \setminus \rightred(m)$, we get
\begin{align} \nonumber X_m = & [I - s_1 - s_5] \circ \\
\nonumber  & [+ s_2 - s_4 + s_7 - s_6 + s_5 - s_5 + s_8 - s_8 + s_6 - s_7 + s_{10} - s_9 + s_8 - s_8] \circ  \\ & [+ s_1 + s_4 + s_5 + s_7 + s_9] \end{align}
Note that $I_{\bullet} = X_m \circ [+ s_8]$. The corresponding elementary light leaf is
\begin{equation} \label{ELLA} \ELL([m,n]) = {
\labellist
\tiny\hair 2pt
 \pinlabel {$I$} [ ] at 4 36
 \pinlabel {$\leftred(m)$} [ ] at 37 15
 \pinlabel {$\rightred(m)$} [ ] at 223 15
 \pinlabel {$J$} [ ] at 301 15
 \pinlabel {$J\sberry$} [ ] at 313 36
\endlabellist
\centering
\igm{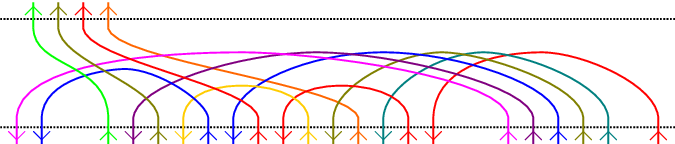}
}. \end{equation}
\end{ex}

\begin{ex} \label{ex:typeAannoying} Consider the Grassmannian coset pair $m \subset n$ from Example \ref{ex:runningpq}. This is essentially the same as the previous example (which is $m_{C_{\sberry}} \subset n_{C_{\sberry}}$), except there is an exterior product with other cosets $m_C = n_C$ for which $\mi{m_C} = \id$. One can follow the same algorithm for $m$ as for $m_{C_{\sberry}}$, and the result will be extremely similar. In this particular example, there are two new simple reflections in $J \setminus \rightred(m)$ which were not in $(J \cap C_{\sberry}) \setminus \rightred(m_{C_{\sberry}})$, which we denote using the colors black and gray below. The picture below is $\ELL([m,n])$. The new $\sberry$-black and $\sberry$-gray crossings have degree zero, as black and gray live in a different connected component from $\sberry$.
\begin{equation} \label{ELLAannoying} \igm{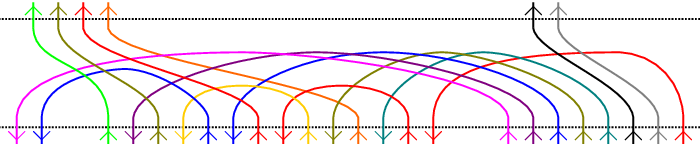} \end{equation}

\end{ex}

\begin{ex}\label{ex.ELLred} Whenever $I \subset J$ and $m$ contains the identity element, the inclusion $[m,n]$ is reduced in the sense of \cite[Definition 2.29]{EKo}. In this case, $m^{\core}$ is the identity $(I,I)$-coset, and $X_m = [[I \subset J]]$. The elementary light leaf is the identity map.
\begin{equation} \ELL([m,n]) = {
\labellist
\tiny\hair 2pt
 \pinlabel {$I$} [ ] at 1 26
 \pinlabel {$J$} [ ] at 57 26
 \pinlabel {$Js$} [ ] at 82 26
\endlabellist
\centering
\igm{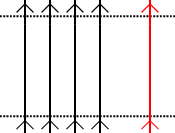}
} \end{equation}
\end{ex}

\begin{ex} \label{ex:ELLidentitycap} Whenever $I \not\subset J$ then $\sberry \in I$, since $I \subset J \sberry$.  If $m$ is the $(I,J)$-coset containing the identity element, $\leftred(m) = I \cap J = I \setminus \sberry$. In this case, $m^{\core}$ is the identity $(I\cap J,I \cap J)$-coset, and $X_m = [[I \supset I \setminus s \subset J]]$. The elementary light leaf is as below.
\begin{equation} \ELL([m,n]) = {
\labellist
\tiny\hair 2pt
 \pinlabel {$I$} [ ] at 8 40
 \pinlabel {$I \setminus \sberry$} [ ] at 23 15
 \pinlabel {$J$} [ ] at 80 15
 \pinlabel {$J\sberry$} [ ] at 96 40
\endlabellist
\centering
\igm{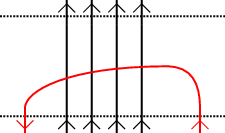}
} \end{equation}
\end{ex}

\begin{ex}\label{ex.ELLD}
Let $S = \{\sberry, \teal, \urial,c\}$ have type $D_4$, where $c$ is the hub of the Dynkin diagram. Let $J = S \setminus \sberry$, and $I = \{\sberry, c\}$. Let $n$ be the unique $(I,J\sberry)$-coset.

Then there are four $(I,J)$-cosets with minimal elements
\[\mi{m_1} = \id,\quad \mi{m_2} = \teal c\sberry,\quad \mi{m_2'} = \urial c\sberry,\quad \mi{m_3} = \urial\teal c\sberry. \]

The $(I,J)$-coset $m_1$ falls under the regime of \Cref{ex:ELLidentitycap}, so
\begin{equation} \ELL([m_1,n]) = \igm{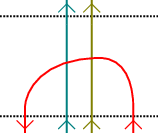}. 
\end{equation}


A reduced expression for $m_2^{\core}$ and the corresponding reduced expression for $m$ are 
\[m_2^{\core} \expr [I +\teal - \sberry], \qquad X_m = [I + \teal - \sberry + \urial].\]
We have 
\begin{equation} \ELL([m_2,n])= \igm{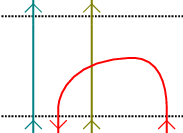}. \end{equation}

The case of $m_2'$ is the same, swapping $\teal$ and $\urial$. Swapping these colors will change the reduced expression $X_n$ for $n$, as required by the enumeration in \Cref{lem:enumeratev}. The effect of changing $X_n$ is that the $\teal$ and $\urial$ strands need not cross.

The core of $m_3$ is an $(\sberry,c)$-coset, two reduced expressions for which are
\[m_3^{\core} \expr [\{\sberry\}+\teal - \teal +\urial -\urial +c -\sberry ]\expr [\{\sberry\}+\teal +\urial  -\urial - \teal +c -\sberry ].\]
To these one composes with $[I - c]$ on the left and $[\{c\} + \teal + \urial]$ on the right to obtain $X_m$.
The two choices above give rise to two different options for $\ELL([m_3,n])$, namely
\begin{equation} \label{eq:ELLD3} \ELL([m_3,n]) = \igm{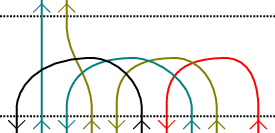} \quad \text{or} \quad \igm{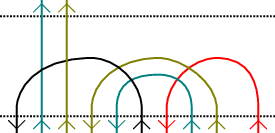}. \end{equation}
There are more reduced expressions than these, e.g. one can swap $\teal$ and $\urial$ (which would necessitate a different choice of $X_n$). 
\end{ex}

\begin{rem} We are being overly rigid in our construction of light leaves for the sake of algorithmic simplicity. In particular, consider the enumeration $\{v_1, \ldots, v_k\}$ specified in \Cref{lem:enumeratev}. In the type $D$ examples above, if one chooses $X_n = [I + \urial + \teal]$ instead of $[I + \teal + \urial]$, then one can draw very similar pictures except with an upward-facing crossing between $\teal$ and $\urial$ strands on top. This upward-facing crossing is an isomorphism in $\Frob$ between $[I + \urial + \teal]$ and $[I + \teal + \urial]$, and it is a rex move. For all practical purposes, we can add such upward crossings willy-nilly without affecting the essential properties of elementary light leaves. Indeed, later stages in the algorithm we pre- and post-compose elementary light leaves with arbitrarily chosen rex moves, which can have the effect of adding such crossings. We have chosen the rigid presentation of elementary light leaves above only for the sake of giving cleaner proofs. \end{rem}

\begin{rem}  There are other ways we are being unnecessarily rigid in our choices of reduced expression. In Example \ref{ex:typeAannoying}, one could choose a reduced expression for both $X_m$ and $X_n$ which starts by adding the black and gray strands. The resulting map $\ELL([m,n])$ would be \eqref{ELLA} tensored on the left with the identity map of the black and gray strands. This version of $\ELL([m,n])$ is related to the one in \eqref{ELLAannoying} by pre- and post-composition with rex moves which are isomorphisms.
\end{rem}

\begin{rem} The two choices of $\ELL([m_3, n])$ in \eqref{eq:ELLD3} are related by precomposition with rex moves which happen to be isomorphisms. More generally, when different rexes for $m$ differ by switchback relations which are not isomorphisms, then precomposition with a rex move will not transform one version of $\ELL([m,n])$ to another! The difference between two versions of $\ELL([m,n])$ (after pre- and post-composition with suitable rexes so they have the same source and target) is in the span of ``lower terms,'' in a sense which is difficult to explain until later in the paper. This comes out in the wash as part of Theorem \ref{thm:COB}. \end{rem}

\subsection{First stage continued: sprinkling polynomials}

\begin{defn} \label{defn:ELLP}
For $f \in R^{\leftred(m)}$ we define the \emph{elementary light leaf with polynomial
} $\ELLP([m,n],f)$ to be the morphism obtained from $\ELL([m,n])$ by adding $f$ to the leftmost $\leftred(m)$-region in the domain. 
\end{defn}

Polynomials are sprinkles flavoring the diagram. For example, one modifies \eqref{ELLA} to obtain
\begin{equation}
{
\labellist
\small\hair 2pt
 \pinlabel {$f$} [ ] at 37 15
\endlabellist
\centering
\igm{ELLA}
}.
\end{equation}

\subsection{First stage completed: elementary light leaves for \texorpdfstring{$[n,m]$}{[n,m]}}

Now we discuss how to construct a map $\ELL([n,m]) \co X_n \ot [Js,J] \to X_m$. By flipping $\ELL([m,n])$ upside-down and reversing orientation we already have a map $X_n \to X_m \ot [J,Js]$. The map we desire is related to this one by adjunction. Continuing Example \ref{ex:ELLA} we would have
\begin{equation}\ELL([n,m]) =
{
\labellist
\tiny\hair 2pt
 \pinlabel {$I$} [ ] at 7 32
 \pinlabel {$J\sberry$} [ ] at 266 14
 \pinlabel {$J$} [ ] at 303 32
 \pinlabel {$\leftred(m)$} [ ] at 36 52
 \pinlabel {$\rightred(m)$} [ ] at 226 52
\endlabellist
\centering
\igm{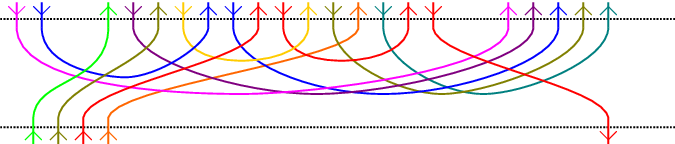}
}. \end{equation}

Let us state this construction more formally.


\begin{notation} Let $\DELL([m,n]):= \duality(\ELL([m,n]))$, where $\duality$ is the duality functor of Definition \ref{defn:duality}. \end{notation}

\begin{defn} Define $\ELL([n,m]) \co X_n \ot [J\sberry,J] \to X_m$ as \begin{equation}  \ELL([n,m]) = (\Id_m\otimes \capcw)\circ (\DELL([m,n])\otimes \Id_{[-\sberry]}). \label{ELLqpdef}\end{equation} \end{defn}


When we draw schematic diagrams, we use a trapezoid to represent an elementary light leaf, whether $\ELL([m,n])$ or $\ELLP([m,n],f)$. 
\begin{equation*} {
\labellist
\small\hair 2pt
 \pinlabel {$\ELL([m,n])$} [ ] at 38 31
\endlabellist
\centering
\igm{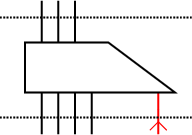}
} \qquad {
\labellist
\small\hair 2pt
 \pinlabel {$\ELL([n,m])$} [ ] at 38 31
\endlabellist
\centering
} \end{equation*}
Applying the duality functor one obtains an upside-down trapezoid.
\begin{equation*} {
\labellist
\tiny\hair 2pt
 \pinlabel {$\DELL([m,n])$} [ ] at 39 33
\endlabellist
\centering
\igm{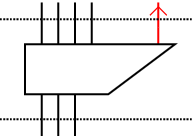}
} \qquad {
\labellist
\tiny\hair 2pt
 \pinlabel {$\DELL([n,m])$} [ ] at 39 33
\endlabellist
\centering
}\end{equation*}
We have defined
\begin{equation} {
\labellist
\small\hair 2pt
 \pinlabel {$\ELL([n,m])$} [ ] at 38 31
\endlabellist
\centering
\igm{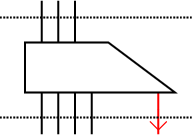}
} := {
\labellist
\tiny\hair 2pt
 \pinlabel {$\DELL([m,n])$} [ ] at 38 30
\endlabellist
\centering
\igm{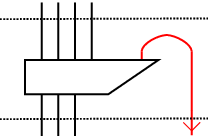}
}. \end{equation}

We will not need to place a polynomial within $\ELL([n,m])$. By convention, 
\begin{equation*}\ELLP([n,m],1) := \ELL([n,m]).\end{equation*}

\begin{lem}\label{degnm} We have $\deg(\ELL([n,m])) = \defect([n,m])$. \end{lem}

\begin{proof}
The duality functor $\mathcal{D}$ preserves degree, and the clockwise cap has degree $\ell(Js)-\ell(J)$ so 
\[\deg(\ELL([n,m]))=\deg(\ELL([m,n]))+\ell(Js)-\ell(J).\]
We conclude the proof by \Cref{lem:noclock} and \eqref{eq:defqpdefpq}.
\end{proof}

We will give more examples in \S\ref{ssec:llex}, but first we must discuss a very special case.

\begin{ex}\label{ex.ELLforward}
Let $m \subset n$ be the unique $(I,J,s)$ Grassmannian coset pair for which $\ma{m} = w_{Js}$. That is, $m$ is maximal in $n$, and the step $[n,m]$ is forward in the sense of \cite[Definition 2.21]{EKo}, and $\defect([n,m]) = 0$. Then $n$ has a reduced expression $[[I \subset Js]]$, while $m$ has a reduced expression $X'_m = [[I \subset Js \supset J]]$, which is also the source of $\ELL([n,m])$. It seems desirable to let $\ELL([n,m])$ be the identity map of $X'_m$ in this case, and sometimes it is. Indeed, we can allow this as an exception to our general algorithm, defining $\ELL([n,m])$ to be the identity in this case, and the corresponding light leaves will satisfy all the desired properties for our proofs in the next chapter to work verbatim, e.g. \Cref{ll1=1} holds.

However, our overly rigid construction above requires something different: it requires the target of $\ELL([n,m])$ to be a reduced expression $X_m$ for $m$ factoring through $m^{\core}$. We conjecture that $\ELL([n,m])$ is a rex move (see \Cref{rmk:unicity} for how one might prove this), though we will not need this statement. Regardless, Lemma \ref{rexELLisID} states that postcomposing $\ELL([n,m])$ with a rex move $X_m \to X_m'$ yields the identity map of $X_m'$ (at least after applying the evaluation functor). We give several subexamples below.

Later on we postcompose elementary light leaves with rex moves, so we lose nothing by assuming $\ELL([n,m])$ is the identity. We do not make this assumption (outside of \S\ref{ssec:fibered}). \end{ex}


\begin{ex} Let $I = J = \mt$, let $n = \{\id,s\}$ and $m = \{s\}$. Then we have $X_m = [\mt,s,\mt]$ and $X_n = [\mt,s]$ and
\begin{equation} \ELL([m,n]) = \igs{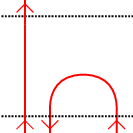}. \end{equation}
Consequently we have
\begin{equation} \ELL([n,m]) = \igs{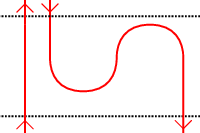} = \igs{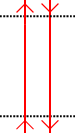}. \end{equation}
\end{ex}

\begin{ex} We continue Example \ref{ex:dihedralsinister}, assuming that $m_{s,t} = 6$, so that $m$ contains the longest element. Then
\begin{equation} \ELL([n,m]) = \igs{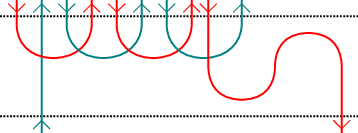} = \igs{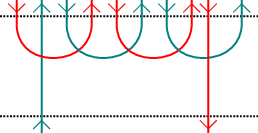}. \end{equation}
This is the elementary rex move associated to a switchback relation, and if we post-compose it with its upside-down flip, we get the identity map by Theorem \ref{thm:braidpiisone}.
\end{ex}

\begin{ex} In Example \ref{ex.ELLD}, the coset $m_3$ is maximal in $n$. As an exercise, the reader should take either version of $\ELL([m,n])$ from \eqref{eq:ELLD3}, use it to construct $\ELL([n,m])$, and then post-compose with rex moves until one obtains the identity map of $[I + \teal + \urial - \sberry]$. \end{ex}

\subsection{Second stage: singlestep light leaves}\label{ss.singlestepLL}

In this subsection we fix an $(I,J,s)$-coset pair $p \subset q$ which is not necessarily a Grassmannian coset pair. We construct light leaves for $[p,q]$ and $[q,p]$.

\begin{notation} \label{notation:SSELL} Theorem \ref{thmB} constructs an $(I,\rightred(q))$-coset $z$, and a $(\rightred(q),J,s)$-coset pair $m \subset n$ such that $z.m = p$ and $z.n = q$. We fix reduced expressions $X_m$ and $X_n$ as in Notation \ref{notation:ELL}. We fix any reduced expression $Z_{\bullet}$ for $z$, and set
\begin{equation} X_p := Z_{\bullet} \circ X_m, \qquad X_q := Z_{\bullet} \circ X_n. \end{equation}
\end{notation}

\begin{defn}\label{SSLL}
We define the \emph{single step light leaf} as
\begin{equation} \SSLL([p,q]) = \id_{Z_{\bullet}} \ot \ELL([m,n]), \qquad \SSLL([q,p]) = \id_{Z_{\bullet}} \ot \ELL([n,m]). \end{equation}
Moreover, we define the \emph{single step light leaf with polynomial} as
\begin{equation} \SSLLP([p,q],f) = \id_{Z_{\bullet}} \ot \ELLP([m,n],f). \end{equation}
Here, $f \in R^{\leftred(m)}$. By convention, $\SSLLP([q,p],1) = \SSLL([q,p])$.
\end{defn}

Schematically, we have
\begin{equation} {
\labellist
\small\hair 2pt
 \pinlabel {$\SSLL([p,q])$} [ ] at 51 31
 \pinlabel {$z$} [ ] at 8 51
 \pinlabel {$n$} [ ] at 48 51
 \pinlabel {$z$} [ ] at 7 14
 \pinlabel {$m$} [ ] at 47 14
\endlabellist
\centering
\igm{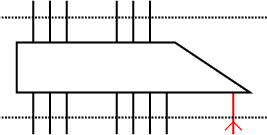}
} := {
\labellist
\small\hair 2pt
 \pinlabel {$z$} [ ] at 8 51
 \pinlabel {$n$} [ ] at 48 51
 \pinlabel {$z$} [ ] at 7 14
 \pinlabel {$m$} [ ] at 47 14
 \pinlabel {\tiny $\ELL([m,n])$} [ ] at 73 31
\endlabellist
\centering
\igm{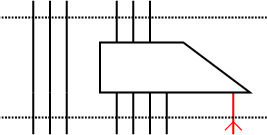}
}. \end{equation}

Recall from \eqref{eq:leftredn} that 
\begin{equation} \label{eq:leftrednredux} \leftred(m) = \mi{q}^{-1} \leftred(p) \mi{q}. \end{equation}

\begin{lem}\label{degSSLL}
    We have $\deg(\SSLL([p,q]))=\defect([p,q])$ and $\deg(\SSLL([q,p]))=\defect([q,p])$.
\end{lem}
\begin{proof}
    Clearly, we have 
\[ \deg(\SSLL([p,q]))=\deg(\ELL([m,n])), \quad \deg(\SSLL([q,p]))=\deg(\ELL([n,m])).\] In view of \Cref{lem:noclock} and \Cref{degnm}, the current lemma follows from \Cref{degmnvspq}.
\end{proof}

\begin{rem} It would be more consistent with Definition \ref{defn:ELLP} if we had defined $\SSLLP([p,q],f)$ so that $f \in R^{\leftred(p)}$, and $f$ is placed in the leftmost region of $X_p$ labeled by $\leftred(p)$. However, there need not be a region of $X_p$ labeled by $\leftred(p)$ at all! Instead, $R^{\leftred(p)}$ and $R^{\leftred(m)}$ are isomorphic via multiplication by $\mi{q}$. If desired, one can define an ``action" of $g \in R^{\leftred(p)}$ on $\SSLL([p,q])$ by placing $\mi{q}^{-1}(g)$ in the region labeled $\leftred(m)$; to whit
\begin{equation} ``\SSLLP([p,q],g)" := \SSLLP([p,q],\mi{q}^{-1}(g)).\end{equation}
Alternatively, $p$ has some other reduced expression which starts with $[[I \supset \leftred(p)]]$. One could place polynomials in the $\leftred(p)$ region in this expression, apply a rex move to $X_p$, and then apply $\SSLL([p,q])$. These two different ways for $R^{\leftred(p)}$ to act will ``agree modulo lower terms,'' a fact we do not discuss or need in this paper, but hope to address in future work. \end{rem}

\begin{rem} It is important to note that the reduced expression $X_p$ above (resp. $X_q$) does not depend only on $p$, but also on $q$ and some other choices. If a given double coset $p$ appears in multiple contexts, the meaning of $X_p$ may vary. \end{rem}



\subsection{Third stage: light leaves with polynomials}\label{ss.LL}

Let $I_{\bullet} = [I_0, \ldots, I_d]$ be any singlestep expression. We now inductively construct a light leaf with polynomials $\LLP(t_{\bullet},f_\bullet)$ for a subordinate path $t_\bullet\subset I_\bullet$ with an associated sequence of polynomials $f_{\bullet}$.

\begin{defn} Let $t_\bullet \subset I_\bullet$ be a subordinate path. We say that a sequence $f_{\bullet}=[f_1,\ldots, f_d]$ of polynomials in $R$ is \emph{adapted to $t_{\bullet}$} if for all $0 \le k \le d-1$ we have
\begin{equation}\label{fadapted} \begin{cases} f_{k+1} = 1 & \text{ whenever } t_k \supset t_{k+1}, \\ f_{k+1} \in R^{\leftt} & \text{ whenever } t_k \subset t_{k+1}, \text{ where } \leftt = \mi{t_{k+1}^{-1}}\leftred(t_k)\mi{t_{k+1}}. \end{cases} \end{equation} \end{defn}

Note the compatibility between $\mi{t_{k+1}^{-1}}\leftred(t_k)\mi{t_{k+1}}$ and \eqref{eq:leftrednredux}. 

\begin{defn} Let $t_{\bullet} \subset I_{\bullet}$, and $f_{\bullet}$ be adapted to $t_{\bullet}$. Let $\LLP(t_{\le 0}, f_{\le 0})$ denote the identity map of the identity 1-morphism of $I_0$. 
For each $0 \le k \le d$ let $I_{\le k} = [I_0, \ldots, I_k]$ and  $t_{\le k} = [t_0, \ldots, t_k]$. Suppose that we have already constructed the light leaf with polynomials map
\begin{equation} \LLP(t_{\le k},f_{\le k}) \co I_{\le k} \to X'_{t_k}. \end{equation}
Recall that the singlestep light leaf is a map
\begin{equation} \SSLLP([t_k,t_{k+1}],f_{k+1}) \co X_{t_k} \ot [I_k, I_{k+1}] \to X_{t_{k+1}}. \end{equation}
Above, $X_{t_k}$ and $X'_{t_k}$ are certain reduced expressions for $t_k$, but not necessarily the same reduced expression.
We define $\LLP(t_{\le k+1},f_{\le k+1})$ as the composition
\begin{center}
\begin{tikzpicture}
    \node (a) at (0,0) {$I_{\le k+1} = I_{\le k} \ot [I_k, I_{k+1}]$};
    \node (b) at (0,-1.5) {$X'_{t_k}\ot [I_k, I_{k+1}]$};
    \node (c) at (0,-3) {$ X_{t_k}\ot [I_k, I_{k+1}]$};
    \node (d) at (0,-4.5) {$X_{t_{k+1}}$};
    \draw[->] (a) to node[left] {$\LLP(t_{\le k},f_{\le k}) \ot \id$} (b);
    \draw[->] (b) to node[left] {$\rex \ot \id$} (c);
    \draw[->] (c) to node[left] {$\SSLLP([t_k,t_{k+1}],f_{k+1})$} (d);
    \draw[->] (a.east) to [out=-60,in =40] node[right] {$\LLP(t_{\leq k+1},f_{\leq k+1})$} (d.east);
\end{tikzpicture}
\end{center}
where $\rex$ denotes some chosen rex move, see Definition \ref{defn:rexmove}. Finally, we set $\LLP(t_{\bullet}, f_{\bullet}) = \LLP(t_{\le d},f_{\le d})$.
\end{defn}

In the schematic below, $\SSLLP_{k+1}$ represents $\SSLLP([t_k,t_{k+1}],f_{k+1})$.
\begin{equation} {
\labellist
\tiny\hair 2pt
 \pinlabel {\small $\LLP(t_{\le k+1}, f_{\le k+1})$} [ ] at 41 31
 \pinlabel {\small $I_0$} [ ] at 1 32
 \pinlabel {$I_{\le k}$} [ ] at 7 14
 \pinlabel {\small $I_k$} [ ] at 93 15
 \pinlabel {\small $I_{k+1}$} [ ] at 114 41
 \pinlabel {$X_{t_{k+1}}$} [ ] at 6 51
\endlabellist
\centering
\igb{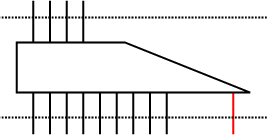}
} := {
\labellist
\tiny\hair 2pt
 \pinlabel {\small $\LLP(t_{\le k},f_{\le k})$} [ ] at 38 32
 \pinlabel {\small $\rex$} [ ] at 31 62
 \pinlabel {\small $\SSLLP_{k+1}$} [ ] at 34 93
 \pinlabel {\small $I_0$} [ ] at 1 35
 \pinlabel {$I_{\le k}$} [ ] at 7 14
 \pinlabel {\small $I_k$} [ ] at 71 45
 \pinlabel {\small $I_{k+1}$} [ ] at 101 71
 \pinlabel {$X'_{t_k}$} [ ] at 7 52
 \pinlabel {$X_{t_k}$} [ ] at 7 76
 \pinlabel {$X_{t_{k+1}}$} [ ] at 6 111
\endlabellist
\centering
\igb{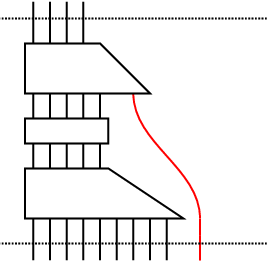}
} \end{equation}

\subsection{Fourth stage: fixing the polynomials}

We now choose our polynomials wisely.

\begin{notation} \label{notation:PKL} Let $I_{\bullet} = [I_0, \ldots, I_d]$ be a singlestep expression. For $0 \le j < d$ we call $j$ an \emph{ascending index} if $I_j \subset I_{j+1}$. When a subordinate path $t_{\bullet}$ is understood, for each ascending index $j$ we set
\begin{equation} \leftt_j = \mi{t_{j+1}}^{-1}\leftred(t_j)\mi{t_{j+1}}, \qquad \rightt_{j+1} = \rightred(t_{j+1}). \end{equation}
\end{notation}

\begin{lem} We have $\leftt_j \subset \rightt_{j+1}$. \end{lem}

\begin{proof} Let $m \subset n$ be the Grassmannian coset pair associated to $t_j \subset t_{j+1}$. As we have noted earlier, $\leftt_j = \leftred(m)$ and $\rightt_{j+1} = \leftred(n)$. We have $\leftred(m) \subset \leftred(n)$ by \cite[Lemma 3.1]{EKo}. \end{proof}

\begin{notation} \label{notation:choosePL} For each finitary subset $L$ we choose a polynomial $P_L$ such that $\pa_L(P_L) = 1$. The existence of $P_L$ is guaranteed by our assumption of generalized Demazure surjectivity, see \cite[S. 3.2]{EKLP1}. \end{notation}

We will be interested in $P_{\rightt_{j+1}}$ for each ascending index $j$ (see Remark \ref{rmk:varyP}).

\begin{defn}
Let $I_\bullet$ be  a singlestep expression.
A \emph{sprinkled subordinate path} is a pair $(t_\bullet,y_\bullet) \subset I_\bullet$,	where $t_\bullet\subset I_\bullet$ is a subordinate path, and $y_\bullet=(y_j)$ is a sequence of elements of $W$, defined only for ascending indices $j$, such that
 $y_j\in W_{\rightt_{j+1}}$ is minimal in its right coset $W_{\leftt_j} y_j$.
\end{defn}

\begin{defn}\label{llty} 
 For a sprinkled path  $(t_\bullet,y_\bullet)$, we define
$\LL(t_\bullet,y_\bullet):=\LLP(t_\bullet, f_\bullet)$, where $f_{\bullet}$ is the adapted sequence of polynomials such that
\begin{equation} f_j:= \pa_{\leftt_j}\pa_{y_j}(P_{\rightt_{j+1}}) \qquad \text{whenever } I_j \subset I_{j+1}.  \end{equation}
\end{defn}

The reason to choose these particular polynomials is the following lemma.

\begin{lem} \label{lem:matchingpolytosprinkles}
Let $t_\bullet \subset I_\bullet$ be a subordinate path and assume that $j$ is an ascending index. The polynomials $f_j=\pa_{\leftt_j}\pa_{y_j}(P_{\rightt_{j+1}})$, ranging over  $y_j\in W_{\rightt_{j+1}}$ minimal in its right coset $W_{\leftt_j} y_j$, form a basis for a vector space of graded dimension $\poly([t_{j},t_{j+1}])$.

Consequently, if one fixes $t_{\bullet}$ and sums over all sprinkled subordinate paths of the form $(t_{\bullet}, y_{\bullet})$, one has
\begin{equation}\label{degpolynomialadapted} \sum_{(t_{\bullet},y_{\bullet})} v^{\deg(f_\bullet)}=\poly(t_\bullet).\end{equation}
\end{lem}

\begin{proof} By \Cref{thm:dualbasisinimage}, the polynomials $\{f_j\}$ form a basis for $R^{\leftt_{j}}$ over $R^{\rightt_{j+1}}$. It remains to verify that $\poly([t_j,t_{j+1}])$ agrees with the graded rank of $R^{\leftt_{j}}$ over $R^{\rightt_{j+1}}$. This follows directly from \Cref{lem:grankIJ}. \end{proof}


We apologize that in our naming conventions, sprinkles seem to refer both to the polynomials $f_j$ sprinkled in a diagram (our preferred interpretation), and the elements $y_j$ of $W$ which index and determine them.

\begin{rem} \label{rmk:whenI0empty} When $I_0 = \emptyset$, we have $\rightt_{j+1} = \leftt_j = \emptyset$ for all $j$ with $I_j \subset I_{j+1}$. The only valid choice of $y_j$ is the identity element of $W$, and $f_j$ is the identity element of $R$. Sprinkles only appear in the truly singular context. \end{rem}

\subsection{Examples} \label{ssec:llex} 
Here we illustrate the first four stages via examples. The first example is special.

\begin{ex}\label{ex.LLrex}
Let $I_\bullet\expr p$ be a reduced expression. Then there is a unique subordinate path $t_\bullet\subset I_\bullet$ with $\term(t_\bullet)=p$, namely the forward path of $I_\bullet$ (see \cite[Lemma 2.22]{EKLP2}).
For the forward path $t_\bullet$, each $\ELL([m,n])$ associated to $t_j\subset t_{j+1}$ is the identity map (see Example~\ref{ex.ELLred}) and each $\ELL([n,m])$ associated to $t_j\supset t_{j+1}$ can be taken to be the identity map (see \Cref{ex.ELLforward} and \Cref{rexELLisID}).
Moreover, for each $I_j\subset I_{j+1}$, since $\leftred(t_j)=\leftred(t_{j+1})$ (see \cite[Definition 2.24]{EKo}) we have $\leftt_j = \rightt_{j+1}$. 
Therefore, for a sprinkled path $(t_\bullet, y_\bullet)\subset I_\bullet$ with terminus  $p$, each element $y_j$ is the identity element in $W$, each polynomial $f_j$ is $1$, and $\LL(t_\bullet,y_\bullet)$ is a rex move without polynomials.
\end{ex}

We give a few examples of light leaves $\LL(t_{\bullet}, y_{\bullet})$ which live in $\Hom(\BS(I_{\bullet}),Y_p)$. Here, $t_{\bullet}$ is a path subordinate to $I_{\bullet}$ with terminus $p$, and $Y_p$ is some reduced expression for $p$. Elementary light leaves are not expected to form a basis for any particular morphism space; this is the role of double leaves, to be defined in the next section. However, when $p$ contains the identity so that it is minimal in the Bruhat order, light leaves should\footnote{This is because of \Cref{prop:llsisbasismodlower}, which states that light leaves form a basis for morphisms modulo lower terms, together with the fact that there are no lower terms in this case.} form a basis for $\Hom(\BS(I_{\bullet}),Y_p)$ as a free module over $R_p$, which acts by postcomposition, placing polynomials within a region of $Y_p$. We try to illustrate this principle, and show why it requires the existence of polynomial sprinkles.

\begin{ex}
In the next series of examples, we will use only one simple reflection $\sberry$. First consider  when $I_{\bullet} = [\sberry,\mt]$, a reduced expression for the $(\sberry,\mt)$-coset $p = \{\id,\sberry\}$. It has a unique subordinate path, the forward path, so the only light leaf is the identity map.
\[ {
\labellist
\small\hair 2pt
 \pinlabel {$\sberry$} [ ] at 0 35
\endlabellist
\centering
\igs{downs}
} \]
We have $\BS(I_{\bullet}) = R$ as an $(R^{\sberry}, R)$-bimodule, whose endomorphism ring is $R = R_p$. So indeed, this morphism space is spanned by the unique light leaf, up to placing polynomials in the region labeled $\mt$.
\end{ex}

\begin{ex} \label{mtsmt} Now consider $I_{\bullet} = [\mt,\sberry,\mt]$ and $p = \{\id\}$. There is a unique subordinate path $t_{\bullet}$ with terminus $p$, and no sprinkles can appear, see Remark \ref{rmk:whenI0empty}. The light leaf associated to $t_{\bullet}$ is
\begin{equation} {
\labellist
\small\hair 2pt
 \pinlabel {$\sberry$} [ ] at 31 15
\endlabellist
\centering
\igs{cwcap}
},\end{equation}
related to the previous case by adjunction. The morphism space is related to the previous one by adjunction, so it is also free of rank $1$ over $R = R_p$, which acts by placing polynomials on top. \end{ex}

\begin{ex} \label{smts} We consider one more morphism space related to the previous by adjunction, where the situation is fundamentally different. Let $I_{\bullet} = [\sberry,\mt,\sberry]$. It is a non-reduced expression, with a unique subordinate path $t_{\bullet}$ having terminus $p = \{\id,\sberry\}$, which has a length $0$ reduced expression $[\sberry]$. In this case, $t_1 \subset t_2$ and $\leftt_1 = \mt$ while $\rightt_2 = \{\sberry\}$. Thus there are two choices for sprinkles, $y_1 = \id$ and $f_1 = P_{\sberry}$, or  $y_1 = \sberry$ and $f_1 = \pa_{\sberry}(P_{\sberry}) = 1$. So there are two light leaves:
\begin{equation} {
\labellist
\tiny\hair 2pt
 \pinlabel {\small $\sberry$} [ ] at 50 40
\endlabellist
\centering
\igs{ccwcap}
}, \qquad {
\labellist
\tiny\hair 2pt
 \pinlabel {$P_{\sberry}$} [ ] at 31 18
 \pinlabel {\small $\sberry$} [ ] at 50 40
\endlabellist
\centering
\igs{ccwcap}
}.\end{equation}
Again, the morphism space in question is free of rank $1$ over $R$, but in this case $R_p = R^{\sberry}$, so it is free of rank $2$ over $R_p$ (which acts by placing polynomials on top of the diagram). Because the polynomial ring we can postcompose with is too small, we are forced to keep extra polynomials sprinkled in our light leaves as we go. This is the most basic illustration of the concept: every time one loses ``polynomial degrees of freedom'' when following a subordinate path, we make up for it with polynomial sprinkles.    
\end{ex}

We now give examples where $S = \{\sberry, \teal\}$ and $m_{\sberry \teal} = 3$, so that $W \cong S_3$. Our examples will be for $(I, J)$-expressions with $I = \{\sberry\}$ and $J = \{\teal\}$.

\begin{ex} \label{ex:A2part1} We start with an easy extension of the earlier examples. Let $I_{\bullet} = [I - \sberry + \sberry + \teal - \sberry]$. There is a unique path $t_{\le 3}$ subordinate to $[I-\sberry+\sberry+\teal]$, with terminus $t_3$ equal to the $(I,S)$-coset containing all of $W$. So there are two paths subordinate to $I_{\bullet}$ corresponding to the two $(I,J)$-cosets inside $t_3$: one with $\ma{t_4} = \sberry\teal\sberry$ and one with $\ma{t_4} = \sberry \teal$. The only step for which the redundancy increases is the second step $+\sberry$, so only $y_2$ has the potential to be nontrivial. In all, this leads to four light leaves:
\begin{equation} {
\labellist
\tiny\hair 2pt
 \pinlabel {\small $\sberry$} [ ] at 50 40
\endlabellist
\centering
\igs{ccwcap}
} \igs{upt} \igs{downs}, \quad {
\labellist
\tiny\hair 2pt
 \pinlabel {$P_{\sberry}$} [ ] at 31 18
 \pinlabel {\small $\sberry$} [ ] at 50 40
\endlabellist
\centering
\igs{ccwcap}
} \igs{upt} \igs{downs}, \quad {
\labellist
\tiny\hair 2pt
 \pinlabel {\small $\sberry$} [ ] at 50 40
\endlabellist
\centering
\igs{ccwcap}
} \igs{rightcross}, \quad {
\labellist
\tiny\hair 2pt
 \pinlabel {$P_{\sberry}$} [ ] at 31 18
 \pinlabel {\small $\sberry$} [ ] at 50 40
\endlabellist
\centering
\igs{ccwcap}
} \igs{rightcross}.
\end{equation}
\end{ex}

\begin{ex} \label{ex:A2part2}   Now let $I_{\bullet} = [I-\sberry+\teal-\teal+\sberry-\sberry+\teal]$. Note that $[I-\sberry+\teal-\teal+\sberry-\sberry]$ is a reduced expression for the $(I,\mt)$-coset containing the longest element $\sberry\teal\sberry$.

We will record subordinate paths by listing the maximal elements of each coset. Every subordinate path $t_{\bullet}$ has $\ma{t_0} = \sberry$, $\ma{t_1} = \sberry$, and $\ma{t_2} = \sberry \teal$. These first three indices are as in the forward path of a reduced expression, and the corresponding light leaf is the identity map. The ascending indices are $1$, $3$, and $5$, and $\leftt_j = \mt$ for each one. Since $\rightt_2 = \mt$ we have $y_1 = \id$. Here are the options for how a sprinkled subordinate path might continue:
\begin{enumerate}
\item The forward path continues as
\[ \ma{t_3} = \sberry \teal, \quad \ma{t_4} = \sberry t\sberry, \quad \ma{t_5} = \sberry \teal\sberry, \quad \ma{t_6} = \sberry \teal\sberry\]
We have $y_3 = \id$ since $\rightt_4 = \mt$. However, $\rightt_6 = \{t\}$, so $y_5 \in \{\id, \teal\}$. The Grassmannian pair associated to $[t_5, t_6]$ is $[p,q]$ for the $(\teal,\mt)$-coset $p = \{\id,\teal\}$ and the $(\teal,\teal)$-coset $q = \{\id,\teal\}$. The corresponding elementary light leaves (with polynomials) are just like those in Example \ref{smts}. The singlestep light leaf is the same but tensored with an identity map of $[I+\teal-\sberry]$: it is a morphism from $[I + \teal - \sberry - \teal + \teal]$ to $[I + \teal - \sberry]$.

Since $t_{\le 5}$ is the forward path of a reduced expression, the corresponding light leaf may be taken to be the identity map. However, we must apply a rex move to reach the source of the final singlestep light leaf.

Two conclude, the two sprinkled subordinated paths with the forward subordinate path have the following light leaves: 
\begin{equation} {
\labellist
\small\hair 2pt
 \pinlabel {$\sberry$} [ ] at 10 39
\endlabellist
\centering
\igs{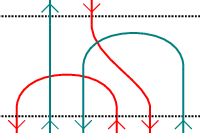}
}, \qquad {
\labellist
\small\hair 2pt
 \pinlabel {$\sberry$} [ ] at 10 39
 \pinlabel {\tiny $P_{\teal}$} [ ] at 70 36
\endlabellist
\centering
\igs{LLA2example}
}. \end{equation}

\item Another path continues as
\[ \ma{t_3} = \sberry \teal, \quad  \ma{t_4} = \sberry \teal\sberry, \quad \ma{t_5} = \sberry \teal, \quad \ma{t_6} = \sberry \teal \]
Again $y_3 = \id$, but now $\rightt_6 = \mt$ so $y_5 = \id$. Only the first four steps are reduced. The Grassmannian pair associated to $[t_4,t_5]$ is $[q,p]$ for the $(\mt,\mt)$-coset $p = \{\sberry\}$ and the $(\mt,\sberry)$-coset $q = \{\id,\sberry\}$, as in Example \ref{mtsmt}. The Grassmannian pair associated to $[t_5, t_6]$ is $[p',q']$ for the $(\mt, \mt)$-coset $p' = \{\teal\}$ and the $(\mt,\teal)$-coset $q' = \{\id,\teal\}$, as in Example \ref{smts} but without the need for any polynomials (they could be forced to another region)! 

The corresponding light leaf is 
\begin{equation} \igs{downs} \igs{upt} \igs{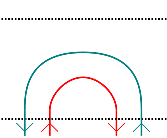}.\end{equation}

\item The final path continues as
\[ \ma{t_3} = \sberry, \quad \ma{t_4} = \sberry, \quad \ma{t_5} = \sberry, \quad \ma{t_6} = \sberry \teal  \]
This time $\rightt_4 = \{\sberry\}$, so $y_3 \in \{\id,\sberry\}$. Meanwhile, $y_5 = \id$ as above. The corresponding light leaves are

\begin{equation} {
\labellist
\small\hair 2pt
 \pinlabel {$\sberry$} [ ] at 9 44
\endlabellist
\centering
\igs{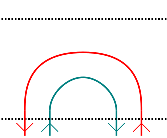}
} \igs{downs} \igs{upt}, \qquad {
\labellist
\small\hair 2pt
 \pinlabel {\tiny $P_{\sberry}$} [ ] at 56 29
 \pinlabel {$\sberry$} [ ] at 9 44
\endlabellist
\centering
\igs{doublecapexample}
} \igs{downs} \igs{upt}. \end{equation}

\end{enumerate}
\end{ex}

\subsection{Fifth stage: double leaves}

\begin{defn} \label{def:mainthmsetup} Let $I_\bullet$ and $I'_\bullet$ be  singlestep $(I,J)$-expressions. A triple of the form $(p,(t_\bullet,y_\bullet),(t'_\bullet,y'_\bullet))$ is called a \emph{coterminal sprinkled triple} for $(I_\bullet,I'_\bullet)$ if
$(t_\bullet,y_\bullet)$ (resp., $(t
_\bullet,y'_\bullet)$) is a sprinkled subordinate path of $I_\bullet$ (resp., $I'_\bullet$), and $\term(t_{\bullet}) = \term(t'_{\bullet}) = p$.

A quadruple $(p,(t_\bullet,y_\bullet),(t'_\bullet,y'_\bullet),g)$, where in addition $g$ is a polynomial in $R_p = R^{\leftred(p)}$, is called a  \emph{coterminal sprinkled quadruple} for $(I_\bullet,I'_\bullet)$. \end{defn}

\begin{rem} Note that $p$ can appear within a coterminal sprinkled triple for $(I_{\bullet}, I'_{\bullet})$ if and only if $p \le I_{\bullet}$ and $p \le I'_{\bullet}$, see Definition \ref{def.bruhat}. There are finitely many such $p$ for each pair $(I_{\bullet}, I'_{\bullet})$. \end{rem}

\begin{defn} Let $Y_p$ be a reduced expression for $p$ of the form $$[[I\supset \leftred(p)]] \; . \; p^{\mathrm{core}} \; . \;[[\rightred(p)\subset J]].$$
We define the \emph{singular double leaf} $\DLL(p,(t_\bullet,y_\bullet),(t'_\bullet,y'_\bullet), g)$ associated to the coterminal sprinkled quadruple $(p,(t_\bullet,y_\bullet),(t'_\bullet,y'_\bullet), g)$ as the composition

\begin{equation} I_\bullet\xrightarrow{\LL(t_\bullet,y_\bullet)}X_{p}\xrightarrow{\rex}Y_{p}\xrightarrow{ g}Y_{p}\xrightarrow{\rex}X'_p\xrightarrow{\duality(\LL(t'_\bullet,y'_\bullet))}I'_\bullet. \end{equation}
The reduced expressions $X_p$ and $X'_p$ are the targets of $\LLP(t_\bullet,y_\bullet)$ and $\LLP(t'_\bullet,y'_\bullet)$ respectively. The maps $\rex$ are some chosen rex moves between these different reduced expressions for $p$, see Definition \ref{defn:rexmove}. The map $g$ in the middle places the polynomial $g$  in the leftmost $\leftred(p)-$region in $Y_p$. \end{defn}

\begin{rem} We would like to emphasize that $Y_p$ is a special choice of reduced expression for $p$, designed so that the object $Y_p$ (which we are using as a proxy for the bimodule $\BS(Y_p)$) admits an obvious action of $R_p = R^{\leftred(p)}$. For any reduced expression of $p$, say $X_p$, we have $\End_{\not <p}(\BS(X_p)) \cong R_p$, see \Cref{gdimlowerterms}, but it is not always so easy to find $R_p$ as a subring of $\End(\BS(X_p))$ itself. \end{rem}

We typically denote a coterminal sprinkled triple with a capital letter like $T$, and write $\DLL(T,g)$, where $g\in R_p$, for the  double leaf associated to the coterminal sprinkled quadruple $(T,g)$.

\begin{ex} We continue Examples \ref{ex:A2part1} and \ref{ex:A2part2}, and consider double leaves from $I_{\bullet} = [I-\sberry + \teal - \teal + \sberry - \sberry + \teal]$ to $I'_{\bullet} = [I - \sberry + \sberry + \teal - \sberry]$. There are a total of ten double leaves: $2 \times 2$ which factor through $p_1$ with $\ma{p_1} = \sberry \teal \sberry$, and $3 \times 2$ which factor through $p_2$ with $\ma{p_2} = \sberry \teal$. The cosets $p_1$ and $p_2$ have unique reduced expressions, so no rex moves are necessary to link the light leaves into double leaves.

Here are three prototypical double leaves.
\begin{equation} {
\labellist
\tiny\hair 2pt
 \pinlabel {$P_{\sberry}$} [ ] at 27 76
 \pinlabel {$P_{\teal}$} [ ] at 72 33
 \pinlabel {$g$} [ ] at 11 50
\endlabellist
\centering
\igs{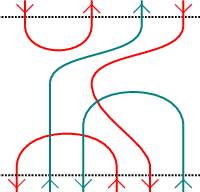}
} \qquad {
\labellist
\tiny\hair 2pt
 \pinlabel {$g$} [ ] at 77 60
 \pinlabel {$P_{\sberry}$} [ ] at 27 77
\endlabellist
\centering
\igs{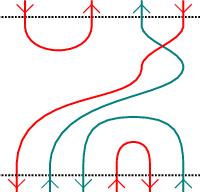}
} \qquad {
\labellist
\tiny\hair 2pt
 \pinlabel {$P_{\sberry}$} [ ] at 27 77
 \pinlabel {$P_{\sberry}$} [ ] at 45 24
 \pinlabel {$g$} [ ] at 78 52
\endlabellist
\centering
\igs{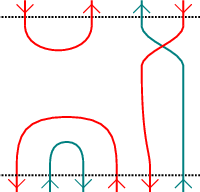}
}. \end{equation}
The first double leaf factors through $p_1$, and $g \in R^s$. The second and third factor through $p_2$, and $g \in R$. One obtains all other double leaves by replacing some of the polynomials $P_{\sberry}$ or $P_{\teal}$ with $1$.
\end{ex}

\subsection{Loose conclusion, and a silly choice of basis}\label{ss.Bp}

With notation as in Definition \ref{def:mainthmsetup}, fix a coterminal sprinkled triple $T = (p,(t_\bullet,y_\bullet),(t'_\bullet,y'_\bullet))$. As $g$ ranges among all polynomials in $R_p$, we think of the set $\{\DLL(T,g)\}$ of double leaf morphisms 
as forming a free $R_p$-module of rank $1$, generated by $\DLL(T,1)$. This $R_p$-module lives inside $\Hom(I_{\bullet}, I'_{\bullet})$, and we refer to it as $R_p \cdot \DLL(T,1)$.

Indeed, fixing $p$ but letting the other data vary, Theorem \ref{mainthm} below will imply that
\[ \{ \DLL(p,(t_\bullet,y_\bullet),(t'_\bullet,y'_\bullet), 1)\} \]
is a basis for a free $R_p$-module inside $\Hom(I_{\bullet}, I'_{\bullet})$, even after applying the evaluation functor to bimodules.
Moreover, Theorem \ref{mainthm} effectively states that the space $\Hom(\BS(I_{\bullet}), \BS(I'_{\bullet}))$ is spanned (freely) by these free $R_p$-modules as $p$ varies. We find this to be a satisfying way to think of our result.

Phrasing this theorem is somewhat awkward, as the ring $R_p$ varies when $p$ varies. To describe a basis for $\Hom(\BS(I_{\bullet}), \BS(I'_{\bullet}))$, we need to choose which ring we are defining a basis over. One could choose $R^I$ or $R^J$, but the most unifying choice is $\Bbbk$. For this purpose alone, we make the fairly pointless choice of a basis of each $R_p$ as a vector space.

\begin{notation} \label{notation:fixBB} For each double coset $p$ we fix a basis $\BB_p$ of $R_p$ 
over $\Bbbk$ consisting of homogeneous elements. \end{notation}

For fixed $T$, the elements $\{\DLL(T,b)\}$ for $b \in \BB_p$ form a $\Bbbk$-basis for $R_p \cdot \DLL(T,1)$. Changing the basis $\BB_p$ will change the basis $\{\DLL(T,b)\}$ of $R_p \cdot \DLL(T,1)$ in exactly the same way, simultaneously, for each coterminal sprinkled triple $T$ associated to $p$.

\begin{rem} \label{rmk:basisasRImod} Recall that $R_p = R^{\leftred(p)}$ is an $(R^I, R^J)$-bimodule, and is free as a left $R^I$-module or a right $R^J$-module. To obtain a basis for $R_p \cdot \DLL(T,1)$ as a left $R^I$-module, we could choose a basis for $R_p$ as a left $R^I$-module, and consider the double leaves $\{\DLL(T,c)\}$ as $c$ ranges among this basis. It is easy to observe that the left $R^I$-action on $\End(Y_p)$ matches the left $R^I$-action on $R_p$, using \eqref{polyslide}.

Similarly, we can obtain bases of $R_p \cdot \DLL(T,1)$ as a right $R^J$-module. However, it is false that the right $R^J$-action on $\End(Y_p)$ matches the right $R^J$-action on $R_p$. This is only true modulo $\End_{< p}(Y_p)$. \end{rem}

\subsection{Conclusion} \label{ss.conclusion}

\begin{thm}\label{mainthm}
After applying the functor $\evaluation$, the set of double leaves \[\{\DLL(T,b)\}\]  yields a $\Bbbk$-basis of $\mathrm{Hom}(BS(I_\bullet), BS(I'_\bullet))$, as $T = (p,(t_\bullet,y_\bullet),(t'_\bullet,y'_\bullet))$ ranges over the set of coterminal sprinkled triples for $(I_\bullet,I'_\bullet)$, and $b$ ranges over $\BB_p$. \end{thm}

The proof comprises the bulk of \S\ref{proofs}, concluding in \S\ref{ssec:doubleleafproof}.

\begin{rem} One can easily adapt \Cref{mainthm} to obtain a basis for this morphism space as a left $R^I$-module, or a right $R^J$-module, see Remark \ref{rmk:basisasRImod}. \end{rem}

The set $\{\DLL(T,b)\}$ depends on a large number of choices: the data chosen in Notation \ref{notation:ELL} for each elementary light leaf, the choice of reduced expression for $z$ in Notation \ref{notation:SSELL}, the choice of rex moves, the choice of basis $\BB_p$, the choice of polynomials $P_L$. All these choices have been omitted from our notation. Different choices will yield different bases, but any choice does indeed yield a basis.

We now ask about the relationship between these different bases. We have already discussed the boring choice of $\BB_p$ in \S~\ref{ss.Bp}. 
Our second theorem states that the other choices yield bases which are related by a unitriangular change of basis matrix. We are vague here, because the partial order which governs this unitriangularity requires a fair amount of work to describe.

\begin{thm} \label{thm:COB} Fix sequences $I_{\bullet}$ and $I'_{\bullet}$ as in \Cref{def:mainthmsetup}. Let $\prec_3$ denote the partial order on coterminal sprinkled triples from Definition \ref{def:orderontriples}. Let $\{\DLL(T,g)\}$ and $\{\DLL'(T,g)\}$ denote two families of double leaf morphisms, as $(T,g)$ ranges among coterminal sprinkled quadruples, which are produced using potentially different choices of reduced expressions, rex moves, and polynomials $P_L$. Then
\begin{equation} \DLL'(T,g) = \DLL(T,g) + \sum_{U \prec_3 T} \DLL(U,h_U) \end{equation}
for some polynomials $h_U$ such that $(U,h_U)$ is a coterminal sprinkled quadruple. \end{thm}

Our third theorem packages this basis into a convenient framework.

\begin{thm} \label{thm:fibered} Let $I, J$ be finitary. A double leaves basis equips $\SBSBim(J,I)$ with the structure of a fibred cellular category in the sense of \cite[Definition 2.17]{ELauda}. \end{thm}

In summary, this says that our basis factors nicely through a family of objects $Y_p$, that the span of basis elements factoring through $q < p$ forms an ideal, and that $\End(Y_p) / \End_{< p}(Y_p)$ is spanned by the identity map over $R_p$. We prove the result in \S\ref{ssec:fibered}.

\section{Proofs}\label{proofs}

\subsection{Sinister diagrams} \label{sec:sinister}

We give a name to the kind of diagrams desired by Definition \ref{defn:ELL}.

\begin{defn} A diagram in $\Frob$ is called \emph{sinister} if:
\begin{enumerate}
\item The 1-manifold of each color looks like one of the options in \eqref{twooptions}.
\item Only left-facing crossings and counterclockwise caps appear.
\item No two strands cross more than once.
\item No polynomials appear.
\end{enumerate}
If the target is an identity 1-morphism, so that each color looks like the left diagram in \eqref{twooptions}, then the sinister diagram is a \emph{sinister cap diagram}.
\end{defn}

 Our running example of a sinister diagram is the one we used in \eqref{ELLA}:
\begin{equation} \ELL = \igm{ELLA}. \end{equation}
We typically use the name $\ELL$ for a sinister diagram, since the main examples are elementary light leaves.

For any sinister diagram $\ELL$, we can twist all strands in the target down to the left to obtain a sinister cap diagram $\LEL$. In the example above, $\LEL$ equals
\begin{equation} \label{LEL} \igm{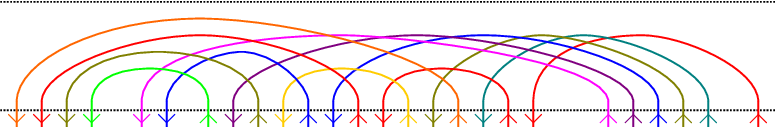}. \end{equation}
Clearly $\ELL$ and $\LEL$ are related by adjunction, and determine each other. We focus on the construction of sinister cap diagrams.

In any sinister cap diagram, the bottom boundary $L_{\bullet}$ evidently satisfies:
\begin{equation} \label{sinisterbottom} \textrm{Each $s \in S$ appears an even number of times, alternating $-s$ then $+s$}. \end{equation}
 
\begin{lem} \label{sinisterconstruction} Let $L_{\bullet}$ be any singlestep expression satisfying \eqref{sinisterbottom}. The algorithm in the proof constructs a well-defined sinister cap diagram with boundary $L_{\bullet}$. \end{lem}

\begin{proof} Let $\LEL$ be a hypothetical sinister cap diagram with bottom boundary $L_{\bullet}$. In this proof, a \emph{strand} refers to a connected component of the $1$-manifold obtained from $\LEL$ by considering just one color. Each strand meets $L_{\bullet}$ twice, and we call these two boundary points the \emph{start} and \emph{end} based on the orientation. By \eqref{twooptions}, the start is to the right of the end.

There are two total orders on the set of strands: \emph{start-order}, which orders the strands from right to left based on their start, and \emph{end-order}, which orders the strands from right to left based on their end. For example, the end-order in \eqref{LEL} is ``red1, teal1, olive1, red2, etcetera." These orders are determined by $L_{\bullet}$.

It is also determined by $L_{\bullet}$ whether two strands will cross one or zero times in any sinister cap diagram. Two strands will cross if and only if they are in the same relative position in start-order as in end-order. For example, a strand colored $\sberry$ will cross a strand colored $\teal$ if and only if they alternate as in
\begin{equation} \label{nicecrossing} [\ldots - \sberry \ldots -\teal \ldots + \sberry \ldots + \teal \ldots], \qquad \igm{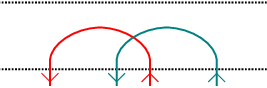}, \end{equation}
or the same with $\sberry$ and $\teal$ swapped.

To construct $\LEL$, work from right to left along $L_{\bullet}$, building a partial diagram one input $\pm s$ at a time. The starting point is the empty diagram. A partially constructed diagram will look like this:
\begin{equation} \label{LELpartial2} \igm{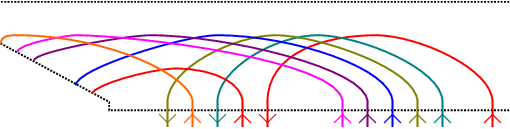} \end{equation}
At each stage there is a collection of unfinished strands which we call \emph{dangling}, that meet the diagonal dashed line. These strands are all aiming down, and waiting in line to eventually hit the bottom. Crucially, they already appear in end-order; no further crossings of dangling strands are required.

Now we consider the next input from $L_{\bullet}$. Suppose it is upward-oriented, part of strand $X$ (the yellow strand below). Based on where $X$ will hit bottom again, it has a place it should be in line amongst the dangling strands. It goes upwards, using left-crossings to cross as many dangling strands as it needs to (possibly none) until it reaches its proper place in end-order. Then $X$ does a counterclockwise cap, and is now facing downwards, becoming a dangling strand.
\begin{equation} \label{LELpartial} \igm{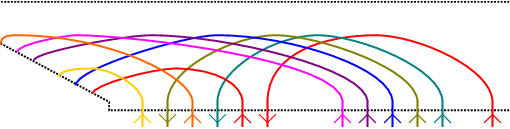} \end{equation}
For this to work, the dangling strands which cross $X$ must appear to the right of those which do not, so that $X$ may cross them without other dangling strands getting in the way. This is guaranteed by the fact that the dangling strands appear in end-order.

Now suppose the next input is downward-oriented (as the red strand below). It must be one of the dangling strands coming down to hit the bottom. Indeed, it must be the first dangling strand, since they are waiting in end-order. So we merely connect this strand to the bottom boundary and continue.
\begin{equation} \label{LELpartial3} \igm{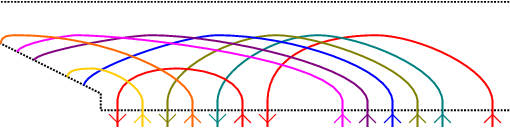} \end{equation}

Clearly, the property that the dangling strands appear in end-order is preserved by either step of the algorithm. Thus this algorithm makes sense, producing a well-defined diagram, which is sinister by construction. \end{proof}

\begin{rem} \label{rmk:unicity} Indeed, there is a unique sinister cap diagram with bottom boundary $L_{\bullet}$ up to isotopy. We do not need this fact, so we content ourselves with a sketch of the proof, since making such proofs precise (i.e. constructing isotopies) is tedious.

The \emph{timeline} of a strand $X$ is the sequence of crossings and its (single) cap, in the order they appear as one travels $X$ from start to end. For example, in \eqref{LEL} the timeline of the yellow strand is ``cross red2, cross blue1, cap, cross red3, cross blue2,''  where red2 represents the second red component from the right.  In the algorithm above, the timeline of every strand $X$ is as follows: first cross all strands $Y$ which appear before $X$ in start-order, then cap, then cross all strands $Z$ which appear after $X$ in start-order (of course, only crossing those strands which $X$ is required to cross). Moreover, the strands $Y$ appear in end-order, while the strands $Z$ appear in start-order. In this way, the timeline of $X$ is determined by $L_{\bullet}$.

We claim that for any sinister cap diagram for $L_{\bullet}$ and any strand $X$, the timeline for $X$ will be determined from $L_{\bullet}$ in the same way. Returning to the diagram in \eqref{nicecrossing}, if $X$ and $Y$ cross, and $Y$ comes before $X$ in start-order (in \eqref{nicecrossing}, $X$ is red and $Y$ is teal) then the crossing occurs before the cap in the timeline for $X$. Similarly, if $X$ and $Z$ cross and $X$ precedes $Z$ in start-order, the crossing comes after the cap in the timeline for $X$. So the timeline for $X$ has the schematic order $(\{Y_i\},\textrm{cap},\{Z_j\})$ as above. We need to argue that the order on $\{Y_i\}$ and $\{Z_j\}$ is as expected.

A strand switches from pointing upward to pointing downward exactly once, at the cap. If $Y_1, Y_2 < X$ in start-order and $Y_1 < Y_2$ in end-order, we need to show that $Y_1$ crosses $X$ before $Y_2$ does (in the timeline for $X$). Suppose to the contrary that $Y_2$ crosses first. Then $Y_1$ and $Y_2$ must cross each other after they have both crossed $X$. However, both $Y_1$ and $Y_2$ have already had their caps (by applying the previous paragraph to $Y_i$), so both are pointed down and remain that way until they hit the bottom. Therefore, they are not permitted to cross with a left-facing crossing. The argument for the order on $\{Z_j\}$ is similar.

Having done the combinatorial work, we leave to the reader the topological proof that a sinister cap diagram is determined up to isotopy by its timelines.
\end{rem}

\begin{rem} There are many other diagrams, allowing various cups, caps, and crossings, which would give the same morphism as $\ELL$ in $\Frob$. For example, one could apply oriented Reidemeister moves \eqref{R2easy} and (the first relation in) \eqref{eq:R3}. We chose to be rigid in order to give a clearer definition, and to avoid needing to prove that all such diagrams are equal (another topological annoyance). However, the reader should think that any ``reasonable" way to connect up the boundary points obeying \eqref{twooptions} should give the same morphism $\ELL$. \end{rem}



\begin{lem} The degree of the sinister cap diagram constructed in Lemma \ref{sinisterconstruction}, with bottom boundary $L_{\bullet}$, satisfies
\begin{equation} \label{LELdegree} \ell(L_{\bullet}) + 2 \deg(\LEL) = 0. \end{equation}
\end{lem}

\begin{proof} We consider the contribution to degree from each step of the algorithm. When a dangling strand connects to the bottom, the degree is unchanged. Suppose a new upward-oriented strand $X$ colored $t$ is added, with $K$ labeling the region to its left, and $Kt$ the region to its right. We claim that the degree is shifted by $\ell(K) - \ell(Kt)$. This should be clear from the following picture, where pluses and minuses indicate the contribution to the degree from the cap and the sideways crossings. Only two terms do not cancel.
\begin{equation} {
\labellist
\small\hair 2pt
 \pinlabel {$-$} [ ] at 26 47
 \pinlabel {$+$} [ ] at 26 41
 \pinlabel {$+$} [ ] at 42 44
 \pinlabel {$+$} [ ] at 42 34
 \pinlabel {$-$} [ ] at 36 38
 \pinlabel {$-$} [ ] at 48 38
 \pinlabel {$+$} [ ] at 53 32
 \pinlabel {$+$} [ ] at 53 22
 \pinlabel {$-$} [ ] at 46 27
 \pinlabel {$-$} [ ] at 57 27
\endlabellist
\centering
\igbb{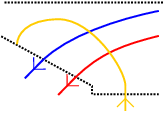}
} \end{equation}
Overall, the degree is determined by $L_{\bullet}$. If we write $L_{\bullet}$ as a multistep expression, in the form
\[ L_{\bullet} = [[K_1 \supset  I_1 \subset K_2 \supset I_2 \subset \ldots \supset I_{m-1} \subset K_m]],\]
then 
\begin{equation} \deg(\LEL) = \ell(I_1) - \ell(K_2) + \ell(I_2) - \ldots + \ell(I_{m-1}) - \ell(K_m). \end{equation}
Meanwhile, by \eqref{explength} we have
\begin{equation}
\ell(L_\bullet):= \ell(K_1) -2\ell(I_1)+2\ell(K_2)-\cdots - 2 \ell(I_{m-1}) + \ell(K_m), \end{equation}
from which it is easy to conclude
\begin{equation} \ell(L_{\bullet}) + 2 \deg(\LEL) = \ell(K_1) - \ell(K_m). \end{equation}
But $K_1 = K_m$ since the top boundary of $\LEL$ is empty, so \eqref{LELdegree} follows.
\end{proof}

\subsection{Well-definedness of elementary light leaves, and degree calculations} \label{ssec:degreeofdiagrams}

\begin{proof}[Proof of Lemma \ref{lem:enumeratev}] By \cite[Lemma 5.21]{EKo}, only simple reflections in $Js$ will appear in any reduced expression for $p$ or $q$. The alternating nature of the appearances of $\pm t$ is just due to the fact that they separate subsets which contain $t$ from subsets which do not. The unicity of the enumeration of $Js \setminus I$ is clear, so long as each $v_i$ does appear in $I_{\bullet}$. But the latter is straightforward: one must add $v_i$ at some point to get from a set which doesn't contain it to a set which does. \end{proof}

\begin{proof}[Proof of Lemma \ref{lem:noclock}] Using the notation of the previous section, this lemma can be rephrased as: there is a sinister diagram $\ELL([p,q])$ with bottom boundary $X_p \ot [J+s]$ and top boundary $X_q$. Equivalently, we can consider the sinister cap diagram $\LEL([p,q])$ obtained by twisting all strands in the target down to the left. This was constructed in Lemma \ref{sinisterconstruction}. It remains to confirm that the degree of $\ELL([p,q])$ agrees with $\defect([p,q])$.

The bottom boundary of $\LEL([p,q])$ is
\begin{equation} L_{\bullet} = [[K_1 = Js \supset I = I_1 \subset K_2 \supset I_2 \subset \ldots \supset I_{m-1} \subset J \subset Js = K_m]] \end{equation}
(possibly $I_1=K_2$ and $I_{m-1} = J$) where
\begin{equation} X_p = [[I_1 \subset K_2 \supset \ldots \subset K_{m-1} \supset I_{m-1} \subset J]] \end{equation}
is a reduced expression for $p$. As a composition of expressions we have
\begin{equation} \ell(L_{\bullet}) = \ell(X_p) + \ell([[Js \supset I]]) + \ell([[J \subset Js]]) = \ell(X_p) + 2 \ell(Js) - \ell(I) - \ell(J). \end{equation}
By \eqref{LELdegree} we have $\ell(L_{\bullet}) = -2 \deg(\LEL([p,q]))$. Meanwhile, $\ell(X_p) = \ell(p)$. Also,
\begin{equation}\label{secondeq} \ell(p) = 2 \ell(\ma{p}) - \ell(I) - \ell(J), \qquad \ell(\ma{p}) = \ell(\mi{p}) + \ell(I) + \ell(J) - \ell(\leftred(p)), \end{equation}
with the latter equality following from \eqref{eq:mapmip}. Elementary manipulation gives us
\begin{equation} \deg(\LEL([p,q]) = 
-\ell(\ma{p}) + \ell(I) + \ell(J) -  \ell(Js)  =
-\ell(\mi{p}) + \ell(\leftred(p)) -  \ell(Js). \end{equation}

The morphism $\ELL([p,q])$ is obtained from $\LEL([p,q])$ by adding clockwise cups on the left, a unit of adjunction associated to the extension $R^{Js} \subset R^I$. So
\begin{equation} \deg(\ELL([p,q])) = \deg(\LEL([p,q])) + \ell(Js) - \ell(I) = -\ell(\mi{p}) + \ell(\leftred(p)) -  \ell(I). \end{equation}
By definition the defect of $p \subset q$ is
\[ \defect([p,q])=\ell(\mi{q})-\ell(\mi{p})-\ell(\leftred(q))+\ell(\leftred(p)). \]
These two formulas match since $\mi{q}$ is the identity and $\leftred(q) = I$.
\end{proof}

\begin{exercise} In Lemma \ref{lem:rexmovesdegzero} we proved that the rex moves attached to switchback relations have degree zero. One can also prove this using a relationship between these rex moves and sinister diagrams, which we wish to point out. Our running example will be the right picture in \eqref{rexmoveE}, which we denote by $\phi$.
By twisting one output strand down to the left and the other down to the right, we obtain a sinister cap diagram.
\begin{equation} \phi = \igm{rexmoveE} \rightsquigarrow \LEL = \igm{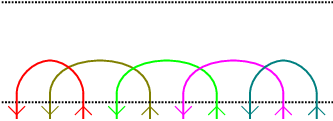} \end{equation}
It is a good exercise to use the degree computations for sinister cap diagrams to prove that $\phi$ has degree zero, using only that the source and target of $\phi$ are both reduced expressions for the same double coset.
\end{exercise}

\subsection{Degree comparisons}
In this section we compare the degree of singular light leaves with the singular Deodhar formula \Cref{cor:singDeodhar}
 and with the graded rank of morphism spaces between Bott--Samelson bimodules.

\begin{lem}\label{degLL}
Let $I_\bullet$ be an $(I,J)$ singlestep expression and $p$ an $(I,J)$ coset. We have the equality 
\begin{equation}
\sum_{
\substack{(t_{\bullet},y_\bullet)\subset I_{\bullet} \\
\term(t_{\bullet})=p
 }
}v^{\deg(\LL(t_\bullet, y_\bullet))}= \sum_{
\substack{t_{\bullet}\subset I_{\bullet} \\
\term(t_{\bullet})=p
 }
} \poly(t_{\bullet})v^{\defect(t_{\bullet})}\ =\gdim\left(\Hom_{\not<p}(\BS(I_\bullet),\BS(Y_p))\right).\end{equation}
\end{lem}
\begin{proof}
The second equality follows from \Cref{gdimlowerterms} and 
\Cref{cor:singDeodhar}.

Let $p \subset q$ be a coset pair. Recall from \Cref{degSSLL} that $\deg(\SSLL([p,q]))=\defect([p,q])$ and $\deg(\SSLL([q,p]))=\defect([q,p])$. Fix $t_\bullet \subset I_\bullet$.  By construction we have
\[ \deg(\LLP(t_\bullet, 1_\bullet))=\sum_k \deg(\SSLL[t_k,t_{k+1}]) \]
from which we deduce that $\deg(\LLP(t_\bullet, 1_\bullet)) = \defect(t_{\bullet})$.

Let $A_{t_\bullet}$ be the set of all allowable sprinkles $y_\bullet,$ i.e. sequences of elements in $W$  that make $(t_\bullet, y_\bullet)$  a sprinkled subordinate path. To prove the first equality it is enough to prove that 
\begin{equation}
    \sum_{
y_\bullet\in A_{t_\bullet} }
v^{\deg(\LL(t_\bullet, y_\bullet))}=
 \poly(t_{\bullet})v^{\defect(t_{\bullet})}.
\end{equation}
Dividing both sides by $v^{\deg(\LLP(t_\bullet, 1_\bullet))} = v^{\defect(t_{\bullet})}$, this is equivalent to
\begin{equation} \sum_{
y_\bullet\in A_{t_\bullet} }
v^{\deg(y_\bullet)}=
 \poly(t_{\bullet}).
\end{equation}
This equality was proven in \Cref{lem:matchingpolytosprinkles}, see \eqref{degpolynomialadapted}. 
\end{proof}

\begin{lem}\label{doubleleavesdegreecounting}
For $I_\bullet$ and $I'_\bullet$ singlestep expressions,  the sum of $v^{\deg(\DLL(T))}$, as $T$ ranges over all coterminal sprinkled quadruples for $(I_\bullet,I'_\bullet)$, is equal to the Laurent polynomial in \eqref{gradeddimhom}, and equal to the graded dimension of $\Hom(\BS(I_\bullet),\BS(I'_\bullet))$.
\end{lem}

\begin{proof}
We have to sum the degrees of  $\DLL(T)$ associated to all coterminal sprinkled quadruples $T=\left((t_\bullet,y_\bullet),(t'_\bullet,y'_\bullet),p, b\right).$ By construction (given that rex moves have degree zero), this is the sum, over all $(I,J)$-cosets $p$ of  
\[\left(\sum_{
\substack{(t_{\bullet},y_\bullet)\subset I_{\bullet} \\
\term(t_{\bullet})=p
 }
}
v^{\deg(\LL(t_\bullet, y_\bullet))}\right)
\left(\sum_{
\substack{(t'_{\bullet},y'_\bullet)\subset I'_{\bullet} \\
\term(t'_{\bullet})=p
 }
}v^{\deg(\mathcal{D}\LL(t'_\bullet, y'_\bullet))}\right)\sum_{b\in\mathbb{B}_p}v^{\deg b} . \]
By definition $\mathbb B_p$ is a graded basis of $R_p$ so we deduce that  $
\sum_{b\in\mathbb{B}_p}v^{\deg b} =\gdim(R_p).$ Also note that the duality functor preserves degrees, so $\duality$ can be removed from the second parenthetical term. The desired result now follows from \Cref{degLL} and \Cref{BShom}.
\end{proof}

\subsection{Evaluation of elementary light leaves} \label{subsec:evalelementary}

We begin the process of proving Theorem \ref{mainthm}. Throughout, we identify a light leaf morphism in $\Frob$ with its image under the evaluation functor $\evaluation$, a morphism between singular Bott--Samelson bimodules. Our first major goal is to prove the linear independence of light leaves. We want to find certain elements of Bott--Samelson bimodules which a given light leaf morphism sends to $\onetensor$. This enables a unitriangularity argument for linear independence in \S\ref{subsec:linearind}.

Throughout this section, $[p,q]$ is a Grassmannian $(I,J,s)$-coset pair.

\begin{lem}\label{dll1=1}
    Let $p \subset q$ be as above. Then $\DELL([p,q])(\onetensor)=\onetensor$.
\end{lem}

\begin{proof} By Lemma~\ref{lem:noclock}, the diagram for $\DELL([p,q])$ consists only of counterclockwise cups and rightward crossings, both of which preserve $\onetensor$. \end{proof}

\begin{lem}\label{ll1=1}
    Let $p \subset q$ be as above. Then
    $\ELL([q,p])(\onetensor)=\onetensor$.
\end{lem}

\begin{proof} 
This follows from Lemma~\ref{dll1=1} and the fact that clockwise caps also preserve $\onetensor$.
\end{proof}

Evaluation of light leaves $\ELL([p,q],f)$ requires more work. The source of this morphism is $\BS(I_{\bullet}) = \BS(X_p) \ot \BS([J,Js])$, and $\BS([J,Js])$ is $R^J$ viewed as an $(R^J, R^{Js})$-bimodule. For $b \in R^J$ we write 
\[ \onetensor_p \ot b \in \BS(X_p) \ot \BS([J,Js]) \]
for the corresponding element of $\BS(I_{\bullet})$. There is also an endomorphism of $\BS(I_{\bullet})$ coming from multiplication by $b$ in the appropriate tensor factor, which we denote
\[ \id_p \ot b \in \End(\BS(X_p) \ot \BS([J,Js])). \]
Meanwhile, the target of $\ELL([p,q])$ is $\BS(X_q)$, which is isomorphic to $R^I$ as an $(R^I, R^{Js})$-bimodule. We write $\onetensor_q$ for the element $1 \in R^I$, and $\mg_q$ for the subbimodule of positive degree elements of $R^I$.

Let $\leftred = \leftred(p)$ and $\rightred = \rightred(p)$. Since $\leftred \subset I$, there is a Frobenius extension $R^I \subset R^{\leftred}$. Moreover, one can find almost dual bases (see Definition \ref{defn:almostdual}) for the Frobenius extension $R^I \subset R^{\leftred}$ where one of the bases is in the image of $R^J$ under the Demazure operator $\pa_{\mi{p}}$; part of this claim is the fact that $\pa_{\mi{p}}$ sends $R^J$ to $R^{\leftred}$.  This is the main result of \cite{EKLP1}, which we recall below as Theorem~\ref{thm:dualbasisinimage}. The following lemma uses these dual bases.

\begin{lem}\label{circleswithpoly}
Let $\{\dbl\}$ and $\{\dbr\}$ be  almost dual bases for $R^{\leftred}$ over $R^I$ such that $\dbr$ is in the image of $R^J$ under $\pa_{\mi{p}}$.
Let $b_i \in R^J$ be such that $\pa_{\mi{p}}(b_i)= \dbr$. 
Then we have \begin{equation}
\ELL([p,q],\dbl)(\onetensor_p \otimes b_j) \equiv \delta_{ij}\onetensor_q \text{ modulo } \mg_q.
\end{equation}
\end{lem}

\begin{proof}
By Lemma \ref{dll1=1}, $\onetensor_p \ot 1$ is the image of $\onetensor_q$ under $\DELL([p,q])$. Therefore, $\onetensor_p \ot b_j$ is the image of $\onetensor_q$ under $(\id_p \ot b_j) \circ \DELL([p,q])$. We conclude that $\ELL([p,q],\dbl)(\onetensor_p \otimes b_j)$ is the image of $\onetensor_q$ under the diagram \eqref{ELLAALLEstart}.

\begin{equation}\label{ELLAALLEstart}
{
\labellist
\small\hair 2pt
 \pinlabel {$b_j$} [ ] at 305 64
 \pinlabel {$\dbl$} [ ] at 34 64
\endlabellist
\centering
\igm{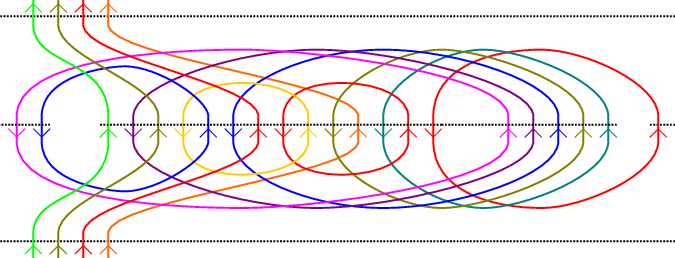}
}
\end{equation}

We resolve this diagram using the following algorithm.

\begin{algorithm} \label{circleresolutionalgorithm} This algorithm resolves any diagram of the form \eqref{ELLAALLEstart}, with the added generality that $\ELL([p,q])$ can be replaced by any sinister diagram (which is then composed with its dual, with polynomials inside).

The algorithm works from right to left, tracking a changing polynomial which we denote $h$, living on a horizontal symmetry line in the middle of the diagram. For this proof, the word ``segment'' will refer to a segment of a 1-manifold, which is a neighborhood of where it meets the symmetry line. This polynomial $h$ will always live in a region where every segment to its right is oriented upwards. We begin with the diagram above, and $h = b_j$. We proceed by examining the segment to the left of $h$.

\begin{enumerate}
\item If the segment to the left of $h$ is pointing upwards, then slide $h$ across it using \eqref{polyslide}. Doing this repeatedly to \eqref{ELLAALLEstart} will produce
\begin{equation}\label{ELLAALLE}
{
\labellist
\small\hair 2pt
 \pinlabel {$b_j$} [ ] at 225 64
 \pinlabel {$\dbl$} [ ] at 34 64
\endlabellist
\centering
\igm{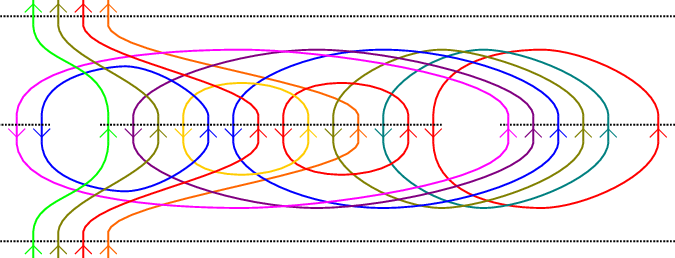}
}\; .
\end{equation}

\item If the segment $X$ to the left of $h$ is pointing downwards, then $X$ is part of a counterclockwise circle. Since every segment to the right of $h$ is oriented upwards, we can apply easy Reidemeister II relations \eqref{R2easy} to pull the circle tight around $h$. So \eqref{ELLAALLE} becomes
\begin{equation}\label{ELLAALLE1}
{
\labellist
\small\hair 2pt
 \pinlabel {$b_j$} [ ] at 222 64
 \pinlabel {$\dbl$} [ ] at 34 64
\endlabellist
\centering
\igm{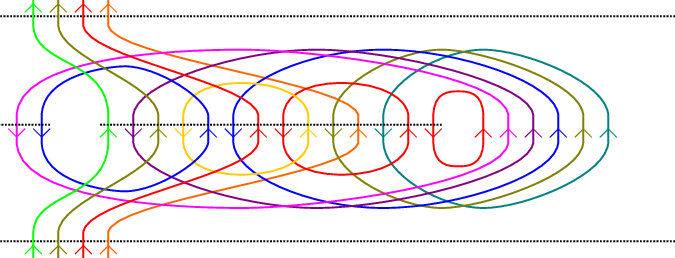}
}
\end{equation}
Now we can apply \eqref{circrelns}, removing the circle of $X$, and replacing $h$ with $\pa^K_{L}(h)$. Here, $K$ is the region containing $h$, while $L$ is the region that was to the left of $X$.
\begin{equation}\label{ELLAALLE2}
{
\labellist
\small\hair 2pt
 \pinlabel {$\pa_{{\color{red} \bullet}} b_j$} [ ] at 218 64
 \pinlabel {$\dbl$} [ ] at 34 64
\endlabellist
\centering
\igm{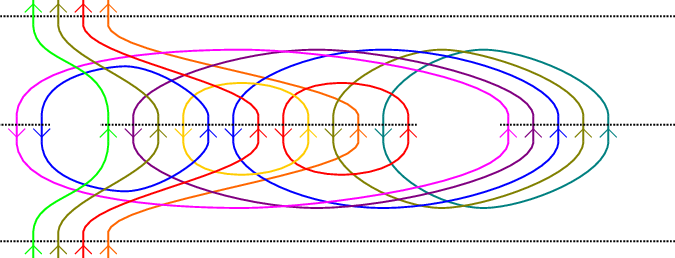}
}
\end{equation}

\item If $h$ enters a region with an existing polynomial (e.g. $f_i$) then it absorbs that polynomial (e.g. replace $h$ with $h \cdot f_i$). Continue by examining the segment to its left.
\end{enumerate}
At each step, the property that only upward-oriented segments appear to the right of $h$ is preserved. Iterate this algorithm until every counterclockwise circle has been removed, and what remains is the identity map of $X_q$ with a polynomial $h_{\last}$.

For example, the next two steps in the algorithm are
\begin{equation}\label{ELLAALLE3}
{
\labellist
\small\hair 2pt
 \pinlabel {$\pa_{{\color{red} \bullet}} b_j$} [ ] at 210 64
 \pinlabel {$\dbl$} [ ] at 34 64
\endlabellist
\centering
\igm{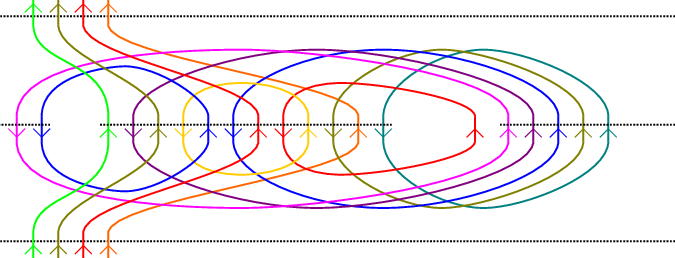}
}
\end{equation}
and
\begin{equation}\label{ELLAALLE4}
{
\labellist
\small\hair 2pt
 \pinlabel {$\pa_{{\color{teal} \bullet}} \pa_{{\color{red} \bullet}} b_j$} [ ] at 200 64
 \pinlabel {$\dbl$} [ ] at 34 64
\endlabellist
\centering
\igm{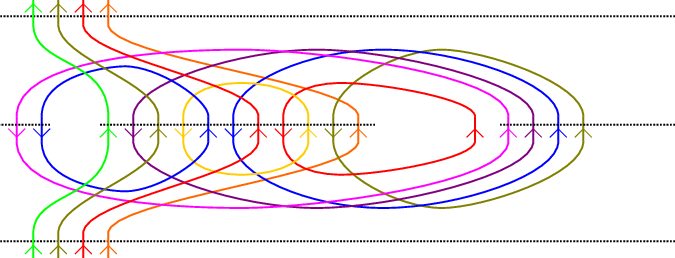}
}\; .
\end{equation}

The question which remains is: what is the final polynomial $h_{\last}$ which appears? The operations which have been applied to $h$ are:
\begin{enumerate} \item Inclusion $R^K \hookrightarrow R^L$, whenever $h$ was slid across an upward strand from a $K$-labeled region to an $L$-labeled region.
\item The Demazure operator $\pa^K_L$, whenever $h$ was squeezed in a circle to get across a downward strand from a $K$-labeled region to an $L$-labeled region.
\end{enumerate}
But these operations are precisely $\pa_{[L,K]}$, see \S\ref{subsec:demazure}. By composing them along any subsequence $I'_{\bullet}$ within the expression $I_{\bullet}$, we end up applying $\pa_{I'_{\bullet}}$. This concludes the algorithm.
\end{algorithm}

We return to the case of $\ELL([p,q])$ and the proof of Lemma \ref{circleswithpoly}. In the process of moving $b_j$ to the region containing $f_i$, we have applied $\pa_{p^{\core}} \circ \pa_{[[\rightred \subset J]]}$. This is just the map $\pa_{\mi{p}} \co R^J \to R^{\leftred}$, see \eqref{eq:papcore}. To get from there to the far left, we apply the Frobenius trace $\pa_{[[I \supset \leftred]]} = \pa^{\leftred}_I$. Thus we conclude that \eqref{ELLAALLEstart} is equal to
\begin{equation} \pa^{\leftred}_I(f_i \pa_{\mi{p}}(b_j)) = \pa^{\leftred}_I(f_i g_j) \equiv \delta_{ij} \onetensor_q \text{ modulo } \mg_q, \end{equation}
by the assumption that $f_i$ and $g_j$ were almost dual bases. \end{proof}

\begin{rem} Without needing to compute anything, one can use degree arguments to say that $\ELL([p,q],\dbl)(\onetensor_p \otimes b_j)$ is either
\begin{enumerate}
\item zero, because it has degree less than $\onetensor_q$,
\item a scalar multiple of $\onetensor_q$, because it has degree equal to $\onetensor_q$,
\item an element of $\mg_q$, because it has degree greater than $\onetensor_q$. \end{enumerate}
So Lemma \ref{circleswithpoly} only has interesting content when $b_i$ and $b_j$ have the same degree. Restricting the degree does not help to illuminate the proof, however. \end{rem}

\subsection{Circle evaluation lemma and consequences}\label{cel}

We interrupt the discussion of light leaves to prove the circle evaluation lemma, Lemma \ref{lem:circleresolution}, which follows from Algorithm \ref{circleresolutionalgorithm}.

\begin{proof}[Proof of Lemma \ref{lem:circleresolution}] We need to resolve the diagram
\begin{equation} {
\labellist
\tiny\hair 2pt
 \pinlabel {$\hat{u_0}$} [ ] at 6 16
 \pinlabel {\small $S$} [ ] at 62 8
 \pinlabel {$\hat{u_d}$} [ ] at 122 16
 \pinlabel {$\hat{u_1}$} [ ] at 32 23
 \pinlabel {$\hat{u_{d-1}}$} [ ] at 94 23
\pinlabel {\small $g$} [ ] at 112 31
\endlabellist
\centering
\igm{doublerexcircles}
}\; . \end{equation}
Our first step is to apply the hard R2 relation \eqref{R2hard} to the neighborhood of the region containing $g$. We obtain
\begin{equation} {
\labellist
\tiny\hair 2pt
 \pinlabel {$\hat{u_0}$} [ ] at 6 16
 \pinlabel {\small $S$} [ ] at 62 8
 \pinlabel {$\hat{u_d}$} [ ] at 120 16
 \pinlabel {$\hat{u_1}$} [ ] at 32 23
 \pinlabel {\small $*$} [ ] at 105 36
 \pinlabel {\small $\Delta^{\hat{u_d}}_{S, (2)}$} [ ] at 150 36
\endlabellist
\centering
\igm{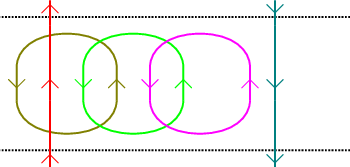}
}\; . \end{equation}
Here $*$ represents $\pa^{M}_{\hat{u}_{d-1}}(g \Delta^{\hat{u_d}}_{S, (1)})$, via the second line of \eqref{eq:leftcrossingnotsym}. Each term in this Sweedler sum is a tensor product of the morphism
\begin{equation} \label{resolvemeplease} {
\labellist
\tiny\hair 2pt
 \pinlabel {$\hat{u_0}$} [ ] at 6 16
 \pinlabel {\small $S$} [ ] at 62 8
 \pinlabel {$\hat{u_d}$} [ ] at 120 16
 \pinlabel {$\hat{u_1}$} [ ] at 32 23
 \pinlabel {\small $*$} [ ] at 105 36
\endlabellist
\centering
\igm{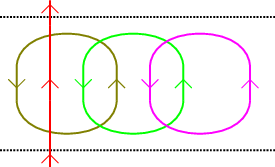}
} \end{equation}
with an identity map and a polynomial. We can resolve \eqref{resolvemeplease} using Algorithm \ref{circleresolutionalgorithm}. The result is precisely $\pa_{L'_{\bullet}}(g \Delta^{\hat{u_d}}_{S, (1)})$ on the left of a red identity strand. Overall, we obtain the right-hand side of \eqref{eq:circleresformula}.
\end{proof}

 Until now we have not used the results from \Cref{ch.rexmove}, which relied upon the lemma just proven. Now we can use Theorem \ref{thm:braidmovespreservecbot} to prove the following lemma.

\begin{lem} \label{rexELLisID} Continue the assumptions and notation of Example \ref{ex.ELLforward}. Let $\rex$ denote a rex move $X_m \to X_m'$. Then after applying the evaluation functor we have
\begin{equation} \evaluation(\rex \circ \ELL([n,m])) = \evaluation(\id_{X_m'}). \end{equation}
\end{lem}

\begin{proof} Both the morphism $\ELL([n,m]) \co X_m' \to X_m$ and a rex move $X_m \to X_m'$ have degree zero. By \Cref{ll1=1}, $\ELL([n,m])$ preserves the one-tensor, and by Theorem \ref{thm:braidmovespreservecbot}, $\rex$ preserves the one-tensor. Thus their composition is an endomorphism of $X_m' = \BS([[I \subset Js \supset J]])$ which preserves the one-tensor. This Bott-Samelson bimodule is a grading shift of $R^I \ot_{R^{Js}} R^J$, which is cyclic as a bimodule and generated by the one-tensor. Thus any bimodule endomorphism which preserves the one-tensor is the identity map. \end{proof}

\subsection{Dual sprinkles}

Let us now recall from \cite{EKLP1} the construction of a pair of almost dual bases of the extension $R^I\subset R^{\leftred(p)}$ that could be used for Lemma \ref{circleswithpoly}. Recall that Demazure surjectivity implies that for any finitary $I$ there exists an element $P_I\in R$ such that $\pa_I(P_I)=1$. 


\begin{thm}[{\cite[Theorem 4.14]{EKLP1}}] \label{thm:dualbasisinimage}
Let $I, J, L \subset S$ be such that $I \cup J \subset L$ and $L$ is finitary. Pick $P_I, P_L \in R$ such that $\pa_I(P_I) = 1$ and $\pa_L(P_L) = 1$. Let $p$ be an $(I,J)$-coset contained in $W_L$, and let $\leftred = \leftred(p)$. Let $y \in W_I$ be a minimal representative for its coset $W_{\leftred} y$, and let $x \in W_L$ be arbitrary. We say that $x$ is \emph{dual} to $y$ (with respect to $p$) if
\begin{equation} \label{eq:circdual} y^{-1} . \mi{p} . w_J . x = w_L. \end{equation}
Then $x$ is dual to $y$ if and only if $x = y^{\circ}$, where
\begin{equation}\label{eq:ycirc} y^{\circ} = w_J \mi{p}^{-1} y w_L. \qedhere\end{equation}
Moreover,
\begin{equation} \label{eq:leadstoalmostdual} \pa^{\leftred}_I(\pa_{\leftred} \pa_y(P_I) \cdot \pa_{\mi{p}} \pa_J \pa_x(P_L)) = \delta_{x,y^{\circ}} + \mg_{I}. \end{equation}
Here, $\mg_I$ represents the ideal of positive degree elements in $R^I$. In particular,
\begin{equation}\label{eq:almostdualbasis}
    \{\pa_{\leftred} \pa_y(P_I)\},\quad \{\pa_{\mi{p}} \pa_J \pa_{y^\circ}(P_L)\}
\end{equation} are almost dual bases for $R^{\leftred}$ over $R^I$, where $y$ ranges among minimal coset representatives for right cosets in $W_{\leftred} \backslash W_I$.
\end{thm}

\begin{rem} \label{rmk:worksforallz} This theorem is the first of many uses of the operation $y \mapsto y^{\circ}$. We take this opportunity to emphasize that not every $x \in W_L$ has the form $y^{\circ}$ for some $y$ (minimal in its coset $W_{\leftred}y$), but that \eqref{eq:leadstoalmostdual} applies to any element $x \in W_L$ regardless. \end{rem}

\begin{rem} \label{rmk:canchangeP} Note that the choices of $P_I$ and $P_L$ are independent. By varying $P_I$ but fixing $P_L$, one obtains a family of bases (the left side of \eqref{eq:almostdualbasis}) which are all almost dual to the same basis (the right side of \eqref{eq:almostdualbasis}). \end{rem}

We can summarize the results of \S\ref{subsec:evalelementary} in the light of Theorem \ref{thm:dualbasisinimage}.

\begin{prop}\label{prop.summarize}
Let $I_\bullet$ be a singlestep expression and $t_\bullet \subset I_\bullet$ be a subordinate path with terminus $p$. Use Notation \ref{notation:PKL}.

    \begin{enumerate}
    \item If $I_k\supset I_{k+1}$ then
\begin{equation} \SSLL([t_k,t_{k+1}])(\onetensor) = \onetensor. \end{equation}

 \item 
 If $I_k\subset I_{k+1}$, let $x \in W_{I_{k+1}}$ be arbitrary, and let $y$ be a minimal coset representative for a coset in $W_{\leftt_k}\backslash W_{\rightt_{k+1}}$. Then 
\begin{equation}  \label{LLdown} \SSLLP([t_k,t_{k+1}],\pa_{\leftt_k}\pa_y(P_{\rightt_{k+1}}))(\onetensor_{t_k} \otimes \pa_{I_{k}}\pa_x(P_{I_{k+1}}))=\delta_{x,y^{\circ}}
    \onetensor+\mg. \end{equation}
Here, $\mg$ represents the elements of $\BS(X_{t_{k+1}})$ of strictly higher degree than $\onetensor$, and $y^\circ = w_{I_k}\mi{t_k}\inv\mi{t_{k+1}}yw_{I_{k+1}}$.
 
\end{enumerate}

\end{prop}

\begin{proof}
For the first claim, let $m \subset n$ be the Grassmannian coset pair associated to $t_{k+1} \subset t_k$. Lemma~\ref{ll1=1} implies that $\ELL([n,m])$ preserves $\onetensor$. Since $\SSLL([t_k,t_{k+1}]) = \id_{Z_{\bullet}} \ot \ELL([n,m])$ for some $Z_{\bullet}$, it also preserves $\onetensor$. 

Similarly, one can analyze $\ELL([m,n])$ using \Cref{circleswithpoly} and \Cref{thm:dualbasisinimage}. Note from Theorem \ref{thmB} that $\mi{m} = \mi{t_{k+1}}^{-1}\mi{t_k}$, which explains the relationship between the formula for $y^{\circ}$ above and the formula in \eqref{eq:ycirc}. We obtain the desired result by tensoring with $\id_{Z_{\bullet}}$ on the left. \end{proof}

\Cref{prop.summarize} motivates the definition below. Remember that an ascending index $j$ relative to an expression $I_{\bullet}$ is an index for which $I_j \subset I_{j+1}$.

\begin{defn}
Let $(t_\bullet,y_\bullet)\subset I_\bullet$ be a sprinkled subordinate path. For any ascending index $j$, we define 
$y_j^\circ$ to be the dual of $y_j$ with respect to $m_j$, where $[m_j,n_j]$ is the Grassmannian pair associated to $[t_j,t_{j+1}]$, that is, we let
\begin{equation} y_j^{\circ} := w_{I_j}\mi{t_j}\inv\mi{t_{j+1}}y_jw_{I_{j+1}}.\label{betadef}\end{equation}
We call the sequence $y^\circ_{\bullet}$ the \emph{dual sequence} of $(t_{\bullet},y_{\bullet})$.
\end{defn}

We emphasize that both $y_{\bullet}$ and $y^\circ_{\bullet}$ are only defined for ascending indices.

Let $(t_{\bullet}, y_{\bullet})$ be a sprinkled subordinate path. For an ascending index $j$, the corresponding single step light leaf is the map $\SSLLP([t_j,t_{j+1}],f_j)$ where $f_j = \pa_{\leftt_j} \pa_{y_j}(P_{\rightt_{j+1}})$. 
As seen in \eqref{LLdown}, $y_j^{\circ}$ helps to define $b_k = \pa_{I_k} \pa_{y_j^{\circ}}(P_{I_{k+1}})$, a polynomial which is suitable for plugging into $\SSLL([t_j,t_{j+1}])$. We think of these polynomials $b_k$ as ``dual sprinkles.''

\begin{notation} \label{notation:b} Let $(t_{\bullet}, y_{\bullet})$ be a sprinkled subordinate path. For each ascending index $j$, let $P_{I_{j+1}}$ be as in \Cref{thm:dualbasisinimage}. For every $k$ let
\begin{equation} b(y^{\circ}_k):=\begin{cases} 1 &\text{if }I_k\supset I_{k+1}\\
\partial_{I_k}\partial_{y_k^{\circ}}(P_{I_{k+1}}) & \text{if }I_k\subset I_{k+1},\end{cases}\end{equation}
and let
\begin{equation}\label{bbetadef}
    b(t_\bullet,y_\bullet):=1\otimes b(y^{\circ}_1) \otimes b(y^{\circ}_2) \otimes \ldots \otimes b(y^{\circ}_d)\in \BS(I_\bullet).
\end{equation}
\end{notation}

We are gearing up for \Cref{LLtriang} which states that $\LL(t_{\bullet}, y_{\bullet})$ applied to $b(u_{\bullet}, z_{\bullet})$ is equal to $\onetensor$ when $(t_{\bullet}, y_{\bullet}) = (u_{\bullet}, z_{\bullet})$, and is equal to zero when one is bigger than the other for some particular total order. This crucial result will be used to prove linear independence of light leaves, and most of the remaining theorems will fall to its hand. First, we need to develop the theory of dual sequences.

\subsection{Dual sequences} \label{ssec:dualsequences}

At the moment it is not obvious that the dual sequence $y_{\bullet}^{\circ}$ necessarily contains very much information about $(t_{\bullet}, y_{\bullet})$. For example, one might imagine that nonequal sprinkled subordinate paths have the same dual sequence. We prove this is not the case in Lemma \ref{precistotalorder}. First we initiate a running example illustrating that dual sequences contain a great deal of information.

\begin{ex}\label{A2circex} Let $W$ have type $A_2$, with simple reflections $s$ and $t$. Let $I_{\bullet} = [\emptyset, s, st, s, \emptyset, s, st]$. When $I_0 = \emptyset$, recall from Remark \ref{rmk:whenI0empty} that $y_j=1$ for all valid $j$, so a sprinkled subordinate path is determined just by $t_{\bullet}$.
We claim that there are six (sprinkled) subordinate paths $t_{\bullet}$, one for each element of $W$, and they are uniquely determined by their dual sequences.

Any subordinate path $t_{\bullet}$ starts with $t_0 = \{\id\} \subset t_1 = \{\id,s\} \subset t_2 = W$. So any dual sequence has $y_0^{\circ} = s$ and $y_1^{\circ} = ts$. Next is $t_2 \supset t_3 \supset t_4$, where $t_4 = \{w\}$ for some $w \in W$. The final steps are $t_4 \subset t_5 = \{w,ws\} \subset t_6 = W$.

The element $y_4^{\circ}$ contains some information about $w$: we have $y_4^{\circ} = s$ if $w < ws$ and $y_4^{\circ} = \id$ if $ws > w$. The element $y_5^{\circ}$ contains the rest of the information, being equal to either $ts$ if $w \in \{\id,s\}$, or $t$ if $w \in \{t,ts\}$, or $\id$ if $w \in \{st,sts\}$. 
\end{ex}

We need the following lemma to use Proposition \ref{prop.summarize} with our dual sequence.

\begin{lem} \label{lem:itsinIj} For any sprinkled subordinate path $(t_\bullet,y_\bullet)\subset I_\bullet$, and any ascending index $j$ we have $y_j^{\circ} \in W_{I_{j+1}}$. \end{lem}

\begin{proof} Since $t_j \subset t_{j+1}$, with the former being an $(I_0, I_j)$-coset and the later an $(I_0, I_{j+1})$-coset, it is clear from Howlett's theorem (\cite[Lemma 2.12]{EKo}) that $\mi{t_j}$ is obtained from $\mi{t_{j+1}}$ by right multiplying by an element of $W_{I_{j+1}}$. Thus $\mi{t_j}^{-1} \mi{t_{j+1}} \in W_{I_{j+1}}$. By definition of a sprinkled subordinate path, $y_j \in W_{\rightt_{j+1}} \subset W_{I_{j+1}}$. Also $w_{I_j}, w_{I_{j+1}} \in W_{I_{j+1}}$, so the overall product lives in $W_{I_{j+1}}$. \end{proof}

We define now a dominance order on the set of sprinkled paths. We fix a total order $\triangleleft$ on the Coxeter group $W$ compatible with length, that is, such that $x\triangleleft y$ whenever $\ell(x)<\ell(y)$.

\begin{defn}\label{defn::pathorder}
Let $(t_\bullet,y_\bullet)$ and $(u_\bullet,z_\bullet)$ be coterminal sprinkled paths subordinate to $I_\bullet$. Let $y_\bullet^{\circ}$ and $z^{\circ}_\bullet$ be the respective dual sequences.

We write $(t_\bullet,y_\bullet)\prec (u_\bullet,z_\bullet)$ if there exists $j$ with $I_j \subset I_{j+1}$ such that $y^\circ_j\triangleleft z^\circ_j$, and for all $i< j$ such that $I_i\subset I_{i+1}$ we have
$y^\circ_i=z^\circ_i$.
\end{defn}

That is, we compare $(t_{\bullet},y_{\bullet})$ with other sprinkled paths using the lexicographic order on $y^{\circ}_{\bullet}$ induced by the order $\triangleleft$ on elements. By nature of the lexicographic order, $\prec$ is transitive.

\begin{ex} We continue Example \ref{A2circex}, identifying a subordinate sprinkled path with the element $w$ for which $t_4 = \{w\}$. Then $\prec$ becomes an unusual total order on $W$:
\[ sts \prec ts \prec s \prec st \prec t \prec \id. \]
\end{ex}

In the next lemma, we prove that $\prec$ is a total order on coterminal sprinkled paths. It is important to emphasize that this relation $\prec$ is only intended to compare sprinkled paths with the same terminus. For example, if $[I_0 \supset I_1 \supset \dots \supset I_d]$ then there are no sprinkles at all, i.e. $y_{\bullet}$ is the empty sequence. In this case, the sequence $t_{\bullet}$ is determined uniquely by its terminus, so there is no comparison to be made anyway.

\begin{lem} \label{precistotalorder}
Let $(t_\bullet,y_\bullet)$ and $(u_\bullet,z_\bullet)$ be coterminal sprinkled paths subordinate to $I_{\bullet}$ which are not equal. Then $y^\circ_\bullet\neq z^\circ_\bullet$. Consequently, $\prec$ is a total order.
\end{lem}

\begin{proof}
Suppose first that $t_{\bullet} = u_{\bullet}$. Then $y_j \ne z_j$ for some ascending index $j$. Directly from the definition \eqref{betadef} we have $y_j^{\circ} \ne z_j^{\circ}$, as desired.

Assume now $t_\bullet\neq u_\bullet$. We need to show that there exists an ascending index $j$ with $y^\circ_j\neq z^\circ_j$. Though the lexicographic order examines these sequences from left to right, we will find the difference between them by examining from right to left.

Let $j$ be the biggest index with $t_j\neq u_j$. Clearly $j < d$. If $I_j \supset I_{j+1}$ then $t_j$ is the unique coset containing $t_{j+1}$, and $u_j$ is the unique coset containing $u_{j+1}$. But $t_{j+1} = u_{j+1}$, so $t_j = u_j$, a contradiction. We conclude that $I_j\subset I_{j+1}$, and $t_j \ne u_j \subset t_{j+1}$.

Let $m \subset n$ and $m' \subset n$  be the Grassmannian pairs corresponding to $t_j \subset t_{j+1}$ and $u_j \subset t_{j+1}$ respectively. Note that $n$ is the $(\rightt_{j+1}, I_{j+1})$ coset containing the identity, whereas $m$ and $m'$ are determined by $t_j$ and $u_j$, and are nonequal.
Recall from Theorem \ref{thmB} that 
\begin{equation} \mi{m}=\mi{t_{j+1}}^{-1}\mi{t_j} \quad \text{ and }\quad \mi{m'}=\mi{t_{j+1}}^{-1}\mi{u_j}. \end{equation}
Plugging this into \eqref{betadef} we get
\begin{equation} \label{eq:cancelfromme} y^\circ_j=w_{I_j}\mi{m}^{-1}y_j w_{I_{j+1}}\quad \text{ and }\quad z^\circ_j=w_{I_j}\mi{m'}^{-1}z_jw_{I_{j+1}}.\end{equation}

If it were the case that $y^\circ_j = z^\circ_j$, then taking the inverse of both sides of \eqref{eq:cancelfromme} and setting them equal yields
\begin{equation} y_j^{-1} \mi{m} = z_j^{-1} \mi{m'}. \end{equation}
But $m$ and $m'$ are $(\rightt_{j+1},I_j)$-cosets, and  $y_j, z_j \in W_{\rightt_{j+1}}$, so $m$ and $m'$ are equal. This is a contradiction.

Given that $\prec$ is transitive, it will be a total order if any two non-equal elements are comparable (and inequivalent). This is what we have just shown.
\end{proof}

\subsection{Evaluation of light leaves}


Recall the meaning of $b(t_{\bullet}, y_{\bullet})$ from \Cref{notation:b}.

\begin{prop}\label{LLtriang}
Let $(t_{\bullet},y_{\bullet})$ and $(u_{\bullet},z_{\bullet})$ be coterminal sprinkled paths subordinate to $I_{\bullet}$. We have
 \begin{equation} \label{eq:LLtriang} \LL(t_\bullet,y_\bullet)(b(u_\bullet,z_\bullet))=\begin{cases}\onetensor &\text{if }(t_\bullet,y_\bullet)=(u_\bullet,z_\bullet)\\
 0 &\text{if }(t_\bullet,y_\bullet) \prec(u_\bullet,z_\bullet)
 \end{cases} \end{equation}
\end{prop}

\begin{proof}
Light leaves are built from single step light leaves and rex moves. Thus the case $(t_\bullet,y_\bullet)=(u_\bullet,z_\bullet)$ follows by iterating \Cref{prop.summarize} and \Cref{thm:braidmovespreservecbot}.

In case $(t_\bullet,y_\bullet)\prec(u_\bullet,z_\bullet)$, let $j$ be the smallest index such that $y_j^{\circ} \neq z^{\circ}_j$. While $(t_{< j}, y_{< j})$ and $(u_{< j},z_{< j})$ may be different, we have $b(t_{< j},y_{< j}) = b(u_{< j},z_{< j})$, since the definition of $b$ only uses the dual sequence and the overarching expression $I_{\bullet}$. Thus, as in the previous paragraph, we have 
\[ \LL(t_{< j},y_{< j})(b(u_{< j},z_{< j})) = \LL(t_{< j},y_{< j})(b(t_{< j},y_{< j})) = \onetensor.\]

By the inductive construction of light leaves, the next piece of $\LL(t_{\bullet}, y_{\bullet})$ will be $\SSLLP([t_j,t_{j+1}],\pa_{\leftt_j}\pa_{y_j}(P_{\rightt_{j+1}}))$. If we apply this to $\onetensor \ot b(y^\circ_j)$, we get the one-tensor by \eqref{LLdown}. Instead, we apply it to $\onetensor \ot b(z^{\circ}_j)$. Let $f$ be the result, an element of $\BS(X_{t_{j+1}})$. Even though $z^{\circ}$ was built using a potentially different subordinate path, Lemma \ref{lem:itsinIj} implies that $z_j^{\circ} \in W_{I_{j+1}}$, so we can apply Proposition \ref{prop.summarize} part (2). By \eqref{LLdown}, since $z^{\circ}_j \ne y^{\circ}_j$, we have $f \in \mg$. 

However, by construction of $\prec$ we have $y_j^{\circ} \triangleleft {z_j}^{\circ}$,
so $\ell(y_j^{\circ}) \le \ell({z_j}^{\circ})$. Then $\deg b(z^\circ_j) \le \deg b(y^\circ_j)$, meaning that the degree of $f$ is at most the degree of the one-tensor. Any element of $\mg$ with such a degree is zero, as desired.
\end{proof}

   


Note that \Cref{LLtriang} says absolutely nothing about $\LL(t_{\bullet}, y_{\bullet})(b(u_{\bullet}, z_{\bullet}))$ when $(u_{\bullet}, z_{\bullet}) \prec (t_{\bullet}, y_{\bullet})$. These will be mysterious and unknown elements of the target bimodule. 

We have now developed a crucial technical tool, \Cref{LLtriang}. Remember that light leaves depend on a number of choices in their construction, as do the elements $b(u_{\bullet}, z_{\bullet})$ above. This is as good a place as any to summarize the various choices made and dependencies on these choices.

Given a sequence $I_{\bullet}$, the set of sprinkled paths $(t_{\bullet},y_{\bullet})$ is a combinatorial object independent of any choices. The dual sequence $y^{\circ}_{\bullet}$, and the total order $\prec$, are also purely combinatorial and choice-independent.

The elements $b(u_{\bullet}, z_{\bullet}) \in \BS(I_{\bullet})$ depend only on a choice of $P_{I_{k+1}}$ for each ascending index $k$. The light leaves $\LL(t_{\bullet}, y_{\bullet})$ depend on a host of choices. Some are combinatorial: choices of reduced expressions, choices of rex moves between reduced expressions. Others are algebraic: a choice of $P_{\rightt_{j+1}}$ for each ascending index $j$, which are used to define the adapted sequence of polynomials $f_{\bullet}$.

The crucial point is that the choice of $P_{I_{k+1}}$ when defining $b$ is independent of the choices made when defining $\LL$, see Remark \ref{rmk:canchangeP}. Thus \eqref{eq:LLtriang} will hold for any possible choices.

\begin{rem} Changing the choices above will certainly change the mysterious values of $\LL(t_{\bullet}, y_{\bullet})(b(u_{\bullet}, z_{\bullet}))$ when $(u_{\bullet}, z_{\bullet}) \prec (t_{\bullet}, y_{\bullet})$. \end{rem}

\begin{rem} \label{rmk:varyP} One can vary the choice of $P_{I_{k+1}}$ when defining $b(u_{\bullet}, z_{\bullet})$ as $(u_{\bullet}, z_{\bullet})$ varies, and can vary the choice of $P_{\rightt_{j+1}}$ as $(t_{\bullet}, y_{\bullet})$ and $j$ vary. Still \eqref{eq:LLtriang} will hold. \end{rem}

\subsection{Linear independence of light leaves}\label{subsec:linearind}

Consider an expression $I_{\bullet} = [I_0, \ldots, I_d]$ and a subordinate path $t_{\bullet}$ with terminus $p$, and a sprinkling $y_{\bullet}$ which is irrelevant for this discussion. The light leaf $\LL(t_{\bullet}, y_{\bullet})$ is a morphism from $I_{\bullet}$ to $X_{t_d}$, where $X_{t_d}$ is some reduced expression for $p$. Different such light leaves with the same terminus may be morphisms to different reduced expression for $p$, and it is unwise to compare morphisms which live in different Hom spaces.

Instead, recall in our construction of double leaves that we have fixed one particular reduced expression $Y_p$ of $p$, which double leaves factor through. We now force our light leaves to have $Y_p$ as a target, so we can compare and sum them.

\begin{notation} 
For each subordinate sprinkled path $(t_{\bullet}, y_{\bullet}) \subset I_{\bullet}$ we choose a rex move $\rex(t_\bullet,y_\bullet)$ such that its target is $Y_p$ and its source is the target of $\LL(t_{\bullet}, y_{\bullet})$. We consider the following set of morphisms
\begin{equation} \label{eq:defininglls}
    \lls :=\{\rex(t_\bullet,y_\bullet) \circ \LL(t_\bullet,y_\bullet)\ |\ (t_\bullet,y_\bullet)\subset I_\bullet, \; \mathbf{term}(t_\bullet)=p\}.
\end{equation}
\end{notation}

Below we will apply the functor $\evaluation$ to $\lls$, to get a family of morphisms between singular Bott--Samelson bimodules.

Recall that the reduced expression $Y_p$ factors through $\leftred(p)$, so there is an action of $R^{\leftred(p)}$ on $\BS(Y_p)$ by placing polynomials in the appropriate region.

\begin{lem} \label{lem:findtopterminF} Let $F \in \Hom(\BS(I_{\bullet}), \BS(Y_p))$ be a linear combination
\begin{equation} F:=\sum_i\lambda_i \LL(t^i_\bullet,y^i_\bullet) \end{equation}
for distinct elements $\LL(t^i_\bullet,y^i_\bullet) \in \lls$, with nonzero coefficients $\lambda_i \in R^{\leftred(p)}$. Let $k$ be such that $(t^k_\bullet,y^k_\bullet)$ is maximal with respect to the order $\prec$. Then we have
\begin{equation} F(b(u_{\bullet}, z_{\bullet})) = \begin{cases} \lambda_k \cdot \onetensor &\text{if }(t^k_\bullet,y^k_\bullet)=(u_\bullet,z_\bullet)\\
 0 &\text{if }(t^k_\bullet,y^k_\bullet) \prec(u_\bullet,z_\bullet).
 \end{cases} \end{equation}
In particular, the maximal sprinkled path with nonzero coefficient in $F$ agrees with the maximal $(u_\bullet,z_\bullet)$ for which $F(b(u_\bullet,z_\bullet)) \ne 0$.
\end{lem}

\begin{proof} This follows immediately from \Cref{LLtriang}, since when $\sum \LL(t^i_\bullet,y^i_\bullet)$ is applied to $b(u_{\bullet}, z_{\bullet})$ in the regime $(t_{\bullet}^k, y_{\bullet}^k) \preceq (u_{\bullet}, z_{\bullet})$, at most one term survives. \end{proof}

Recall the definition of the ideal of lower terms $\Hom_{<p}$ and of the quotient $\Hom_{\not <p}$ from \Cref{sec:lowerterms}. 

\begin{thm} \label{llindep}
For an $(I,J)$-coset $p$ and an expression $I_\bullet$, the set $\evaluation(\lls)$  is linearly independent in $\Hom_{\not 
<p}(\BS(I_\bullet),\BS(Y_p))$ as a left module for $R^{\leftred(p)}$.
\end{thm}

\begin{proof}
First we note that Proposition \ref{LLtriang} will still hold with $\LL(t_{\bullet}, y_{\bullet})$ replaced by the composition $\rex(t_\bullet,y_\bullet)\circ \LL(t_\bullet,y_\bullet)\in \lls$. There are two easy arguments for this fact. The first is that rex moves preserve $\onetensor$, see Theorem \ref{thm:braidmovespreservecbot}. The second is that Proposition \ref{LLtriang} was proven for arbitrary constructions of light leaves, and one could well have chosen $X_{t_d} = Y_p$ in the first place, and let $\rex(t_\bullet,y_\bullet)$ be the identity map. Indeed, for the remainder of this proof, we shorten $\rex(t_\bullet,y_\bullet) \circ \LL(t_\bullet,y_\bullet)$ to $\LL(t_\bullet,y_\bullet)$ for brevity.

Let $F \in \Hom(\BS(I_{\bullet}), \BS(Y_p))$ be a nonzero linear combination
\[ F:=\sum_i\lambda_i \LL(t^i_\bullet,y^i_\bullet) \]
as in \Cref{lem:findtopterminF}. We need to show that 
$\evaluation(F) \not\in \Hom_{< p}(\BS(I_\bullet),\BS(Y_p))$.

Let $k$ be such that $(t^k_\bullet,y^k_\bullet)$ is maximal with respect to the order $\prec$. Then by \Cref{lem:findtopterminF} we have $F(b(t^k_\bullet,y^k_\bullet))=\lambda_k\onetensor$, where $\lambda_k$ is nonzero. Now \Cref{prop:lowertermsR} immediately implies that $\evaluation(F) \not\in\Hom_{< p}(\BS(I_\bullet),\BS(Y_p))$.
\end{proof}

\begin{prop} \label{prop:llsisbasismodlower}
The set $\evaluation(\lls)$ is a basis of
$\Hom_{\not< p}(\BS(I_\bullet),\BS(Y_p))$ as a module over $R^{\leftred(p)}$.
\end{prop}

\begin{proof}
We know from \Cref{llindep} that the set $\evaluation(\lls)$ is linearly independent in $\Hom_{\not< p}(\BS(I_\bullet),\BS(Y_p))$ over $R^{\leftred(p)}$.
By \Cref{degLL}, the graded ranks of $\Hom_{\not< p}(\BS(I_\bullet),\BS(Y_p))$ and of its submodule generated by $\evaluation(\lls)$ are equal. We conclude by the graded Nakayama Lemma that the two modules coincide.
\end{proof}


\begin{cor}\label{flippedbasis} The set of flipped light leaves 
\[ \duality\lls := \{ \duality \LL(t_{\bullet}, y_{\bullet})\circ \duality\rex(t_\bullet,y_\bullet)\} \]  evaluates to an $R^{\leftred(p)}$-basis for $\Hom_{\not< p}(\BS(Y_p), \BS(I_{\bullet}))$.
\end{cor}

\begin{proof} This follows immediately from Proposition \ref{prop:llsisbasismodlower} and the fact that $\duality$ preserves the ideal of lower terms $\Hom_{<p}$. \end{proof}

\begin{thm} \label{lluppertri}
 For an $(I,J)$-coset $p$ and an expression $I_\bullet$, let $\lls$ be a collection of light leaves, and let $\llsprime$ be another collection of light leaves, made using different choices (of rex, of rex move, of $P_{\rightt_{j+1}}$). Then the change of basis matrix between $\evaluation(\lls)$ and $\evaluation(\llsprime)$ within $\Hom_{\not< p}(\BS(I_\bullet),\BS(Y_p))$ is unitriangular (i.e. upper triangular with ones on the diagonal) with respect to the total order $\prec$. \end{thm}

\begin{proof} Fix some sprinkled path $(t_{\bullet}, y_{\bullet})$, and write $\LL'(t_{\bullet}, y_{\bullet})\in \llsprime$ as a linear combination with respect to the basis $\lls$ (we omit the evaluation functor for simplicity of notation). That is,
\begin{equation}\label{eq:expressoneinother} \LL'(t_{\bullet}, y_{\bullet}) =  \sum_i \lambda_i \LL(t^i_{\bullet}, z^i_{\bullet}), \end{equation}
where we assume that $\lambda_i \ne 0$ for all $i$ appearing in the sum. Let $k$ be such that $(t^k_\bullet,y^k_\bullet)$ is maximal with respect to the order $\prec$. 

By \Cref{lem:findtopterminF}, the maximal $(u_\bullet,z_\bullet)$ such that $\LL'(t_{\bullet},y_{\bullet})$ applied to $b(u_\bullet,z_\bullet)$ is nonzero is $(u_\bullet,z_\bullet) = (t^k_\bullet,y^k_\bullet)$. But by \Cref{LLtriang}, this maximal $(u_\bullet,z_\bullet)$ is also $(t_{\bullet}, y_{\bullet})$. Thus $(t_{\bullet}, y_{\bullet})$ is the maximal index with nonzero coefficient in the sum \eqref{eq:expressoneinother}, proving upper triangularity of the change of basis matrix. Moreover, evaluating on $b(t_{\bullet}, y_{\bullet})$ using \Cref{lem:findtopterminF} and \Cref{LLtriang} gives $\lambda_k = 1$, proving unitriangularity.
\end{proof}


\subsection{Double leaves} \label{ssec:doubleleafproof}

We combine light leaves and flipped light leaves to obtain a basis of the $\Hom$-spaces in $\SBSBim$.

\begin{lem} \label{idealsagree} Let $p$ be an $(I,J)$-coset, and $I_{\bullet}$ and $I'_{\bullet}$ be two expressions which begin in $I$ and end in $J$. Within $\Hom(\BS(I_{\bullet}), \BS(I'_{\bullet}))$, the span of all double leaves which factor through $q$ with $q \le p$ (i.e. associated to coterminal sprinkled quadruples with terminus $q$) is equal to $\Hom_{\le p}(\BS(I_{\bullet}), \BS(I'_{\bullet}))$. \end{lem}

\begin{proof} We prove by induction on $p$ (with respect to the Bruhat order). Our argument has no need for a base case. Suppose that all morphisms in $\Hom_{< p}$ are spanned by double leaves, and let $F \in \Hom_{\le p}(\BS(I_\bullet),\BS(I'_\bullet))$. So we can write $F = \sum G_i\circ H_i$, with $H_i\in \Hom(\BS(I_\bullet),B_{q_i}) \subset \Hom(\BS(I_\bullet),\BS(Y_{q_i}))$ and $G \in \Hom(B_{q_i},\BS(I'_\bullet)) \subset \Hom(\BS(Y_{q_i}),\BS(I'_\bullet))$, where $q_i\leq p$ for all $i$.

From \Cref{prop:llsisbasismodlower}, we can  write each $H_i$ as an $R^{\leftred(q_i)}$-linear combination of light leaves modulo terms lower than $q_i$. Similarly, from \Cref{flippedbasis} we can write $G_i$ as a linear combination of flipped light leaves modulo terms lower than $q_i$. So we can write $F$ as a linear combination of double leaves modulo terms lower than $p$. By induction, the terms smaller than $p$ are a linear combination of double leaves, hence we conclude that $F$ is a linear combination of double leaves. \end{proof}

\begin{rem}\label{rem34} The argument of the above lemma applies to any finite downward-closed subset of the Bruhat order on double cosets, not just $\{\le p\}$. \end{rem}

\begin{proof}[Proof of Theorem \ref{mainthm}]

Thanks to Soergel's hom formula and \Cref{doubleleavesdegreecounting} it is enough to show that double light leaves $\{\DLL(T,b)\}$ span the morphism space in question, as $T$ ranges over the set of coterminal sprinkled triples and $b$ ranges over $\BB_p$. By 
letting $p$ in \Cref{idealsagree} be sufficiently large (or taking any sufficiently large downward-closed subset as in \Cref{rem34}), we obtain the desired conclusion.
\end{proof}

\begin{defn} \label{def:orderontriples} Let $T = (p,(t_{\bullet}, y_{\bullet}), (t'_{\bullet}, y'_{\bullet}))$ and $U = (q, (u_{\bullet}, z_{\bullet}), (u'_{\bullet}, z'_{\bullet}))$ be two coterminal sprinkled triples associated to $I_{\bullet}$ and $I'_{\bullet}$. We write $U \preceq_3 T$ if either \begin{itemize}
\item $q < p$, or
\item $q = p$ and $(u_{\bullet}, z_{\bullet}) \preceq (t_{\bullet}, y_{\bullet})$ and $(u'_{\bullet}, z'_{\bullet}) \preceq (t'_{\bullet}, y'_{\bullet})$. \end{itemize}
\end{defn}

Unlike the total order $\prec$ on sprinkled subordinate paths, the relation $\prec_3$ on coterminal sprinkled triples is only a partial order.

\begin{proof}[Proof of \Cref{thm:COB}.]
\Cref{idealsagree} applies to both versions of the double leaves basis. As a consequence, the change of basis matrix is block upper triangular, with blocks indexed by $(I,J)$-cosets $p$.

Fix $p$. Let $T = (p,(t_{\bullet}, y_{\bullet}), (t'_{\bullet}, y'_{\bullet}))$. A double leaf $\DLL'(T,g)$ is equal to $G \circ g \circ H$, where $G = \LL'(t_{\bullet}, y_{\bullet})$ and $H = \duality\LL'(t'_{\bullet}, y'_{\bullet})$. By \Cref{lluppertri}, $G$ can be rewritten as $\LL(t_\bullet, y_{\bullet})$ plus a linear combination (over $R_p$) of $\LL$ which are strictly smaller with respect to $\prec$, all modulo $\Hom_{< p}$. Similarly with $H$. Thus, modulo $\Hom_{< p}$, $\DLL'(T,g)$ can be written as $\DLL(T,g)$ plus a sum of various  $\DLL(U,h_U)$ for $U \prec_3 T$. Since lower terms $\Hom_{< p}$ belong to lower blocks, this suffices to prove unitriangularity within each block.
\end{proof}

\subsection{A fibered cellular category} \label{ssec:fibered}

The notion of a fibered cellular category can be found in \cite[Definition 2.17]{ELauda}. We will recall this definition as we demonstrate how it applies to $\SBSBim(J,I)$, thus proving Theorem \ref{thm:fibered}.

A fibered cellular category is a $\Bbbk$-linear category for some commutative base ring $\Bbbk$ (in \cite{ELauda}, $\Bbbk$ is denoted by $R$). It comes equipped with a set of objects $\Lambda$ which also has a relation $\le$ giving it the structure of a poset. For $\SBSBim(J,I)$, the set $\Lambda$ is the collection of objects $\BS(Y_p)$, in bijection with the set of $(I,J)$-cosets $p$, and is equipped correspondingly with the Bruhat order. The relation $\le$ is required to have the descending chain condition, which the Bruhat order does.

For each object $X$ and each $\lambda \in \Lambda$ we need finite sets $M(\lambda, X)$ and $E(X,\lambda)$ which are in fixed bijection with each other. We also need a map $c$ which transforms elements of $M$ or $E$ into morphisms; notationally e.g. $c_U \in \Hom(X,\lambda)$ for $U \in E(X,\lambda)$. In our case, any object $X = \BS(I_{\bullet})$ is a Bott-Samelson bimodule for some $(I,J)$-expression, and $E(\BS(I_{\bullet}),\BS(Y_p))$ and $M(\BS(Y_p), \BS(I_{\bullet}))$ are in bijection with the set of sprinkled paths $\{(t_{\bullet}, y_{\bullet})\}$ subordinate to $I_{\bullet}$ with terminus $p$. In fact, it is easiest to set $E(\BS(I_{\bullet}),\BS(Y_p))$ equal to the set of light leaves $\lls$, which is already a subset of $\Hom(\BS(I_{\bullet}), \BS(Y_p))$, so that $c$ does nothing. Similarly, $M(\BS(Y_p), \BS(I_{\bullet}))$ can be defined as the set of morphisms obtained from $\lls$ via the duality functor.

For each $\lambda \in \Lambda$ we need a commutative free $\Bbbk$-algebra $A_{\lambda}$ inside $\End(\lambda)$, equipped with an involution $\iota$. In our case, $A_p := R_p$, viewed as a ring acting on $\BS(Y_p)$, and $\iota$ is trivial.

Then for $S \in M(\lambda, Y)$ and $a \in A_{\lambda}$ and $U \in E(X,\lambda)$ one defines the composition
\[ c_{S,a,U} := c_S \circ a \circ c_U.\]
In our case, for $a \in R_p$, and $T$ being the coterminal sprinkled triple associated to $S$ and $U$, this is exactly the definition of $\DLL(T,a)$.

This data is required to satisfy four conditions. The first states that $\{c_{S,a_i,U}\}$ forms a basis for each Hom space, as $a_i$ ranges over a $\Bbbk$-basis for $A_{\lambda}$. In our case, this is exactly the main theorem of our paper.

The second is that the bijection between $M(\lambda,X)$ and $E(X,\lambda)$, together with the involution $\iota$ on $A_{\lambda}$, extends to an anti-automorphism of the category. This is just the duality functor, flipping diagrams upside-down. Note that duality acts trivially on $R_p \in \End(\BS(Y_p))$. 

The fourth condition states that $E(\lambda, \lambda)$ consists of a single element whose associated map is the identity, and similarly for $M(\lambda, \lambda)$. In our case, since $Y_p$ is a reduced expression for $p$, it has a unique sprinkled path with terminus $p$, see  Example~\ref{ex.LLrex}. The associated light leaf could in theory be any rex move from $Y_p$ to itself, but by \Cref{rexELLisID} one valid choice is the identity map, and that is the choice we make.

Finally, the third condition is the cellular condition, which we choose not to recall. The purpose of object-adapted cellular categories is that the cellular condition is a consequence (see \cite[Lemma 2.8]{ELauda}) of an easier property: that $c(E(X,\lambda))$ spans all maps $X \to \lambda$ modulo the ideal spanned by basis elements associated to $\mu < \lambda$. In our case, this amounts to the fact that $\lls$ is a basis for $\Hom_{\not< p}(\BS(I_{\bullet}), \BS(Y_p))$, which was proven in Proposition~\ref{prop:llsisbasismodlower}.

\printbibliography
\end{document}